\documentclass{amsart}
\usepackage{amssymb, amsmath, latexsym, amscd, graphicx, tabularx, color}
\usepackage{rotating}
\usepackage[all]{xy}
\usepackage[dvipdfm]{hyperref}
\usepackage[all]{hypcap}
\usepackage[toc,page]{appendix}
\newtheorem{theorem}{Theorem}[section]
\newtheorem{lemma}[theorem]{Lemma}
\newtheorem{cor}[theorem]{Corollary}
\newtheorem{definition}[theorem]{Definition}
\newtheorem{proposition}[theorem]{Proposition}
\newtheorem{remark}[theorem]{Remark}
\newtheorem{example}[theorem]{Example}
\newtheorem{conclusions}[theorem]{Conclusions}

\def\pagenumber{1}

\begin{document}
\setcounter{page}{\pagenumber}
\newcommand{\T}{\mathbb{T}}
\newcommand{\R}{\mathbb{R}}
\newcommand{\Q}{\mathbb{Q}}
\newcommand{\N}{\mathbb{N}}
\newcommand{\Z}{\mathbb{Z}}
\newcommand{\tx}[1]{\quad\mbox{#1}\quad}
\parindent=0pt
\def\SRA{\hskip 2pt\hbox{$\joinrel\mathrel\circ\joinrel\to$}}
\def\tbox{\hskip 1pt\frame{\vbox{\vbox{\hbox{\boldmath$\scriptstyle\times$}}}}\hskip 2pt}
\def\circvert{\vbox{\hbox to 8.9pt{$\mid$\hskip -3.6pt $\circ$}}}
\def\IM{\hbox{\rm im}\hskip 2pt}
\def\COIM{\hbox{\rm coim}\hskip 2pt}
\def\COKER{\hbox{\rm coker}\hskip 2pt}
\def\TR{\hbox{\rm tr}\hskip 2pt}
\def\GRAD{\hbox{\rm grad}\hskip 2pt}
\def\RANK{\hbox{\rm rank}\hskip 2pt}
\def\MOD{\hbox{\rm mod}\hskip 2pt}
\def\DEN{\hbox{\rm den}\hskip 2pt}
\def\DEG{\hbox{\rm deg}\hskip 2pt}

\title[Strong reactions in quantum super PDE's. III]{STRONG REACTIONS IN QUANTUM SUPER PDE's. III:\\ EXOTIC QUANTUM SUPERGRAVITY}

\author{Agostino Pr\'astaro}
\maketitle
\vspace{-.5cm}
{\footnotesize
\begin{center}
Department SBAI - Mathematics, University of Rome ''La Sapienza'', Via A.Scarpa 16,
00161 Rome, Italy. \\
E-mail: {\tt agostino.prastaro@uniroma1.it}
\end{center}
}

\vskip 0.5cm
\centerline{\em This work in three parts is dedicated to Albert Einstein and Max Planck.}
\vskip 0.5cm

\begin{abstract}
Following the previous two parts, of a work devoted to encode strong reaction dynamics in the A. Pr\'astaro's algebraic topology of quantum super PDE's, nonlinear quantum propagators in the observed quantum super Yang-Mills PDE, $\widehat{(YM)}[i]$, are further characterized.  In particular, nonlinear quantum propagators with non-zero defect quantum electric-charge, are interpreted as {\em exotic-quantum supergravity} effects. As an application, the recently discovered bound-state called $Zc(3900)$, is obtained as a neutral quasi-particle, generated in a $Q$-quantum exotic supergravity process. {\em Quantum entanglement} is justified by means of the algebraic topologic structure of nonlinear quantum propagators. Quantum Cheshire cats are considered as examples of quantum entanglements. Higgs quantum particle $H^0$, $Z^0$ boson quantum particle and massive quantum graviton particle $G'$ are proved to be examples of massive bound states of quantum photons, entangled with quantum photons by means of nonlinear quantum propagators. Existence theorem for solutions of $\widehat{(YM)}[i]$ admitting negative local temperatures ({\em quantum thermodynamic-exotic solutions}) is obtained too and related to quantum entanglement. Such exotic solutions are used to encode Universe at the Planck-epoch. It is proved that the Universe's expansion at the Planck epoch is justified by the fact that it is encoded by a nonlinear quantum propagator having thermodynamic quantum exotic components in its boundary. This effect produces also an increasing of energy in the Universe at the Einstein epoch: {\em Planck-epoch-legacy} on the boundary of our Universe. This is the main source of the Universe's expansion and solves the problem of the non-apparent energy-matter ({\em dark-energy-matter}) in the actual Universe.  Breit-Wheeler-type processes have been proved in the framework of the Pr\'astaro's algebraic topology of quantum super Yang-Mills PDEs. Numerical comparisons of nonlinear quantum propagators with Weinberg-Salam electroweak theory in Standard Model are given.
\end{abstract}

\vskip 0.5cm


\vspace{.08in} \noindent \textbf{Keywords}: Integral (co)bordism groups in quantum (super) PDE's; Conservation laws;
$Q$-exotic nonlinear quantum propagators; Quantum gravity; Quantum exotic decay of $Zc(3900)$; Quantum entanglement; Quantum Cheshire cats; Neutral quantum bound states of photons; Massive quantum gravitons; Quantum negative temperature; Quantum cosmology; Dark matter-energy; Weinberg-Salam theory; Standard Model.

\section[Introduction]{\bf Introduction}

\rightline{\footnotesize {\em How exotic is the $Q$-exotic electron decay ?}}

\rightline{\footnotesize {\em How exotic is the charmonium decay $Zc(3900)\to \pi^-J/\psi$  ?}}

\rightline{\footnotesize {\em How paradoxical ``quantum Cheshire cats"  are ?}}

\rightline{\footnotesize {\em ``Do gravitons exist ? ... Is a graviton detectable ? "}}

\rightline{\footnotesize {\em How many gluons there are in proton ? }}

\rightline{\footnotesize {\em Why the Universe expands ?} }

\rightline{\footnotesize {\em What is the dark-energy-matter ?} }

\vskip 0.5cm

In two previous papers \cite{PRAS28, PRAS29}, we have characterized nonlinear quantum propagators of the observed quantum super Yang-Mills PDE, $\widehat{(YM)}[i]$, proving that the total quantum energy and the total quantum electric-charge of incident particles in quantum reactions do not necessitate to be the same of outgoing particles. These important phenomena, that are related to symmetry properties and gauge invariance of $\widehat{(YM)}[i]$, were non-well previously understood and wrongly interpreted.
Really the main origins of such phenomena are just symmetry properties of $\widehat{(YM)}[i]$, beside the structure of the nonlinear quantum propagators.  This fundamental aspect of quantum reactions in $\widehat{(YM)}[i]$, gives strong theoretical support to the guess about existence of quantum reactions where the ``electric-charge'' is not conserved. The conservation of the electric charge was quasi a dogma in particle physics. However, there are in the world many heretical experimental efforts to prove existence of decays like the following $e^-\to \gamma+\nu$, i.e. electron decay into a photon and neutrino. In this direction some first weak experimental evidences were recently obtained.\footnote{See \cite{GIAMMARCHI}. Some other exotic decays were also investigated, as for example the exotic neutron's decay: $n\to p+\nu+\bar\nu$. \cite{NORMANN-BACHALL-GOLDHABER}.}

Aim of this third paper is to further characterize a criterion to recognize under which constraints, nonlinear quantum propagators preserve quantum electric charges between incoming and outgoing particles. The main result of this third part is the existence of a sub-equation $\widetilde{\widehat{(YM)}[i]_{\bullet}}\subset \widehat{(YM)}[i]$, were live solutions strictly respecting the conservation of the quantum electric charge. Instead, solutions bording Cauchy data contained in the sub-equation $\widetilde{\widehat{(YM)}[i]_{\bullet}}\subset \widehat{(YM)}[i]$, that are globally outside $\widetilde{\widehat{(YM)}[i]_{\bullet}}$, can violate the conservation of the quantum electric charge. This effect is interpreted caused by the quantum supergravity. The action of the quantum supergravity is able to guarantee existence of such more general nonlinear quantum propagators in quantum super Yang-Mills PDEs. In fact quantum supergravity can deform nonlinear quantum propagators in order that they can produce such exotic solutions of $\widehat{(YM)}[i]$. In other words, quantum exotic strong reactions exist as a by-product of quantum supergravity that produces non-flat nonlinear quantum propagators. In the Standard Model quantum supergravity is completely forgotten. Without quantum supergravity, exotic quantum propagators cannot be realized ! The main results are the following. Theorem \ref{zero-defect-quantum-electric-charge} formulates a criterion to recognize under which constraints nonlinear quantum propagators have zero defect quantum electric-charge. Our main result is the identification of a sub-equation $\widetilde{\widehat{(YM)}[i]_{\bullet}}\subset \widehat{(YM)}[i]$, that is formally integrable and completely integrable, such that all quantum reactions encoded there, are characterized by non-$Q$-exotic nonlinear quantum propagators. Theorem \ref{non-conservation-of-quantum-electric-charge} states that $Q$-exotic nonlinear quantum propagators of $\widehat{(YM)}[i]$ are solutions with exotic-quantum supergravity, i.e., having non zero observed quantum curvature components $\widetilde{R}^{\beta\alpha}_K$. As an application we consider in Example \ref{quantum-exotic-zc-3900-decay} the $Q$-quantum exotic decay of $Zc(3900)$, considered a neutral charmonium. In fact supported by Theorem \ref{non-conservation-of-quantum-electric-charge}, we are able to prove that there exits a nonlinear quantum propagator encoding such decay where the conservation of charge is violated. The algebraic topologic structure of nonlinear quantum propagators allows us to justify the so-called phenomenon of {\em quantum entanglement}. See Theorem \ref{quantum-entanglement-in-observed-quantum-yang-mills-pdes}. This proves that the EPR paradox is completely solved in the framework of the Algebraic Topology of quantum PDE's, as formulated by A. Pr\'astaro. In fact, EPR paradox is related to a macroscopic model of physical world (Einstein's General Relativity (GR)). In order to reconcile this model with quantum mechanics, it is necessary to extend GR to a noncommutative geometry, as made by the Pr\'astaro's Algebraic Topology of quantum (super) PDE's. In fact the logic of micro-worlds is not commutative, hence it is natural that macroscopic mathematical models cannot justify quantum dynamics. In other words, the incompleteness of quantum mechanics (QM)(see \cite{EINSTEIN-PODOLSKY-ROSEN}) is the complementary incompleteness of the GR, since they talk at different levels: microscopic the first (QM), macroscopic the second (GR). These different points of view can be reconciled by introducing the noncommutative logic of the QM in the geometric point of view of the GR. But this must be made at the dynamic level ! It is not enough to formulate some noncommutative geometry to encode quantum worlds ! Therefore the necessity to formulate a geometric theory of quantum PDE's as has been realized by A. Pr\'astaro. Then many quantum paradoxical phenomena find natural justifications. For example the `` quantum Cheshire cat paradox" is well justified in the framework of the quantum entanglement related to nonlinear quantum propagators. Entangled photons are proved produce neutral bound states of photons. Of this type are also nonlinear quantum propagators able to produce Higgs quantum particle $H^0$, $Z^0$ boson quantum particle and massive graviton quantum particle $G'$. Quantum entanglements allow to obtain dynamic relations between quantum particles considered in the Standard Model and that Standard Model does not permit. (See Example \ref{exotic-quntum-dynamic-standard-model-extension}, Remark \ref{quantum-entanglement-and-quantum-bound-states}, Example \ref{how-many-gluons-there-are-in-a-proton},  Example \ref{gravitons-in-a-proton} and Example \ref{gravitons-proton-decays}.)  Theorem \ref{existence-quantum-thermodynamic-exotic-solutions} proves existence of solutions of $\widehat{(YM)}[i]$ admitting negative (absolute) temperature and Theorem \ref{quantum-entanglement-thermodynamic-exotic-solutions} relates such solutions to quantum entanglement. Exotic quantum solutions are used to encode quantum universe admitting electric-charge violation and non-equilibrium thermodynamic processes. These solutions solve the problem to justify the asymmetry between matter and antimatter existing in our Universe.

Another important subject considered in this third part is the proof that the geometric structure of quantum propagator encoding Universe at the Planck epoch is the cause of the Universe's expansion.\footnote{Georges Lemaître (1927) and  Edwin Hubble (1929) first proposed that the Universe is expanding. Lemaitre used Einstein's General Relativity equations and Hubble estimated value of the rate of expansion by observed redshifts. The most precise measurement of the rate of the Universe's expansion, has been obtained by NASA's Spitzer Space Telescope and published in October 2012. Very recently (March 17, 2014) some scientists of the Harvard-Smithsonian Center for Astrophysics, announced that observations with the telescope Bicep2 (Background Imaging of Cosmic Extragalactic Polarization), located at the South Pole, allowed to give an experimental proof of the Big Bang.  (See the following link \href{http://www.cfa.harvard.edu/news/2014-05}{http://www.cfa.harvard.edu/news/2014-05} reporting on the {\em first direct evidence of Cosmic Inflation}.)}
This expansion is not caused by a strange exoteric force, but it is the boundary-effect of the nonlinear quantum propagator encoding the Universe. In fact this propagator has a boundary with thermodynamic quantum exotic components. The presence of such exotic components produces an increasing of energy contents in the Universe, seen as transversal sections of such a nonlinear quantum propagator. To the increasing of energy corresponds an increasing in the expansion of the Universe. Such expansion has produced the passage of the Universe from the Higgs-universe, the first massive universe at the Planck epoch, to the actual macroscopic one. However, yet in such a macroscopic age the expansion of the Universe can be justified by using the same philosophy. This will be illustrated  by adopting the Einstein's General Relativity equations, but taking into account the effect of its quantum origin ({\em Planck-epoch-legacy}). With this respect one can state that the so-called dark-energy-matter, is nothing else than the increasing in energy produced by the thermodynamic exotic boundary encoding the Universe. Therefore it is a pure geometrodynamic bordism effect that produces an expansion of our Universe also at the Einstein epoch. Paradoxically this is a consequence of the energy conservation law that continues to work whether at the Planck epoch or at the Einstein age.
 The main result is Theorem \ref{expansiveness-criterion-universe-planck-epoch} characterizing the expansion of the Universe at the Planck epoch on the ground of its boundary containing thermodynamic quantum exotic components. Furthermore, Theorem \ref{universe-einstein-epoch-expansion} extends this result also to the Universe at the Einstein epoch. Finally Pr\'astaro's algebraic topology of quantum super Yang-Mills PDEs allows to justify Breit-Wheeler-type processes for the production of matter by means of high energy photons scatterings.(See Theorem \ref{matter-photo-productiion}.) In Appendices some numerical results are given that emphasize the importance of nonlinear quantum propagators also with respect to some great results of the Standard Model like the Weinberg-Salam electroweak theory.

 Let us also here emphasize that these new results have been possible thanks to the Algebraic Topology of Quantum Supermanifolds formulated by A. Pr\'astaro in a series of early published works. These allow to go beyond the usual point of view, adopted by physicists, that considers quantum supergravity as classical supergravity quantized by means of standard methods. In fact, that approach cannot well interpret nonlinear quantum phenomena that are dominant in high energy physics. Really the usual quantization of supergravity is made by means of so-called quantum propagators that are essentially obtained from the classical theory by a proceeding of linearization of such a theory. (See, e.g. \cite{DEW1, DEW2, NIEUWENHUIZEN} and \cite{PRA5, PRA6, PRA7, PRA8} for the corresponding interpretation in the geometry of super PDEs.) Instead the new Pr\'astaro's concept of {\em nonlinear quantum propagator} allows to consider more general non-commutative PDEs and a nonlinear process of integration in the category of quantum supermanifolds. In such a way one can encode in a natural way all very complex nonlinear quantum phenomena in high energy physics as, for example, quantum black-holes and quantum entanglements that it is impossible to completely encode in the the usual quantum supergravity. (See \cite{PRAS14, PRAS15, PRAS19, PRAS21-1, PRAS21, PRAS22, PRAS27, PRAS28, PRAS29}.)

\section{\bf The constraint zero defect quantum electric charge $\widehat{(YM)}[i]_{\blacksquare}$}\label{sec-constraint-zero-defect-quantum-electric-charge}

In this third part we assume as prerequisite the knowledge of the previous two parts \cite{PRAS28, PRAS29}.
 In the following theorem we identify a constraint $\widehat{(YM)}[i]_{\blacksquare}$ that must satisfy nonlinear quantum propagators in order
 should be with zero defect quantum electric-charge. Furthermore a sub-equation $\widetilde{\widehat{(YM)}[i]_{\bullet}}\subset \widehat{(YM)}[i]$ is identified where live nonlinear quantum propagators without mass-gap. Some interesting applications are also considered that give a new mathematical support and justification to recent exotic experimental results like the so-called $Zc(3900)$ exotic decay.\cite{ABLIKIM, LIU}

\begin{theorem}[Conditions for zero defect quantum electric charge]\label{zero-defect-quantum-electric-charge}
Let $\widehat{(YM)}[i]$ be the quantum super Yang-Mills PDE, observed by a quantum relativistic frame $i:N\to M$. Assuming that the fundamental quantum superalgebra $A$ of $\widehat{(YM)}$ has a Noetherian center, then there exists a formally integrable and completely integrable quantum super PDE, $\widetilde{\widehat{(YM)}[i]_{\bullet}}\subset \widehat{(YM)}[i]$, {\em(zero-defect-quantum-electric-charge PDE)}, where nonlinear quantum propagators $V$, without mass-gap, have $\mathfrak{Q}[V]=0$.

One has the following fiber bundle structure $\widetilde{\widehat{(YM)}[i]_{\bullet}}\to N$.

\end{theorem}

\begin{proof}
Let us recall that if $\psi\equiv(i:N\to M;\phi)$ is a quantum relativistic frame, observing the quantum super Yang-Mills PDE $\widehat{(YM)}\subset \hat J{\it D}^2(W)$, then it is identified a quantum super PDE $\widehat{(YM)}[i]\subset \hat J{\it D}^2(E[i])$, where $E[i]$ is the bundle over $N$ defined in the following commutative diagram.

\begin{equation}
\xymatrix{\mathfrak{g}\bigotimes_{\mathbb{K}}T^*N\cong E[i]\equiv
Hom_Z(TN;\mathfrak{g})\ar[d]&i^*W\equiv Hom_Z(TM;\mathfrak{g})|_{i(N)}\ar[l]^(.43){\mathbb{B}(i)}\ar[d]\\
N\ar[r]_{i}&i(N)\subset M\\}
\end{equation}
More precisely one has the following commutative diagram relating $\widehat{(YM)}[i]$ with $\widehat{(YM)}$.

\begin{equation}
\xymatrix{0&\widehat{(YM)}[i]\ar[l]\ar@{^{(}->}[d]&i^*\widehat{(YM)}\ar[l]_{\mathbb{B}(i)_*}\ar@{^{(}->}[d]\ar[r]^{i_*}&\widehat{(YM)}\ar@{^{(}->}[d]\\
0&J\hat{\it D}^2(E[i])\ar[l]\ar[d]&i^*J\hat{\it D}^2(W)
\ar[l]_{\mathbb{B}(i)_*}\ar[d]\ar[r]^{i_*}&J\hat{\it D}^2(W)\ar[d]\\
0&E[i]\ar[l]\ar[d]&i^*W
\ar[l]_{\mathbb{B}(i)}\ar[d]\ar[r]^{i_*}&W\ar[d]_{\hat\pi}\\
&N\ar @{=}[r]&N\ar[r]_{i}&M\\}
\end{equation}
where $\mathbb{B}(i)$ is a canonical epimorphism over $N$ and
$\mathbb{B}(i)_*$ is the induced epimorphism on jet spaces.
Furthermore, since $\widehat{(YM)}$ is formally quantum (super) integrable,
(resp. completely quantum (super) integrable), so is $\widehat{(YM)}[i]$.

The dynamic equation $\widehat{(YM)}$ and $\widehat{(YM)}[i]$ are resumed in Tab. \ref{local-expression-yang-mills-and-bianchi-identity}.

\begin{table}[h]
\caption{Local expression of
${\widehat{(YM)}\subset J\hat D^2(W)}$, $\widehat{(YM)}[i]$ and quantum Bianchi identities.}
\label{local-expression-yang-mills-and-bianchi-identity}
\begin{tabular}{|l|c|}
\hline
{\footnotesize\rm(Field equations)\hskip 1cm $E^A_{K}\equiv-(\partial_{B}.\hat R^{BA}_K)+[\widehat
C^H_{KR}\hat\mu^R_{C},\hat R^{[AC]}_H]_+=0$}&{\footnotesize\rm$\widehat{(YM)}$}\\
\hline \hline
\hfil{\footnotesize\rm $\hat R^K_{A_1A_2}=\left[(\partial X_{
A_1}.\hat\mu^K_{A_2})+\frac{1}{2}\widehat{C}{}^K_{IJ}[\hat\mu^I_{A_1},\hat\mu^J_{A_2}]_+\right]$}\hfil&{\footnotesize\rm(Fields)}\\
\hline
 {\footnotesize\rm(Bianchi identities)\hskip 0.25cm $B^K_{HA_1A_2}\equiv(\partial X_{H}.\hat
R^K_{A_1A_2})+\frac{1}{2} \widehat{C}{}^K_{IJ}[\bar\mu^I_{H},\hat R^J_{A_1A_2}]_+=0$}&{\footnotesize\rm$(B)$}\\
\hline
\end{tabular}
\scalebox{0.939}{$\begin{tabular}{|l|c|}
\hline
{\footnotesize\rm(Observed Field Equations)\hskip 0.5cm$(\partial_{\alpha}.\tilde R^{K\alpha\beta})+[\widehat
C^K_{IJ}\tilde\mu^I_{\alpha},\tilde R^{J\alpha\beta}]_+
=0$}&{\footnotesize $\widehat{(YM)}[i]$} \\
\hline
\hfil{\footnotesize\rm $\tilde
R^K_{\alpha_1\alpha_2}=(\partial\xi_{[\alpha_1}.\tilde\mu^K_{\alpha_2]})+
\frac{1}{2}\widehat{C}{}^K_{IJ}\tilde\mu^I_{[\alpha_2}\tilde\mu^J_{\alpha_1]}$}\hfil&{\footnotesize\rm(Observed Fields)}\\
\hline
{\footnotesize\rm(Observed Bianchi Identities)\hskip 2pt$(\partial
\xi_{[\gamma}.\tilde R^K_{\alpha_1\alpha_2]})+\frac{1}{2}
\widehat{C}{}^K_{IJ}\tilde\mu^I_{[\gamma}\tilde
R^J_{\alpha_1\alpha_2]}=0$}&{\footnotesize$(B)[i]$}\\
\hline
\multicolumn {2}{l}{\footnotesize\rm$\hat R^K_{A_1A_2}:\Omega_1\subset J\hat D(W)\to\mathop{\widehat{A}}\limits^2;\quad
B^K_{HA_1A_2}:\Omega_2\subset J\hat D^2(W)\to\mathop{\widehat{A}}\limits^3;\quad E^A_{K}:\Omega_2\subset
J\hat D^2(W)\to\mathop{\widehat{A}}\limits^{3}.$}\\ \end{tabular}
$}\end{table}
These equations are also formally integrable and completely integrable. (For details see \cite{PRAS11,PRAS14,PRAS21,PRAS22}.)
In the following we recall a theorem characterizing $Q$-exoticity of nonlinear quantum propagators. (See Theorem 4.10 in \cite{PRAS29}.)
\begin{theorem}[$Q$-exotic nonlinear quantum propagators of ${\widehat{(YM)}[i]}$ \cite{PRAS29}\label{non-conservation-of-quantum-electric-charge}]
For any observed nonlinear quantum propagator $V$ of $\widehat{(YM)}[i]$, such that $\partial V=N_0\sqcup P\sqcup N_1$, where $N_i$, $i=0,1$, are $3$-dimensional space-like admissible Cauchy data of $\widehat{(YM)}[i]$, and $P$ is a suitable time-like $3$-dimensional integral manifold with $\partial P=\partial N_0\sqcup\partial N_1$, the following equation holds.
$\hat Q[i|t_0]=\hat Q[i|t_1] \, {\rm mod}\hskip 3pt \mathfrak{Q}[V]\in A$,
 where $\hat Q[i|t_r]\in A$ is the quantum electric charge on $N_r$, $r=0,1$, and $\mathfrak{Q}[V]\in A$, is a term that in general is not zero and that we call {\em defect quantum electric-charge}.
We call {\em $Q$-exotic nonlinear quantum propagators}, nonlinear quantum propagators such that $\mathfrak{Q}[V]\not=0\in A$.\footnote{This agrees with the conservation of the observed quantum Hamiltonian. See Theorem 3.20 in \cite{PRAS28}. In fact, the observed quantum electromagnetic energy is a form of quantum energy, even if, in general, it does not coincide with the observed quantum Hamiltonian.}
\end{theorem}

Let us now consider that ${\widehat{(YM)}[i]}$ is invariant for gauge transformations. From this invariance we have the conservation law reported in (\ref{conservation-law-gauge-invariance}).
\begin{equation}\label{conservation-law-gauge-invariance}
\scalebox{0.8}{$\left\{
\begin{array}{ll}
\omega_q&=(\partial\widetilde{\mu}^{\alpha 0}_{K}.L)\, \Xi^K_\alpha\,\otimes d\xi^1\wedge  d\xi^2\wedge  d\xi^3\\
&+\sum_{1\le j\le 3}(-1)^{j}(\partial\widetilde{\mu}^{\alpha j}_{K}.L)\,\otimes \Xi^K_\alpha\, d\xi^0\wedge  d\xi^1\wedge\cdots\wedge \widetilde{d\xi^j}\wedge\cdots\wedge  d\xi^3\\
\end{array}
\right.$}
\end{equation}

where $\widetilde{d\xi^j}$ means absent. ($\Xi^K_\alpha$ are the gauge parameters.) The conservation law $\omega_q$ is of the same type of ones considered in \cite{PRAS28}, hence it is closed on the solutions of ${\widehat{(YM)}[i]}$. In particular if $V$ is a nonlinear quantum propagator such that $\partial V=N_0\sqcup P \sqcup N_1$, where $N_0$ is a initial Cauchy data representing a set of incident particles in the quantum reaction, and $N_1$ is a final Cauchy data representing a set of outgoing particles in the quantum reaction. Both $N_0$ and $N_1$ are considered space-like, and identifying some time $t_0$ and $t_1$, of the proper time of the quantum relativistic frame. Instead $P$ is considered a $3$-dimensional time-like manifold, such that $\partial P=\partial N_0\sqcup\partial N_1$. Then one has the following equation:
$0=<d\omega_q,V>=<\omega_q,\partial V>=<\omega_q,N_0>-<\omega_q,N_1>+<\omega_q,P>$.
In particular one has
\begin{equation}\label{conservation-law-gauge-invariance-b}
\scalebox{0.8}{$\left\{
\begin{array}{ll}
{Q[i|t_0]}&{\equiv <\omega_q,N_0>=<(\partial\widetilde{\mu}^{\alpha 0}_{K}.L)\Xi^K_\alpha\,\otimes  d\xi^1\wedge  d\xi^2\wedge  d\xi^3,N_0>\in A}\\
{Q[i|t_1]}&{\equiv <\omega_q,N_1>=<(\partial\widetilde{\mu}^{\alpha 0}_{K}.L)\Xi^K_\alpha\,\otimes  d\xi^1\wedge  d\xi^2\wedge  d\xi^3,N_1>\in A}\\
{\mathfrak{Q}[V]}&{\equiv <\omega_q,P>}
{=<\sum_{1\le j\le 3}(-1)^j(\partial\widetilde{\mu}^{\alpha j}_{K}.L)\Xi^K_\alpha\,\otimes  d\xi^0\wedge  d\xi^1\wedge\cdots\wedge \widetilde{d\xi^j}\wedge\cdots\wedge  d\xi^3,P>}\\
&+<(\partial\widetilde{\mu}^{\alpha 0}_{K}.L)\Xi^K_\alpha\, \xi^3_0\,\otimes  d\xi^0\wedge  d\xi^1\wedge d\xi^2,P>\in A.\\
\end{array}\right.$}
\end{equation}
By a direct calculation one can see that

\begin{equation}\label{conservation-law-gauge-invariance-c}
  \omega_q|_P=[(\partial\widetilde{\mu}^{\alpha 0}_{K}.L)\xi^3_0+(\partial\widetilde{\mu}^{\alpha 1}_{K}.L)\xi^3_1+(\partial\widetilde{\mu}^{\alpha 2}_{K}.L)\xi^3_2-(\partial\widetilde{\mu}^{\alpha 3}_{K}.L)]\Xi^K_\alpha \otimes d\xi^0\wedge d\xi^1\wedge d\xi^2.
\end{equation}
Therefore the condition that $V$ has zero defect quantum electric-charge is the following:

 \begin{equation}\label{conservation-law-gauge-invariance-cc}
  [(\partial\widetilde{\mu}^{\alpha 0}_{K}.L)\xi^3_0+(\partial\widetilde{\mu}^{\alpha 1}_{K}.L)\xi^3_1+(\partial\widetilde{\mu}^{\alpha 2}_{K}.L)\xi^3_2-(\partial\widetilde{\mu}^{\alpha 3}_{K}.L)]\Xi^K_\alpha =0.
\end{equation}

Furthermore one can see that $(\partial\widetilde{\mu}^{\alpha 0}_{K}.L)=\widetilde{R}^{0\alpha}_{K}$ and
$(\partial\widetilde{\mu}^{\alpha j}_{K}.L)=\widetilde{R}^{j\alpha}_K$. (For details see \cite{PRAS11,PRAS14}.)

Let us require also that equations (\ref{conservation-law-gauge-invariance-cc}) should be satisfied for any $\Xi^K_\alpha$, i.e., for any gauge transformation. Then we identify the constraint equation $\widehat{(YM)}[i]_\blacksquare $, given in (\ref{sub-equation-zero-defect-quantum-electric-charge-a}).

\begin{equation}\label{sub-equation-zero-defect-quantum-electric-charge-a}
\scalebox{0.8}{$\widehat{(YM)}[i]_\blacksquare:\hskip 2pt\left\{
\begin{array}{l}
\widetilde{R}^{\alpha 0}_{K}\xi^3_0+\widetilde{R}^{\alpha 1}_{K}\xi^3_1+\widetilde{R}^{\alpha 2}_{K}\xi^3_2-\widetilde{R}^{\alpha 3}_{K}=0\\
(\partial_{\alpha}.\tilde R^{K\alpha \beta})+[\widehat
C^K_{IJ}\tilde\mu^I_{\alpha},\tilde R^{J\alpha\beta}]_+=0\\
\end{array}
\right\}.$}
\end{equation}
One can interpret the first equation in (\ref{sub-equation-zero-defect-quantum-electric-charge-a}) a condition on the functions $\xi^3=\xi^3(\xi^1,\xi^2,\xi^0)$ characterizing the manifold $P$ in order that ${\mathfrak{Q}[V]}=0$. Equation (\ref{sub-equation-zero-defect-quantum-electric-charge-a}) is surely verified when $\widetilde{R}^{\beta\alpha}_{K}=0$, namely in the sub-equation

\begin{equation}\label{sub-equation-zero-defect-quantum-electric-charge}
\scalebox{0.8}{$\widehat{(YM)}[i]_\bullet\subset \widehat{(YM)}[i]:\hskip 2pt\left\{
\begin{array}{l}
\widetilde{R}^{\beta\alpha}_K=0\\
{(\partial_{\alpha}.\tilde R^{K\alpha \beta})+[\widehat
C^K_{IJ}\tilde\mu^I_{\alpha},\tilde R^{J\alpha\beta}]_+=0}\\
\end{array}
\right\}.$}
\end{equation}
One can prove that, under the condition that the quantum superalgebra $B$, underling these quantum supermanifolds, has Noetherian center, $Z$, $\widehat{(YM)}[i]_\bullet$ canonically identifies a formally integral and completely integrable quantum super PDE, $\widetilde{\widehat{(YM)}[i]_\bullet}\subset \widehat{(YM)}[i]$ too. The proof follows some geometric approaches previously formulated by A. Pr\'astaro for PDEs in the category of manifolds and quantum (super)manifolds, and that consists to use enough prolongations of original equations, in order to obtain formally integrable equations having the same solutions of $\widehat{(YM)}[i]_\bullet$.\footnote{Applications of these methods to the Navier-Stokes PDEs and MHD-PDEs were first obtained by A.Pr\'astaro to encode complex dynamics in regular equations.}  Therefore, for any point $u\in \widetilde{\widehat{(YM)}[i]_\bullet}$ passes a solution, that since $\widetilde{\widehat{(YM)}[i]_\bullet}$ is of class $Q_w^\omega$, it can be fixed of class $Q_w^\omega$ too. All nonlinear quantum propagators $V$ in $\widetilde{\widehat{(YM)}[i]_\bullet}$ have ${\mathfrak{Q}[V]}=0$. Let us go in some detail of this proof, since it is necessary to use some technicality. Let us first note that $\widehat{(YM)}[i]$  is a formally integrable and completely integrable quantum super PDE. In fact, let us denote  $B=\mathbb{R}\times A$, the quantum superalgebra, model of the quantum supermanifold $\hat J{\it D}^2(W[i])$, with $A$ the quantum superalgebra such that $\dim_A\mathfrak{g}=(r|s)$. Then one has $\dim_B\hat J{\it D}^2(W[i])=(4,(r|s)60)$, $\dim_B\hat J{\it D}^3(W[i])=(4,(r|s)140)$, $\dim_B\widehat{(YM)}[i]=(4,(r|s)56)$ and $\dim_B(\widehat{(YM)}[i])_{+1}=(4,(r|s)120)$, where $(\widehat{(YM)}[i])_{+1}$ denotes the first prolongation of $\widehat{(YM)}[i]$. The we have the surjection of the canonical mapping $\pi_{3,2}:(\widehat{(YM)}[i])_{+1}\to \widehat{(YM)}[i]$. In fact, one has: $
\dim_B(\widehat{(YM)}[i])_{+1}=\dim_B\widehat{(YM)}[i]+\dim_B(\widehat{g}[i])_{+1}$,
where $(\widehat{g}[i])_{+1}$ is the symbol of $(\widehat{(YM)}[i])_{+1}$, and $\dim_B(\widehat{g}[i])_{+1}=(0,(r|s)64)$. This property can be extended for iteration to any order of prolongation $h\ge0$. Then, taking into account that $B$ has Noetherian center, we can state that $\widehat{(YM)}[i]$ is regular and $\delta$-regular in the sense specified in \cite{PRAS11,PRAS14}. This is enough to state that $\widehat{(YM)}[i]$ is formally integrable, and since it is of class $Q^\omega_w$, it follows that it is also completely integrable. By following a similar road to look to the formal integrability of $\widehat{(YM)}[i]_\bullet$, we arrive to prove that it is not formally integrable, since there is not the surjectivity between the first prolongation $(\widehat{(YM)}[i]_\bullet)_{+1}$ and $\widehat{(YM)}[i]_\bullet$. On the other hand, by following our geometric theory of quantum (super), we can identify another equation, say $\widetilde{\widehat{(YM)}[i]_\bullet}\subset \hat J{\it D}^2(W[i])$, that is formally integrable and completely integrable and that has the same solutions of $\widehat{(YM)}[i]_\bullet$. This equation is reported in (\ref{formally-integrable-pde-associated-to-zero-defect-quantum-electric-charge}).
\begin{equation}\label{formally-integrable-pde-associated-to-zero-defect-quantum-electric-charge}
 \scalebox{0.8}{$ \widetilde{\widehat{(YM)}[i]_\bullet}\subset \hat J{\it D}^2(W[i]):\hskip 2pt \left\{
   \begin{array}{l}
     (\partial_{\alpha}.\widetilde{R}^{K\alpha\beta})+[\widehat
C^K_{IJ}\widetilde{\mu}^I_{\alpha},\widetilde{R}^{J\alpha\beta}]_+=0 \\
    \widetilde{R}^{\beta\alpha}_K=0\\
    \widetilde{R}^{\beta\alpha}_{K,\gamma}=0\\
   \end{array}
   \right.$}
\end{equation}
Then one can see that the canonical mapping $(\widetilde{\widehat{(YM)}[i]_\bullet})_{+1}\to \widetilde{\widehat{(YM)}[i]_\bullet}$ is surjective. In fact one has the equation reported in (\ref{condition-formally-integrable-pde-associated-to-zero-defect-quantum-electric-charge}).
\begin{equation}\label{condition-formally-integrable-pde-associated-to-zero-defect-quantum-electric-charge}
\scalebox{0.8}{$\left\{
\begin{array}{ll}
   [\dim_B(\widetilde{\widehat{(YM)}[i]_\bullet})_{+1}=(4,(r|s)30)]&=[\dim_B\widetilde{\widehat{(YM)}[i]_\bullet}=(4,(r|s)26)]\\
   &+[\dim_B(\widetilde{\widehat{g_2}[i]_\bullet})_{+1}=(0,(r|s)4)].\\
\end{array}
\right.$}
\end{equation}
This can be generalized for iteration to any order of prolongation. Therefore, $\widetilde{\widehat{(YM)}[i]_\bullet}$ is regular and $\delta$-regular, hence formally integrable. Since it is also of class $Q^\omega_w$, we conclude that $\widetilde{\widehat{(YM)}[i]_\bullet}$ is a completely integrable quantum super PDE.

The fiber bundle structure of $\widetilde{\widehat{(YM)}[i]_\bullet}$ over $N$ follows from the following commutative and exact diagram.
$$\xymatrix{0\ar[r]&\widetilde{\widehat{(YM)}[i]_\bullet}\ar[d]\ar@{^{(}->}[r]&\widehat{(YM)}[i]\ar[d]\\
&N\ar[d]\ar@{=}[r]&N\ar[d]\\
&0&0\\}$$

We can see that exoticity of nonlinear quantum propagators $V$, $\partial V=N_0\bigcup P\bigcup N_1$, $\partial P=\partial N_0\bigcup \partial N_1$, disappears when the manifold $P$ has an orientable smooth structure. In fact, we have the following lemma.

\begin{lemma}[Criterion for non-exoticity]\label{criterion-for-non-exoticity}
A nonlinear quantum propagator $V$, $\partial V=N_0\bigcup P\bigcup N_1$, $\partial P=\partial N_0\bigcup \partial N_1$, of $\widehat{(YM)}[i]$ such that $P$ is a $3$-dimensional time-like smooth orientable manifold, has zero defect quantum electric-charge: $\mathfrak{Q}[V]=0$.
\end{lemma}

\begin{proof}
In fact taking into account equation (\ref{conservation-law-gauge-invariance-c}) we get
\begin{equation}\label{conservation-law-gauge-invariance-cca}
  \omega_q|_P=[\widetilde{R}^{\alpha 0}_{K}\xi^3_0+\widetilde{R}^{\alpha 1}_{K}\xi^3_1+\widetilde{R}^{\alpha 2}_{K}\xi^3_2-\widetilde{R}^{\alpha 3}_{K}]\Xi^K_\alpha \otimes d\xi^0\wedge d\xi^1\wedge d\xi^2.
\end{equation}
On the other hand taking into account that $P$ is a smooth orientable manifold, we can consider that the condition $x^3=x^3(x^1,x^2,x^4)$ holds in some neighbourhood of any of its points, (up to a change of independent coordinates). Then considered a suitable finite open covering $\{U_i\}$ of $P$ one has on each open set $U_i$,
\begin{equation}\label{conserved-electric-charge-a}
  \widetilde{R}^{\alpha 3}_{K}=\widetilde{R}^{\alpha 0}_{K}\xi^3_0+\widetilde{R}^{\alpha 1}_{K}\xi^3_1+\widetilde{R}^{\alpha 2}_{K}\xi^3_2.
\end{equation}
Thus $\omega_q|_{U_i}=0$. Since we can write $\omega_q|_P=\sum_i\alpha_i\omega_q|_{U_i}$, with $\{\alpha_i\}$ a finite partition of unity subordinated to the covering $\{U_i\}$, we conclude that $\omega_q|_P=0$. Therefore, when $P$ is a smooth orientable manifold, one has that $\mathfrak{Q}[V]=0$.
\end{proof}

From above results it follows that a nonlinear quantum propagator $V$ of $\widehat{(YM)}[i]_\bullet$, $\partial V=N_0\bigcup P\bigcup N_1$, $\partial P=\partial N_0\bigcup\partial N_1$, does not necessitate to have $P$ with a smooth orientable manifold structure. (See also relation between formal integrability and integral bordism groups in \cite{PRAS11,PRAS14,PRAS21,PRAS22}.) In conclusion for a general nonlinear quantum propagator $V$, $\mathfrak{Q}[V]$ does not necessitate to be zero.
\end{proof}

\begin{cor}
Nonlinear quantum propagators $V\subset\widehat{(YM)}[i]$, such that $V\not\subset\widetilde{\widehat{(YM)}[i]_\bullet}$, can have ${\mathfrak{Q}[V]}\not=0$, hence they can be $Q$-exotic nonlinear quantum propagators.
\end{cor}

\begin{example}
A particular important case where ${\mathfrak{Q}[V]}=0$, is when $\widetilde{R}^K_{\alpha\beta}=0$, i.e., $V$ is a quantum flat nonlinear propagator. For such a case $V$ cannot have mass-gap, (Corollary 2.5 in \cite{PRAS11} or Corollary 5.2 in \cite{PRAS14} and Theorem 3.19 in \cite{PRAS22}), and neither quantum electric-gap. Therefore, such a $V$ is necessarily contained in $\widetilde{\widehat{(YM)}[i]_\bullet}$, and $V\bigcap\widehat{(Higgs)}=\varnothing$, $V\bigcap\widehat{(YM)}[i]_w=\varnothing$, where $\widehat{(YM)}[i]_w\subset \widehat{(YM)}[i]$ is the {\em quantum electromagnetic-Higgs PDE}. (See part II.) Such a quantum propagator can be one encoding photons, gluons or non-massive neutrinos.
\end{example}

\begin{example}
Let us consider a massive, electric-charged particle $a\subset\widehat{(YM)}[i]$ that is in a steady-state. Its nonlinear quantum propagator $V\subset \widehat{(YM)}[i]$ is such that $\partial V=a\sqcup P \sqcup a$, with $P\cong \partial a\times I$. Then if $a$ is a smooth orientable $3$-dimensional manifold, so is $P$ and from Lemma \ref{criterion-for-non-exoticity} we can state that now $\mathfrak{Q}[V]=0$. This means that $V\subset \widehat{(Higgs)}$, but also that $V$ satisfies the constraint $\widehat{(YM)}[i]_\blacksquare$, given in (\ref{sub-equation-zero-defect-quantum-electric-charge-a}). This condition holds for all massive particles in steady-state, and agrees with the fact that in steady-state configurations the observed phenomenological quantum Hamiltonian $H[i|t]$ is constant. (Example 3.24 in \cite{PRAS28}.) This does not contradict what we know that nonlinear quantum propagators for massive particles cannot be quantum flat. In fact, this means only that nonlinear quantum propagators for massive particles, hence contained inside $\widehat{(Higgs)}[i]$, can conserve electric charge, even if they are not quantum flat. In other words the quantum flatness is a condition sufficient but not necessary to have ${\mathfrak{Q}[V]}=0$. By conclusion a quantum massive particle in steady-state, conserves its quantum electric charge despite its nonlinear quantum propagator is not quantum flat.

\end{example}
\begin{figure}[h]
\includegraphics[width=4cm]{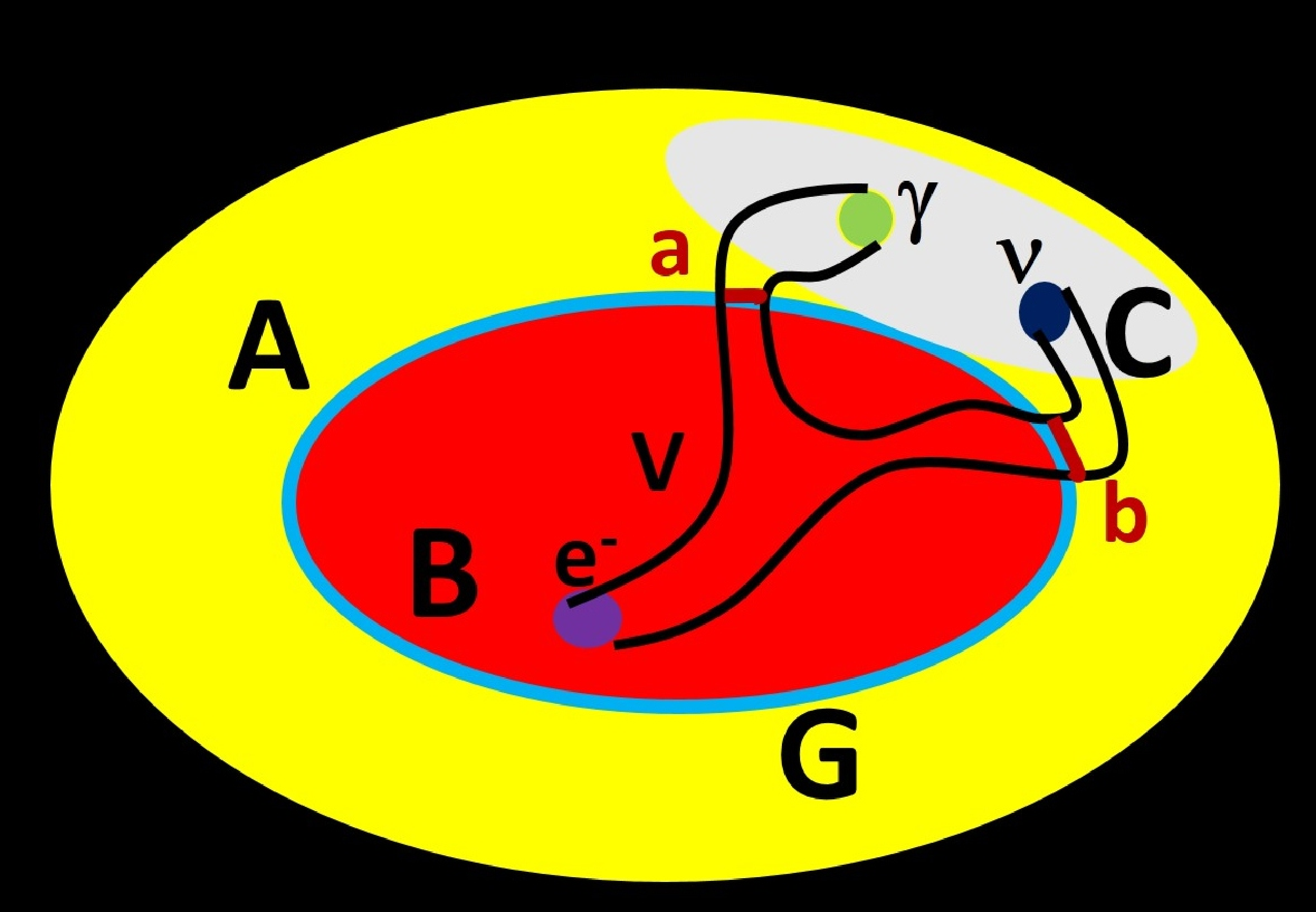}
\renewcommand{\figurename}{Fig.}
\caption{$Q$-exotic electron decay: $e^-\to \gamma+\nu$. There $A$, $B$ and $C$ represent respectively $\widehat{(YM)}[i]$, $\widehat{(Higgs)}[i]$ and $\widetilde{\widehat{(YM)}[i]_\bullet}$. $V$ is the $Q$-exotic nonlinear quantum propagator bording $e^-$ with $\gamma\bigcup \nu$. Therefore $V$ is not globally contained in $\widetilde{\widehat{(YM)}[i]_\bullet}$ and neither satisfies the constraint $\widehat{(YM)}[i]_\blacksquare$. Furthermore, since $V$ bords the massive electron, contained into $\widehat{(Higgs)}[i]$, with the non-massive particles $\gamma$ and $\nu$, it must necessarily cross the Goldstone boundary, $G$, of $\widehat{(Higgs)}[i]$. This has the effect to produce two virtual massive particles, denoted in the picture by the letters $a$ and $b$.}
\label{figure-q-exotic-electron-decay}
\end{figure}

\begin{example}
An exotic nonlinear quantum propagator $V$ that encodes the decay of electron $e^-\to \gamma+\nu$ is one that bords an initial Cauchy $e^-\subset \widehat{(Higgs)}[i]\subset\widehat{(YM)}[i]$ with a Cauchy data $\gamma\bigcup\nu\subset \widetilde{\widehat{(YM)}[i]_\bullet}$. Since $\mathfrak{Q}[V]\not=0$, $V$ must necessarily propagate outside $\widetilde{\widehat{(YM)}[i]_\bullet}$. Let us recall that $\gamma$ and $\nu$ cannot be contained in $\widehat{(Higgs)}[i]$, where, instead, lives electron. (See Fig. \ref{figure-q-exotic-electron-decay}.)
\end{example}
\begin{example}
A similar exotic nonlinear quantum propagator is one that encodes a $Q$-exotic neutron decay $n\to p+\nu+\bar\nu$.
\end{example}

\begin{definition}[Exotic nonlinear quantum propagators]\label{exotic-nonlinear-quantum-propagator}
We say that a nonlinear quantum propagator $V$ of $\widehat{(YM)}[i]$, $\partial V=N_0\bigcup P\bigcap N_1$, $\partial P=\partial N_0\bigcup N_1$, is {\em exotic} if $P$ has singular point, or it is not an orientable $3$-dimensional smooth manifold.
\end{definition}

In the following we resume some fundamental characteristics for exotic nonlinear quantum propagators.
\begin{theorem}[Exotic nonlinear quantum propagators properties]\label{properties-exoticnonlinear-quantum-propagators}

$\bullet$\hskip 2pt Exotic nonlinear quantum propagators of $\widehat{(YM)}[i]$ cannot be contained into $\widetilde{\widehat{(YM)}[i]_\bullet}\subset\widehat{(YM)}[i]$.

$\bullet$\hskip 2pt Exotic nonlinear quantum propagators of $\widehat{(YM)}[i]$ have non-zero defect quantum electric-charge: $\mathfrak{Q}[V]\not=0$.

$\bullet$\hskip 2pt Exotic nonlinear quantum propagators of $\widehat{(YM)}[i]$ have non-zero defect observed phenomenological quantum energy: $\mathfrak{H}[V]_{\partial}\not=0$.

$\bullet$\hskip 2pt Exotic nonlinear quantum propagators of $\widehat{(YM)}[i]$ have non-zero defect observed quantum energy: $\mathfrak{H}[V]\not=0$.

$\bullet$\hskip 2pt Exotic nonlinear quantum propagators of $\widehat{(YM)}[i]$ have non-zero defect observed quantum $r$-momentum: $\mathfrak{P}_r[V]\not=0$, $r=1,2,3$.

$\bullet$\hskip 2pt Exotic nonlinear quantum propagators of $\widehat{(YM)}[i]$ have non-zero defect observed quantum $(\mu,\nu)$-angular-momentum: $\mathfrak{M}_{\mu\nu}[V]\not=0$, $\mu,\, \nu\in\{1,2,3,4\}$.

\end{theorem}
\begin{proof}
The proof is a direct consequence of above definitions, results also contained in Part I.
\end{proof}

\begin{example}[The Quantum Exotic $Zc(3900)$ decay]\label{quantum-exotic-zc-3900-decay}
Two independent research-teams and Laboratories, Belle-Kek of Tsukuha in Japan, and BESIII-The Beijing Electron-Positron Collider of Pechino in China, have recently studied the reaction $e^+e^-\to Y\to \pi^+\pi^-J/\psi$ at $4.26$ Gev, observing the presence of an intermediate bound state about $3.9$ Gev, that have called $Zc(3900)$. The actual more accepted interpretation of such new quasi-particle $Zc(3900)$ is to consider it like a charged tetraquark meson made by $\{c,\bar c,\bar d, u\}$ or $\{c,\bar c,\bar u, d\}$, according to the sign of the charge. It is important to note that this interpretation is motivated by the fact that since the conservation of charge is a rigid law in the usual Standard Model, one cannot admit a decay e.g., $Z_c(3900)\to \pi^+J/\psi$ from a neutral charmonium $Z_c(3900)$, i.e., simply made by a couple $\{c,\bar c\}$. Thus under the necessity to conserve electric charge, physicists have guessed $Z_c(3900)$ as a charged tetraquark particle or an hadron-molecule.\cite{ABLIKIM, LIU} (For more information look, e.g., \href{http://en.wikipedia.org/wiki/Zc(3900)}{Wikipedia - $Zc(3900)$}.) For the moment this is only an ansatz, that necessitates of further experimental data to be confirmed.
On the other hand we can prove that in the framework of the Algebraic Topology of quantum super PDEs, as formulated by A.Pr\'astaro, it is possible to describe such new particle $Z_c(3900)$ like a neutral excited state of charmonium, that is more naturally related to the charmonium $Y$. The price to pay is just the non-conservation of charge ! But we have proved in Theorem \ref{non-conservation-of-quantum-electric-charge}, that processes with such characteristics exist in the quantum super Yang-Mills equation. Therefore, it is not necessary to assume that $Z_c(3900)$ has an exotic tetraquark structure ! The reaction $e^+e^-\to Y\to \pi^+Z_c(3900)\to\pi^-J/\psi$ at $4.26$ Gev, is represented in Fig. \ref{zc-3900-decay}. Here one can see that the nonlinear quantum propagator $V$, bording $N_0=e^+\sqcup e^-$ with $N_1=\pi^+\sqcup\pi^-\sqcup J/\psi$, has boundary $\partial V=N_0\bigcup P\bigcup N_1$, such that $N_0,\, N_1\subset\widehat{(Higgs)}[i]\subset\widehat{(YM)}[i]$, since all the particles involved have non-zero mass. Furthermore, $V$ admits the following decompositions: $V=V_0\bigcap V_1\bigcap V_1$, where $\mathfrak{Q}[V_0]=0$, $\mathfrak{Q}[V_1]\not=0$ and $\mathfrak{Q}[V_2]\not=0$. Therefore $V_0$ satisfies the constraint $\widehat{(YM)}[i]_{\blacksquare}$, i.e., the {\em(zero-defect-quantum-electric-charge-constraint)}, but not $V_1$ and $V_2$. These last, being massive particles, with non-zero defect quantum electric charge, cannot neither stay inside $\widetilde{\widehat{(YM)}[i]_\bullet}$. More precisely, $V$ admits the following decomposition $V=V_0\bigcup V_1\bigcup V_2$, such that $\partial V_0=N_0\bigcup P_0\bigcup Y(4260)$, where $Y(4260)$ represents a charmonium at $4.26$ Gev. Furthermore, $\partial V_1=Y(4260)\bigcup P_1\bigcup \{\pi^+\sqcup Zc(3900)\}$, where $Zc(3900)$ represents a charmonium at $3.9$ Gev, and $\partial V_2=Zc(3900)\bigcup P_2\bigcup \{\pi^-\sqcup J/\psi\}$, where $J/\psi$ is again a charmonium particle. The nonlinear quantum propagators $V_1,\, V_2$ do not respect the constraint $\widehat{(YM)}[i]_\blacksquare$. Therefore the nonlinear quantum propagator $V$, even if has a first part $V_0$, respecting the electric charge conservation, globally cannot be contained into $\widetilde{\widehat{(YM)}[i]_\bullet}$ and neither respects the constraint $\widehat{(YM)}[i]_\blacksquare$, hence represents a $Q$-exotic nonlinear quantum propagators of $\widehat{(YM)}[i]$.
\end{example}

\begin{figure}[h]
\includegraphics[width=4cm]{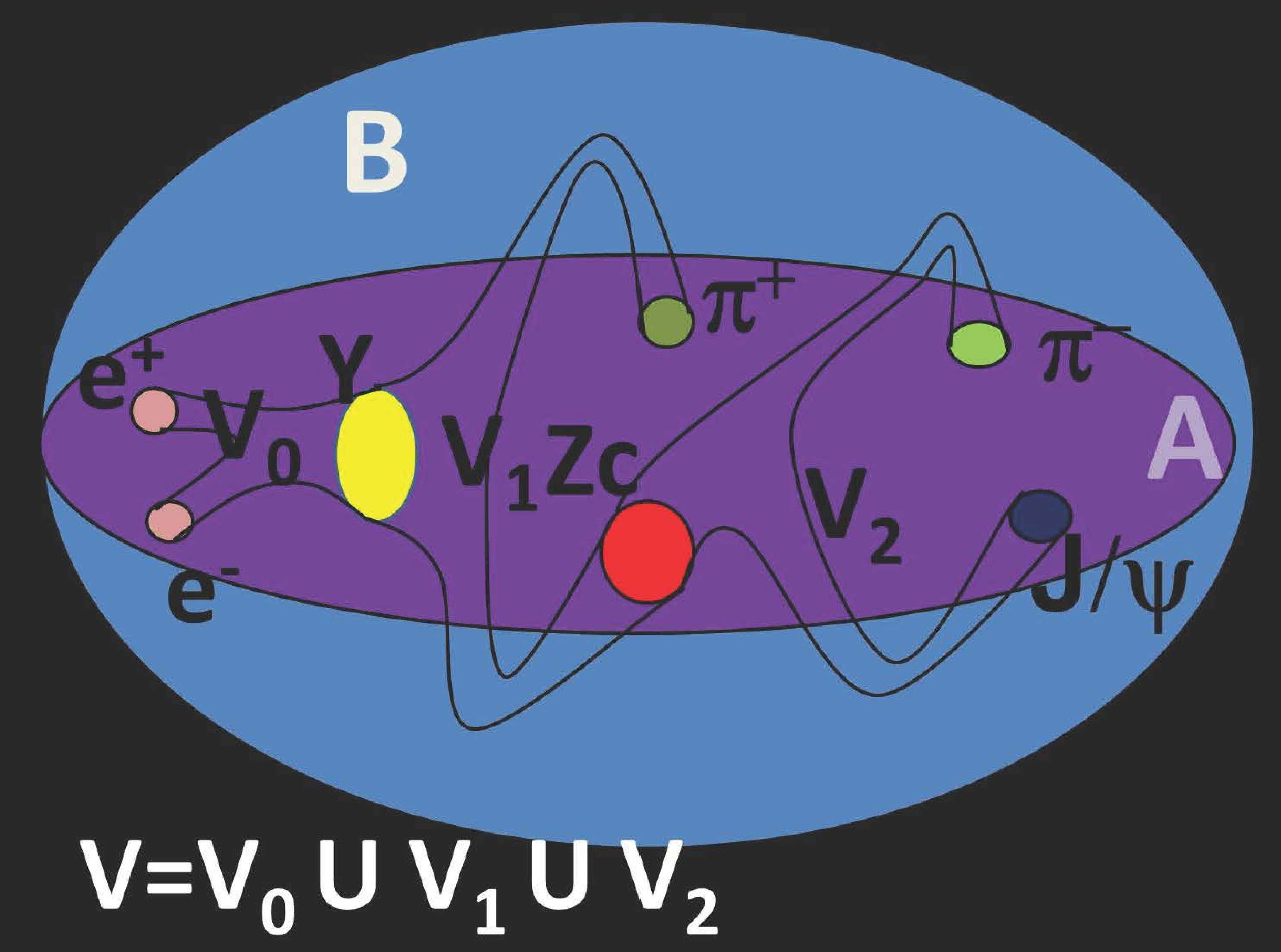}
\renewcommand{\figurename}{Fig.}
\caption{Representation of the exotic nonlinear quantum propagator $V$ encoding the reaction $e^+e^-\to \pi^+\pi^- J/\psi$ and relative splitting $V=V_0\bigcup V_1\bigcup V2$, in such a way to form intermediate neutral charmonium quasi-particles $Y(4260)$ and $Zc(3900)$. The nonlinear quantum propagator parts $V_1$ and $V_2$ do not conserve the charge. This is interpreted as a quantum supergravity effect, that it is impossible to describe remaining in the Standard Model. In the picture $A=\widehat{(YM)}[i]_\blacksquare$ and $B=\widehat{(YM)}[i]$.}
\label{zc-3900-decay}
\end{figure}

\section{\bf Quantum entanglement in nonlinear quantum propagators}\label{sec-quantum-entanglement-in-quantum-nonlinear-propagators}

In this section we shall relate nonlinear quantum propagators of quantum super Yang-Mills PDEs with the so-called {\em quantum entanglement phenomenon}. In quantum mechanics a composite system is called quantum entangled if it is encoded by global states which cannot be represented as products of the states of individual subsystems. This phenomenon has as a consequence that if a subsystem is observed, the other parts of the composite system have instantaneous reactions. This aspect were considered in contradiction with the concept of causality. In particular, quantum entanglement were considered in conflict with the Einstein's General Relativity. (See \cite{BELL, EINSTEIN-PODOLSKY-ROSEN, NEUMANN, SCHRODINGER}.) Nowadays the quantum entanglement is a well experimentally accepted and verified phenomenon. (See, e.g., \cite{BEHBOOD-ET-ALT,ZEILINGER}.) Let us emphasize that this paradox has been completely solved in the Pr\'astaro's formulation of quantum gravity in the framework of the Algebraic Topology of quantum PDE's. See \cite{PRAS19}, where the observed quantum gravity has been related to the EPR paradox. The following theorems will give a more detailed justification of the quantum entanglement on the ground of the bordism structure of nonlinear quantum propagators and quantum conservation laws as formulated by A. Pr\'astaro.

\begin{theorem}[Quantum entanglement in {$\widehat{(YM)}[i]$}]\label{quantum-entanglement-in-observed-quantum-yang-mills-pdes}
Let $V\subset \widehat{(YM)}[i]$ be an nonlinear quantum propagator. If $\partial V=N_0\sqcup P\sqcup N_1$, then $N_0$ and $N_1$ are {\em quantum entangled particles}, i.e.,
\begin{equation}\label{exotic-quantum-entangled-particles}
<N_0,\omega>=<N_1,\omega>\, {\rm mod} <\omega,P>\in A
\end{equation}

for any quantum conservation law $\omega$ of $\widehat{(YM)}[i]$.
\end{theorem}

\begin{proof}
In fact, for any quantum conservation law $\alpha$ of $\widehat{(YM)}[i]$, we get $<N_0,\alpha>-<N_1,\alpha>=<\alpha,P>\in A$, where $A$ is the fundamental quantum (super)algebra of the $\widehat{(YM)}[i]$. Then quantum super-number of $N_0$ are ``{\em instantaneously}" related to quantum super-numbers of $N_1$ by means of the quantum supernumbers of $P$.
\end{proof}

\begin{example}
In Part II are considered many examples of nonlinear quantum propagators encoding specific quantum reactions. In all those cases outgoing particles are all quantum entangled each other. For example in the annihilation electron-positron: $e^++e^-\to\gamma+\gamma$, photons are quantum entangled each other. But what is less considered, is that photons are quantum entangled to both $e^+$ and $e^-$ too. In fact, this follows from the quantum conservation laws and the structure of nonlinear quantum propagator encoding such a reaction. But this has nothing to do with the violation of causality ! It is a simple direct consequence of the noncommutative logic of the microscopic world applied to the algebraic topological structure of nonlinear quantum propagators !
\end{example}

\begin{figure}[h]
\includegraphics[width=4.3cm]{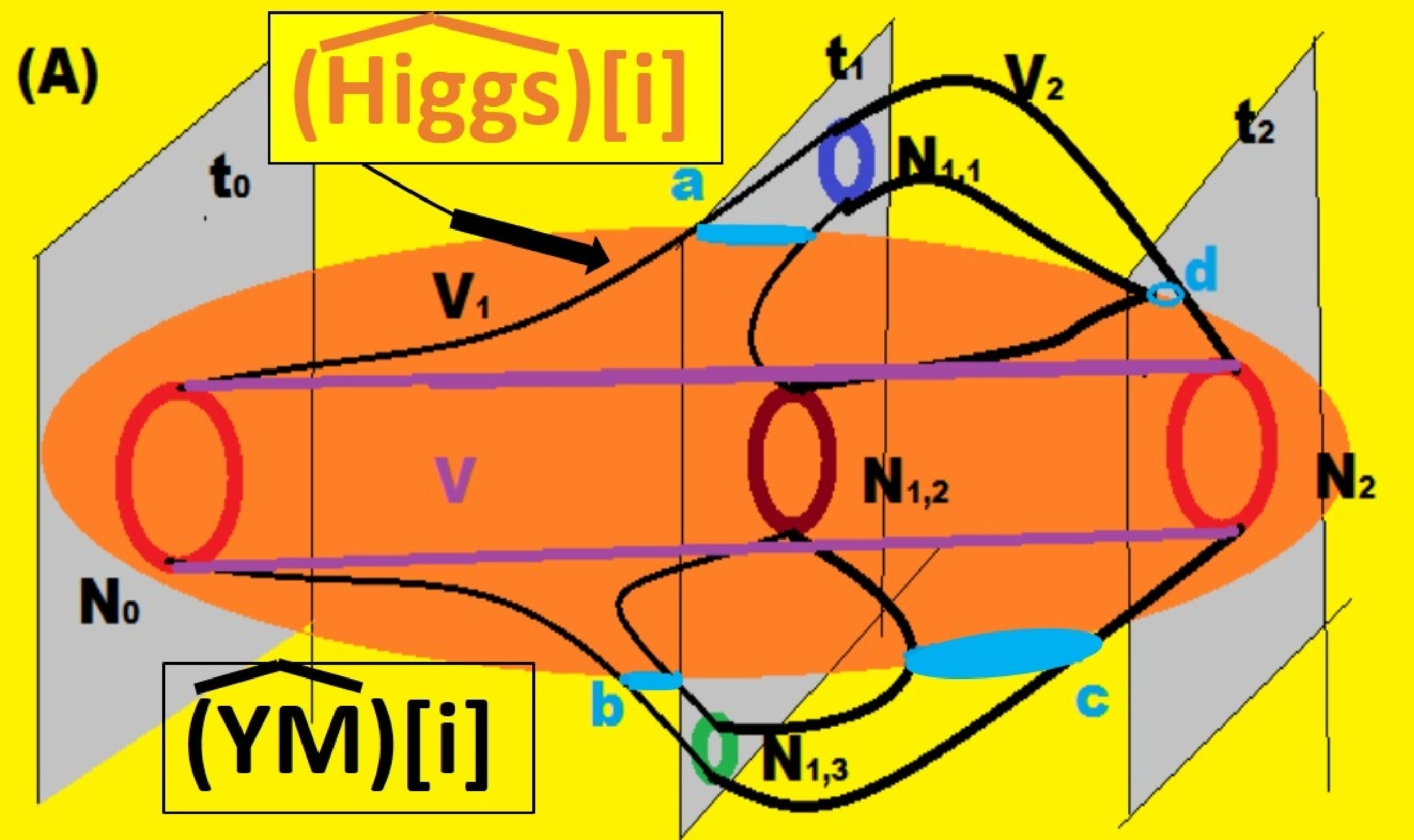}
\includegraphics[width=4cm]{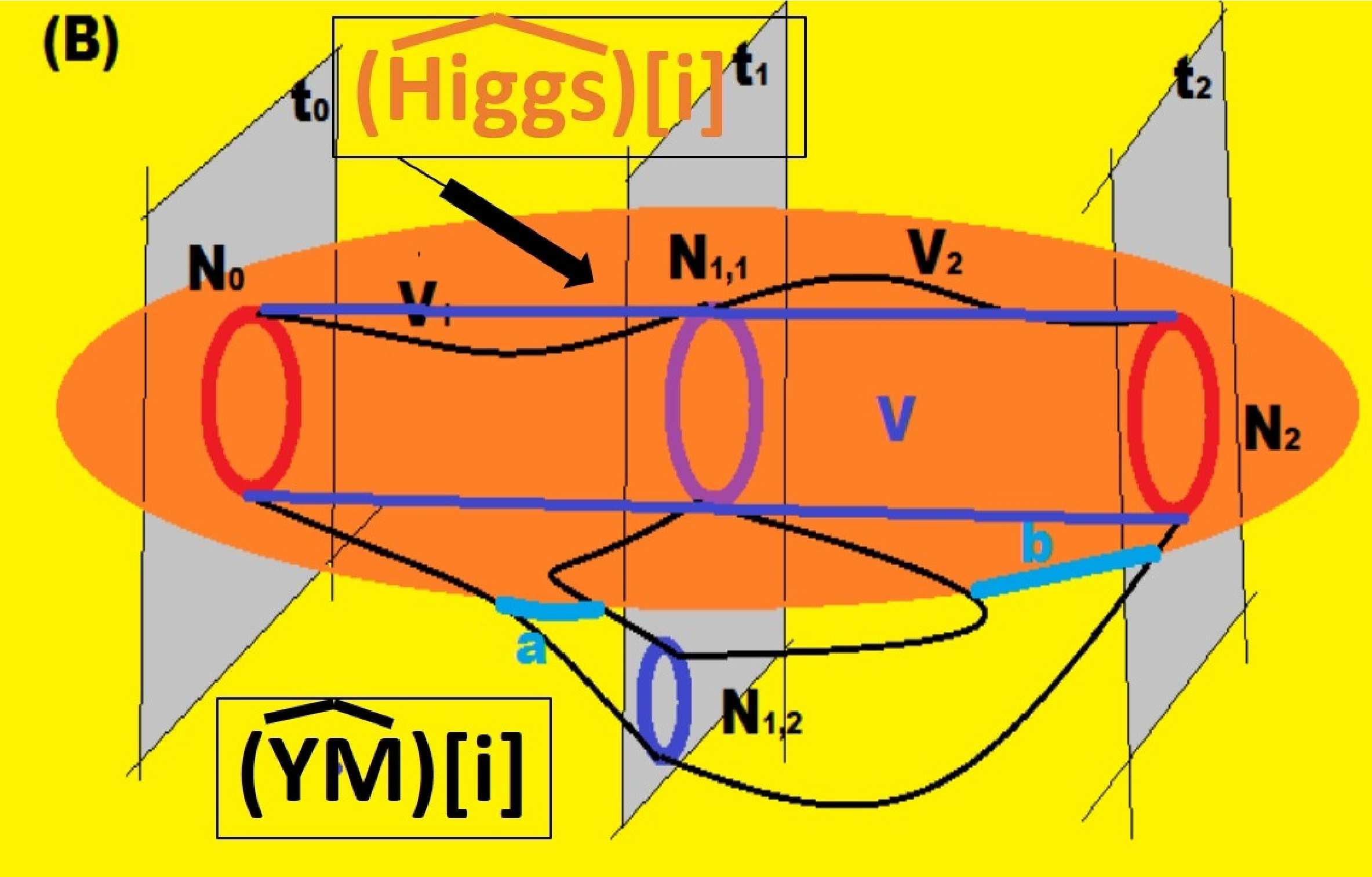}
\renewcommand{\figurename}{Fig.}
\caption{(A): Representation of a quantum Cheshire cat nonlinear propagator that spatially separates mass and spin from its original quantum particle, and next recombines these qualities in the same space-like quantum particle but at a next time. (B): Representation of a quantum Cheshire cat nonlinear propagator that spatially separates mass and electric charge from its original quantum particle and next recombines these qualities in the same space-like quantum particle, but at a next time. Such nonlinear quantum propagators are allowed on the ground of the Algebraic Topologic theory of observed quantum super Yang-Mills PDEs. In both figures the (yellow) framework represents the observed quantum super Yang-Mills PDEs, $\widehat{(YM)}[i]$, and the oval colored region represents the sub-equation $\widehat{(Higgs})[i]$, where live solutions with mass-gap. In both figures (A) and (B), the cylinders $V$ bording $N_0$ with $N_1$ represent the stationary smooth solutions.}
\label{figure-quantum-cheshire-cat-nonlinear-propagators}
\end{figure}

\begin{example}[Quantum Cheshire cat nonlinear propagators]\label{quantum-cheshire-cat-nonlinear-propagators}
The so-called ``quantum Cheshire cat" paradox is strictly related to the concept of quantum entanglement in nonlinear quantum propagators.
It refers to a quantum process where a quality of a particle, e.g., its quantum spin, can be spatially separated from the particle itself, but also recombined to it. This apparent paradox takes the name by the well known Cheshire cat in the Lewis Carroll's book, {\em Alice in Wonderland}, that can be separated from its grin, and, next, magically recombined to it. Nowadays there are experimental observations that prove such possibility realized at the quantum level. (See \cite{DENKMAYR-GEPPERT-SPONAR-LEMMEL-MATZKIN-TOLLAKSEN-HASEGAWA}). Such a phenomenon was first theoretically foreseen by means of the standard quantum mechanics, on the ground of weak-measurements. Some first interpretations considered this paradoxical behaviour as an illusion. For example, in the case of photons, the experiments do not measure the location of the photons and their polarizations simultaneously, but at different times. However recent experiments appear to prove that quantum Cheshire cats can be realized in the quantum world despite their apparent nonsense. (See also \cite{AHARONOV-ET-AL,AHARONOV-ROHRLICH,AHARONOV-POPESCU-ROHRLICH-SKRZYPCZYK,MATZKIN-PAN}.)

Here we shall prove that quantum Cheshire cats are natural phenomena in the Algebraic Topology of quantum super Yang-Mills PDEs, namely in the Pr\'astaro theory that has unified Einstein's General Relativity with Quantum Mechanics.
In fact, the observed quantum super Yang-Mills PDEs, $\widehat{(YM)}[i]$, admits nonlinear quantum propagators $V$ that can separate a particle from its qualities. We call such propagators {\em quantum Cheshire cat nonlinear propagators}.
Really, let $N_0\subset \widehat{(YM)}[i]$ be a space-like Cauchy data representing a quantum massive, charged particle with quantum spin in stationary state. We shall prove that exist nonlinear quantum propagators such that all these qualities of $N_0$ can be delocalized from $N_0$ at some instant $t_1>t_0$, where $t_0$ is the initial time of $N_0$, and then recombined together at some $t_2>t_1$. In fact we can consider a nonlinear quantum propagator $V_1$, such that $\partial V_1=N_0\bigcup P_1\bigcup (N_{1,1}\bigcup N_{1,2}\bigcup N_{1,3})$, where $\partial P_1=\partial N_0\bigcup \partial (N_{1,1}\bigcup N_{1,2}\bigcup N_{1,3})=\partial N_0\bigcup \partial N_{1,1}\bigcup \partial N_{1,2}\bigcup \partial N_{1,3}$. Let us denote respectively by $\hat m(N)$ and $\hat s(N)$ the quantum mass and quantum spin of the space-like Cauchy data $N$. Let us assume the quantum conditions reported in Tab \ref{quantum-condition-example-quantumcheshire-cat-propagator}. We call respectively $N_{1,1}$ the {\em naked $N_0$-particle}, $N_{1,2}$ the {\em massive half-naked $N_0$-particle} and $N_{1,3}$ the {\em mass-free half-naked $N_0$-particle}. From what we said in our previous definitions, it results that $N_0\subset \widehat{(Higgs)}[i]$, $N_{1,2}\subset \widehat{(Higgs)}[i]$, but $N_{1,1}\not\subset\widehat{(Higgs)}[i]$ and $N_{1,3}\not\subset \widehat{(Higgs)}[i]$. The quality of the space-like unconnected Cauchy data $N_1=N_{1,1}\bigcup N_{1,2}\bigcup N_{1,3}$ is exactly the same of $N_0$, but in this last the qualities quantum mass and quantum spin are located on the same connected component, hence they refer to the same space-like region, instead in the case $N_1$, they are spatially separated. These qualities can be recombined at a time $t_2>t_1$, whether one considers a nonlinear quantum propagator $V_2$, such that $\partial V_2=N_1\bigcup P_2\bigcup N_2$, such that $\partial P_2=\partial N_1\bigcup \partial N_2$, with $\hat m(N_2)=\hat m_0$ and $\hat s(N_2)=\hat s(N_0)$. Let us note that standing the invariance of the equation $\widehat{(YM)}[i]$ for time-translations, we can reproduce $N_0$ at any time $t_2>t_0$, as an orientable stationary space-like Cauchy data $N_2$, having the same quantum energy of $N_1$. Since $N_0$ and $N_2$ have the same quantum numbers and are related by a rigid time translation, hence are related by the identity, there exists a smooth solutions $V$ bording $N_0$ with $N_2$. Such a solution necessarily coincides with the stationary one bording $N_0$ with $N_2$. (See Example 3.23 in part I.) The solutions $V_1$ and $V_2$ are instead singular solutions. Their existence is guaranteed by what considered on integral bordism groups of $\widehat{(YM)}$ and $\widehat{(YM)}[i]$ in part I (Lemma 3.19(I) ) and in part II (Section 2). In particular one has $0=<\alpha,N_0>+<\alpha,N_1>+<\alpha,P_1>$, with $\partial P_1=\partial N_0\bigcup \partial N_1$, for all conservation laws $\alpha$. Then the $4$-dimensional smooth integral manifold $V_1$ such that $\partial V_1=N_0\bigcup P_1\bigcup N_1$, is a singular solution representing a nonlinear quantum propagator from $N_0$ and $N_1$. Similarly, there exists a singular solution $V_2$ bording $N_1$ with $N_2$. Therefore, we can state the existence of a singular nonlinear quantum propagator $X=V_1\bigcup V_2$ bording $N_0$ with $N_2$, and passing trough $N_1$. Thus in the framework of the Pr\'astaro's Algebraic Topology of quantum super Yang-Mills PDEs, the quantum Cheshire cat phenomenon is assured by the fact that the stationary smooth solution $V$, $\partial V=N_0\bigcup R\bigcup N_2$, can be represented by a singular solution $X=V_1\bigcup V_2$. (We can also call $X$ a quantum Cheshire cat propagator.) In other words the quantum Cheshire cat phenomenon is not a quantum paradox but a structural bordism property of the observed quantum super Yang-Mills PDEs. Let us also emphasize that by changing the quantum qualities in $N_0$ we can corresponding change the quantum qualities in $N_1$ and $N_2$ as quantum entangled effects. In other words quantum mass, quantum spin and other quantum qualities, can be spatially separated by means of the algebraic topological properties of nonlinear quantum propagators encoding dynamics in the observed quantum super Yang-Mills PDEs, $\widehat{(YM)}[i]$. It is interesting to emphasize that the nonlinear quantum propagator $V_1$ must necessarily intersect the Goldstone boundary into two regions, identifying two massive virtual particles $a$ and $b$ before to bord with $N_{1,1}$ and $N_{1,3}$. The same happens with $V_2$. It identifies two massive virtual particles $c$ and $d$ before to bord $N_1$ with $N_2$, since this last is contained into $\widehat{(Higgs)}[i]$. (See Fig. \ref{figure-quantum-cheshire-cat-nonlinear-propagators}(A).\footnote{In that picture (and in similar ones reported below) the planes $t_0$, $t_1$ and $t_2$ mean that the corresponding framed particles belong to the fibers $\widehat{(YM)}[i]_t=\tau^{-1}(t)$, with $t=t_0,\, t_1,\, t_2$, respectively. Here $\tau$ is defined by composition: $\xymatrix{\widehat{(YM)}[i]\ar[r]^(0.6){\pi_2}\ar@/_1pc/[rrr]_(0.5){\tau}&N\ar@{=}[r]^(0.4){\backsim}&T\times S_\psi\ar[r]^(0.6){pr_1}&T}$ .})
\begin{table}[h]
\caption{Quantum conditions in example quantum Cheshire cat propagator.}
\label{quantum-condition-example-quantumcheshire-cat-propagator}
\begin{tabular}{|l|l|}
\hline
{\footnotesize\rm $N_0$}&{\footnotesize\rm$N_1=N_{1,1}\bigcup N_{1,2}\bigcup N_{1,3}$}\\
\hline \hline
\hfil{\footnotesize\rm $\hat m_0=\hat m(N_0) $}\hfil&{\footnotesize\rm$\hat m_{1,1}=\hat m(N_{1,1})=0 $}\\
\hline
 {\footnotesize\rm$\hat s_0=\hat s(N_0)$}&{\footnotesize\rm$\hat m_{1,2}=\hat m(N_{1,2})=\hat m_0 $}\\
\hline
 {\footnotesize\rm}&{\footnotesize\rm$\hat m_{1,3}=\hat m(N_{1,3})=0$}\\
\hline
 {\footnotesize\rm}&{\footnotesize\rm$\hat s_{1,1}=\hat s(N_{1,1})=0$}\\
\hline
 {\footnotesize\rm}&{\footnotesize\rm$\hat s_{1,2}=\hat s(N_{12})=0$}\\
\hline
 {\footnotesize\rm}&{\footnotesize\rm$\hat s_{1,3}=\hat s(N_{1,3})=\hat s_0$}\\
\hline
\end{tabular}
\end{table}

Quantum Cheshire cat nonlinear propagators can also separate electric charges from particles. For example we can start from a massive and electric charged quantum particle, identified with a stionary space-like Cauchy data $N_0\subset \widehat{(Higgs)}[i]\subset\widehat{(YM)}[i]$, and by means of a nonlinear quantum propagator $V_1$, $\partial V_1=N_0\bigcup Q_1\bigcup N_1$, with $N_1=N_{1,1}\bigcup N_{1,2}$, $\hat m(N_{1,1})=\hat m(N_0)=\hat m_0$,  $\hat q(N_{1,1})=0$, $\hat m(N_{1,2})=0$,  $\hat q(N_{1,2})=\hat q(N_0)=\hat q_0$. In this way one spatially separates the electric charge from the massive particle. Then by a next quantum nonlinear bordism $V_2$, $\partial V_2=N_1\bigcup N_2$, where $N_2$ is another stionary Cauchy data at a time $t_2>t_0$, having the same qualities of $N_0$, we can recombine the electric charge with the mass-free electric charged half-naked $N_0$-particle $N_{1,2}$. Similarly to the previous case one realizes two virtual massive particle $a=V_1\bigcup\widehat{(Goldstone)}[i]$ and $b=V_2\bigcup\widehat{(Goldstone)}[i]$. Let us emphasize that both nonlinear quantum propagators $V_1$ and $V_2$ satisfy the constraint $\widehat{(YM)}[i]_\blacksquare$, since they have $\mathfrak{Q}[-]=0$, but they cannot stay inside $\widetilde{\widehat{(YM)}[i]_\bullet}$, i.e., the completely integrable and formally integrable sub-equation of $\widehat{(YM)}[i]$, where live exotic-quantum supergravity-free solutions. In fact they propagate quantum massive particles, hence cannot have quantum flat propagators. (See Fig. \ref{figure-quantum-cheshire-cat-nonlinear-propagators}(B).) In other words quantum Cheshire cat nonlinear propagators cannot be  exotic-quantum supergravity free solutions.
\end{example}

\begin{example}[Quantum photons entangled with massive bound photons]\label{quantum-photons-entangled-with-massive-bound-photons}
Other important examples of quantum entanglements interest photons and massive bound-states of photons. Really quantum photons can be constrained to be confined in bound states thanks to the geometric structure of $\widehat{(YM)}$, that contains the sub-equation $\widehat{(Higgs)}$, and thanks to nonlinear quantum propagators of the type pictured in Fig. \ref{entanglement-photons-massive-bound-photons}. There $\sigma$ is a quantum massive particle representing two massive bound photons. Such particles can have spin $=0,\, 1,\, 2$. For example for $s=0$, $\sigma$ can be identified with a Higgs-particle $H^0$, for $s=1$ $\sigma$ can be taken as $Z^0$ and for $s=2$, we can consider $\sigma$ as a massive graviton $G'$.\footnote{The quantum particle graviton is usually considered a massless neutral particle of spin $2$. Let us emphasize also that recent experiments realized bound states of photons. See, e.g., \cite{CHANG, LUKIN-VULETIC}.} This result answer also to a question by F. J. Dyson about the existence of gravitons \cite{DYSON1, DYSON2}.\footnote{See also the related paper \cite{ROTHMAN-BOUGHN} and the following \href{http://en.wikipedia.org/wiki/Graviton}{Graviton-Wikipedia link}.} In fact, from above considerations we can assume also the existence of a massive quantum graviton particle, $G'$, with respect to a quantum relativistic frame, namely inside $\widehat{(Higgs)}[i]$. Then it should remain to prove the existence of a massless quantum graviton particle, namely a neutral massless quantum particle with spin $s=2$. But this can be similarly obtained by considering an inverse process of the previous one, and inside $\widehat{(YM)}[i]$. In other words, we can consider $G'$ as a bound state of two massless gravitons $G$, obtained by means of a nonlinear quantum propagator $V'$, such that $\partial V'=(G\bigcup G)\bigcup P'\bigcup G'$, $\partial P'=(\partial G\bigcup \partial G)\bigcup\partial G'$. (See Fig. \ref{entanglement-photons-massive-bound-photons}.) Let us emphasize that a bound state from two particle of spin $s=2$, can have $s=4,\, 3,\, 2,\, 1,\, 0$. Therefore, the assumption that $G'$ is a massive bound state of two massless graviton particle is justified. However, from these considerations we can also state existence of massive bound states $\sigma(G)$, of two quantum gravitons, having spin $s\in\{0,1,3,4\}$. By conclusion, we can answer to the Dyson's question by saying that even if it is very hard to detect massless quantum graviton particles, it should be less hard to observe neutral massive particles with spin $0\le s\le 4$, that we could more easily relate to massless neutral quantum particles with spin $s=2$, hence to quantum graviton particles.
\end{example}
\begin{figure}[h]
\includegraphics[width=4.3cm]{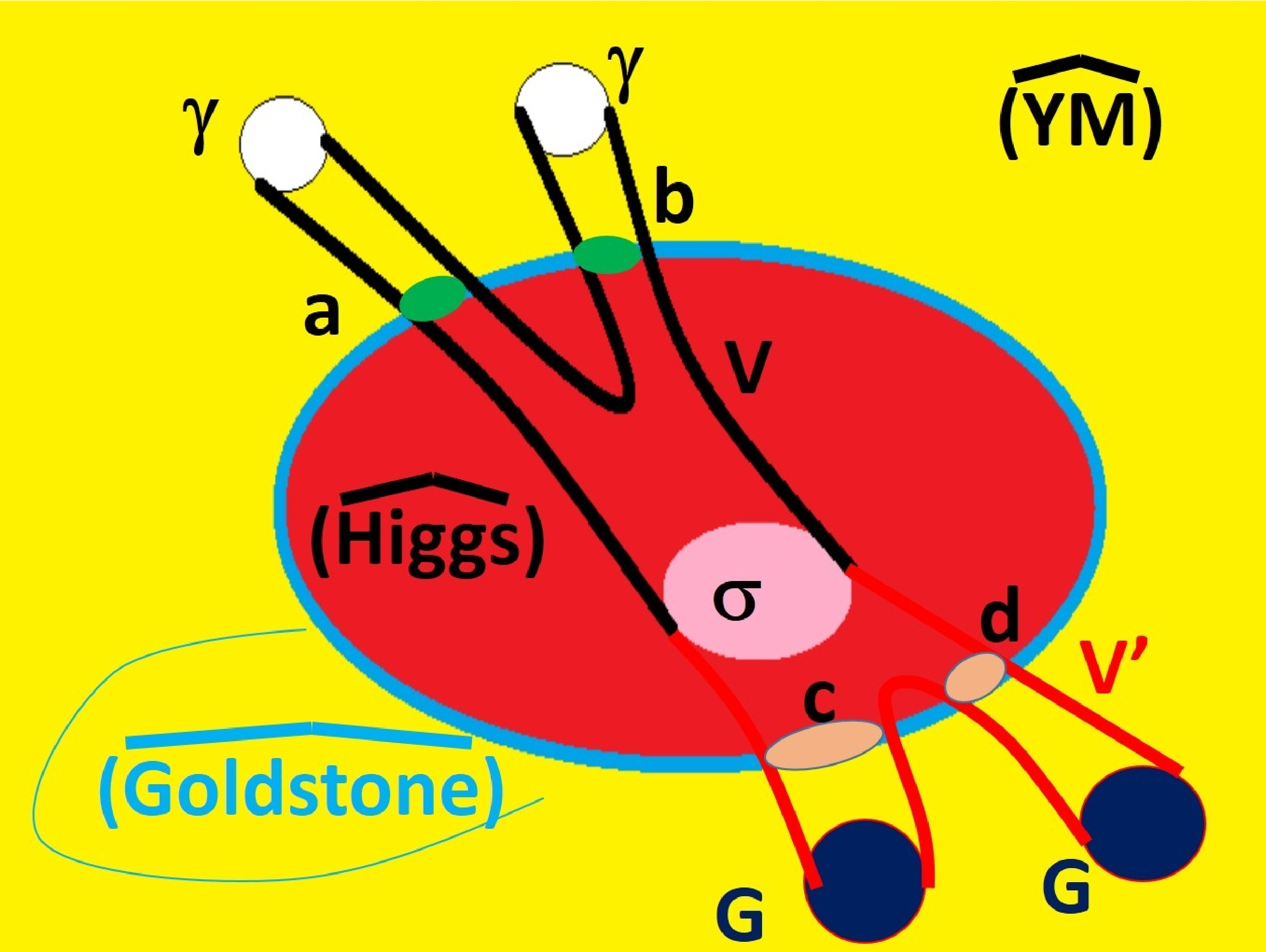}
\renewcommand{\figurename}{Fig.}
\caption{Representation of a nonlinear quantum propagator, $V$, encoding the realization of two-bound massive photons, represented in the picture by $\sigma$. This particle can be very massive, depending on the initial energy of the two photons $\gamma$. In this way one can produce also the following neutral quantum particles: Higgs particle ($H^0$, mass-gap $125.3\, GeV$, spin $s=0$), $Z$ boson particle ($Z^0$, mass-gap $91.2\, GeV$, spin $s=1$), massive graviton particle ($G'$, mass-gap $m_{G'}\, GeV$, spin $s=2$). (The mass of the graviton is left undetermined.) In the picture are also represented two massive photons $a$ and $b$ identified by nonlinear quantum propagator crossing the Goldstone boundary $\widehat{(Goldstone)}$. These massive particles strongly interact between them and give a massive bound state. In the same picture it is represented also a nonlinear quantum propagator $V'$ that gives the massive quantum graviton $G'$ as a bound state of two massless quantum gravitons $G$. (Similar representations can be obtained by considering observed nonlinear quantum propagators, namely nonlinear quantum propagators inside $\widehat{(YM)}[i]$.) Also $V'$ identifies two massive particles $c$ and $d$ that strongly interact generating the bound state $G'$. Therefore $G'$ can be considered a bound state of photons, but also a bound state of gravitons.}
\label{entanglement-photons-massive-bound-photons}
\end{figure}

\begin{example}[Quantum chimeras - Exotic quantum dynamic Standard Model extension]\label{exotic-quntum-dynamic-standard-model-extension}
Here we shall prove that by using entanglement phenomena allowed by (exotic) nonlinear quantum propagators, we can go beyond the Standard Model, by showing dynamic relations between the quantum particles considered in the Standard Model, that this framework does not permit.

Let us define {\em quantum chimera} a massless, neutral, white quantum fermion having spin $s=\frac{1}{2}$. The existence of quantum chimera is assured by the geometric structure of $\widehat{(YM)}$. In fact, in $\widehat{(YM)}$ we can distinguish the subsets reported in {\em(\ref{sub-equations-yang-mills})}

\begin{equation}\label{sub-equations-yang-mills}
  \left\{
  \begin{array}{l}
   \hbox{\rm Quantum-ghost-super Yang-Mills PDEs:}\\
   \widehat{(YM)}_{ghost}=\complement(\overline{\widehat{(Higgs)}})\bigcap \complement(\overline{\widehat{(Higgs)}_{w}})\subset \widehat{(YM)}\\
   \hbox{\rm Quantum-boson-ghost-super Yang-Mills PDEs:} \\
    \widehat{(YM)}_{\mathfrak{b}-ghost}=\widehat{(Boson)}\bigcap\widehat{(YM)}_{ghost}\subset \widehat{(YM)}\\
     \hbox{\rm Quantum-fermion-ghost-super Yang-Mills PDEs):} \\
     \widehat{(YM)}_{\mathfrak{f}-ghost}=\widehat{(Fermion)}\bigcap\widehat{(YM)}_{ghost}\subset \widehat{(YM)}\\
   \hbox{\rm Quantum-e-colour-ghost-super Yang-Mills PDEs:} \\
    \widehat{(YM)}_{e-colour-ghost}=\complement(\overline{\widehat{(Higgs)}})\bigcap \widehat{(Higgs)}_{w}\subset \widehat{(YM)}\\
      \hbox{\rm Quantum-boson-e-colour-ghost-super Yang-Mills PDEs:} \\
    \widehat{(YM)}_{\mathfrak{b}-e-colour-ghost}=\widehat{(YM)}_{e-colour-ghost}\bigcap \widehat{(Boson)}\subset \widehat{(YM)}\\
   \hbox{\rm Quantum-fermion-e-colour-ghost-super Yang-Mills PDEs:} \\
    \widehat{(YM)}_{\mathfrak{f}-e-colour-ghost}=\widehat{(YM)}_{e-colour-ghost}\bigcap \widehat{(Fermion)}\subset \widehat{(YM)}\\
  \hbox{\rm Quantum-e-colour-super Yang-Mills PDEs):} \\
     \widehat{(Higgs)}_{e-colour}=\widehat{(Higgs)}\bigcap\widehat{(Higgs)}_w\subset \widehat{(YM)}\\
    \hbox{\rm Quantum-booson-e-colour-super Yang-Mills PDEs):} \\
     \widehat{(Higgs)}_{\mathfrak{b}-e-colour}=\widehat{(Higgs)}_{e-colour}\bigcap\widehat{(Boson)}\subset \widehat{(YM)}\\
      \hbox{\rm Quantum-fermion-e-colour-super Yang-Mills PDEs):} \\
     \widehat{(Higgs)}_{\mathfrak{f}-e-colour}=\widehat{(Higgs)}_{e-colour}\bigcap\widehat{(Fermion)}\subset \widehat{(YM)}\\
\end{array}
  \right.
\end{equation}
Since these equations are obtained by intersection of open subsets of $\widehat{(YM)}$, they are all formally integrable and completely integrable sub-equations of $\widehat{(YM)}$. Therefore, for any point of these equations pass solutions that are contained therein. For example for any $q\in  \widehat{(YM)}_{ghost}$, there are solutions $V\subset\widehat{(YM)}_{ghost}$, such that $q\in V$. Transversal sections of such solutions represent massless, neutral, white quantum particles. Therefore a quantum chimera is a transversal section of a solution $V\subset \widehat{(YM)}_{\mathfrak{f}-ghost}$, having quantum spin with $s=\frac{1}{2}$.\footnote{Remark that intersections listed in (\ref{sub-equations-yang-mills}) cannot be empty sets for construction.
In Tab. \ref{prototype-quantum-particles} are reported some prototype quantum particles belonging to the sub-PDEs listed in (\ref{sub-equations-yang-mills}).}

\begin{table}[h]
\caption{Remarkable opens sub-PDEs $X\subset \widehat{(YM)}$, and prototype quantum particles belonging therein.}
\label{prototype-quantum-particles}
\scalebox{0.8}{$\begin{tabular}{|c|l|}
\hline
{\footnotesize\rm $X\subset\widehat{(YM)}$}&{\footnotesize\rm Prototype quantum particles}\\
\hline \hline
\hfil{\footnotesize\rm $\widehat{(YM)}_{ghost}$}\hfil&{\footnotesize\rm photon ($\gamma$), graviton ($G$), neutrino ($\nu$), chimera ($c$)}\\
\hline
\hfil{\footnotesize\rm $\widehat{(YM)}_{\mathfrak{b}-ghost}$}\hfil&{\footnotesize\rm $\gamma$, $G$}\\
\hline
\hfil{\footnotesize\rm $\widehat{(YM)}_{\mathfrak{f}-ghost}$}\hfil&{\footnotesize\rm $\nu$, $c$}\\
\hline
\hfil{\footnotesize\rm $\widehat{(YM)}_{e-colour-ghost}$}\hfil&{\footnotesize\rm gluon ($g$), colour-graviton (${}_cG$), colour-chimera (${}_cc$)}\\
\hline
\hfil{\footnotesize\rm $\widehat{(YM)}_{\mathfrak{b}-e-colour-ghost}$}\hfil&{\footnotesize\rm $g$, ${}_cG$}\\
\hline
\hfil{\footnotesize\rm $\widehat{(YM)}_{\mathfrak{f}-e-colour-ghost}$}\hfil&{\footnotesize\rm ${}_cc$}\\
\hline
\hfil{\footnotesize\rm $\widehat{(Higgs)}_{e-colour}$}\hfil&{\footnotesize\rm quark ($q$), massive-colour-graviton (${}_cG'$)}\\
\hline
\hfil{\footnotesize\rm $\widehat{(Higgs)}_{\mathfrak{b}-e-colour}$}\hfil&{\footnotesize\rm massive-colour-graviton (${}_cG'$)}\\
\hline
\hfil{\footnotesize\rm $\widehat{(Higgs)}_{\mathfrak{f}-e-colour}$}\hfil&{\footnotesize\rm $q$, massive-colour-chimera (${}_cc'$)}\\
\hline
\multicolumn {2}{l}{\footnotesize\rm Colour-graviton: ${}_cG$: quantum bound state of two gluons with collective spin $s=2$.}\\
\multicolumn {2}{l}{\footnotesize\rm Massive colour-graviton: ${}_cG'$: massive quantum bound state of two gluons with collective spin $s=2$.}\\ \end{tabular}$}
\end{table}

$\bullet$\hskip 3pt Massless, neutral, white quantum bosons, or quantum fermions, of any quantum spin $\hat s$, can be dynamically obtained as bound states of a suitable number of quantum chimeras. In fact, let us consider a bound state $c$ of two quantum chimeras $a$ and $b$ encoded by a nonlinear quantum propagator $V\subset \widehat{(YM)}_{ghost}$ such that $\partial V=(a\bigcup b)\bigcup P\bigcup c$, with $\partial P=(\partial a\bigcup \partial b)\bigcup \partial c$. Then $c$ is a quantum particle having quantum spin with $s=\frac{1}{2}+\frac{1}{2}=1$, or $s=1-1=0$. Therefore, $c$ can be a quantum boson with spin $s=1$ or $s=0$. Furthermore, by considering a bound state $e$ made by $c$ and another quantum chimera $d$, we can obtain a quantum particle with spin $s=1+\frac{1}{2}=\frac{3}{2}$, or $s=\frac{3}{2}-1=\frac{1}{2}$, or $s=0+\frac{1}{2}=\frac{1}{2}$. Therefore, $e$ can be a quantum fermion with spin $s=\frac{3}{2}$ or $s=\frac{1}{2}$.
By continuing in this way we can have bound states in $\widehat{(YM)}_{ghost}$ with spin $s=\frac{3}{2}+\frac{1}{2}=2$, or $s=2-1=1$, or $s=1-1=0$ and $s=\frac{1}{2}+\frac{1}{2}=1$, or $s=1-1=0$. Thus we can again have a quantum boson with spin $s=2$ or $s=1$ or $s=0$. This process can continue and we can dynamically obtain quantum fermions of any fixed quantum spin $s=\frac{r}{2}$, $r\nmid 2$, $r\in\mathbb{N}$.

Furthermore, by using the Goldbach's conjecture (proved in \cite{PRAS30}), we can state also that any integer $n\in\mathbb{N}$ can be represented in the following way: $n=\frac{p_1}{2}+\frac{p_1}{2}$, where $p_1$ and $p_2$ are two prime numbers. Therefore, taking into account above considerations on quantum bound states obtained starting from quantum chimeras, we can state:
Massless, neutral, white quantum bosons of any spin $s=n\in\mathbb{N}$, can be dynamically obtained as bound states of a suitable numbers of quantum chimeras.

\textbf{Warning}. Above results do not mean that all quantum fermions and all quantum bosons can be obtained as bound states of quantum chimeras ! Really the Pauli's exclusion principle implies that two quantum chimeras in a bound state cannot occupy the same energetic level. Therefore, a quantum boson of spin $s=1$, or $s=0$, is a bound state of two quantum chimeras whether its quantum energy has a fine structure. Furthermore, there are quantum fermions of spin $s=\frac{1}{2}$, (e.g., proton and neutron) whose spin cannot be algebraically calculated from the spin of its quarks components. There the collective effects by gluons become dominant.

$\bullet$\hskip 3pt Massless neutrinos of Standard Model can be considered as prototypes of quantum chimeras.

$\bullet$\hskip 3pt A quantum photon $\gamma$ can be dynamically obtained as a bound state of two quantum chimeras $a$ and $b$ if it is not monochromatic. More precisely, there exist a nonlinear quantum propagator $V\subset \widehat{(YM)}_{ghost}$, such that $\partial V=(a\bigcup b)\bigcup P\bigcup \gamma$, with $\partial P=(\partial a\bigcup \partial b)\bigcup\partial\gamma$, whether $\gamma$ has a fine structure.

$\bullet$\hskip 3pt Taking into account Example \ref{quantum-photons-entangled-with-massive-bound-photons} we can also state that the quantum bosons $Z^0$, $H^0$ and $G$ can be dynamically obtained as bound states of quantum chimeras whether it should be possible to prove that they have a fine structure.

$\bullet$\hskip 3pt Massive, electric-charged, colour quantum fermions, of any spin $s$, can be dynamically obtained as bound states of a suitable number of quantum chimeras. In fact, any nonlinear quantum propagator $V\subset\widehat{(YM)}$, such that $\partial V=a\bigcup P\bigcup b$, where $a\in \widehat{(YM)}_{\mathfrak{f}-ghost}$ and $b\in \widehat{(Higgs)}_{e-colour}\subset\widehat{(YM)}$, propagates the quantum massless, neutral, white quantum fermion $a$ into a quantum fermion with the same quantum spin of $a$ but having electric-charge, colour and mass-gap.

\textbf{Warning}. Such a nonlinear quantum propagator $V$, when observed, namely inside $\widehat{(YM)}[i]$, is necessarily exotic one, namely it does not respect the constraint equation $\widehat{(YM)}[i]_\blacksquare $, given in (\ref{sub-equation-zero-defect-quantum-electric-charge-a}).

$\bullet$\hskip 3pt For example, any quantum quark $q$ can be dynamically obtained as a transversal section of a nonlinear quantum propagator $V$ bording a quantum chimera $a\in\widehat{(YM)}_{\mathfrak{f}-ghost}$ with a quantum particle $b\subset \widehat{(Higgs)}_{e-colour}$, having spin $s=\frac{1}{2}$, and all the other quantum numbers of a quark.

$\bullet$\hskip 3pt Similarly one can obtain leptons, identified with suitable white quantum particles inside $\widehat{(Higgs)}_{\mathfrak{f}-e-colour}$.

$\bullet$\hskip 3pt Massless, colour-charged quantum bosons can be dynamically obtained as bound states of a suitable number of quantum chimeras.
In fact, suitable nonlinear quantum propagators $V$ such that $\partial V=a\bigcup P\bigcup b$, where $a\in  \widehat{(YM)}_{\mathfrak{b}-ghost}$ and $b\in  \widehat{(YM)}_{\mathfrak{b}-colour-ghost}$, propagate $a$ into a neutral, massless, colour quantum boson.

\textbf{Warning}. Such a nonlinear quantum propagator $V$, when observed, namely inside $\widehat{(YM)}[i]$, must be necessarily exotic, namely it does not respect the constraint equation $\widehat{(YM)}[i]_\blacksquare $, given in (\ref{sub-equation-zero-defect-quantum-electric-charge-a}).

$\bullet$\hskip 3pt Whether quantum gluons $g$ have a fine structure, they can be dynamically obtained as a transversal section of an exotic nonlinear quantum propagator $V$ bording a quantum bosonic bound state of quantum chimeras, belonging to $\widehat{(YM)}_{\mathfrak{b}-ghost}$, and with spin $s=1$, with a suitable massless neutral, colour spin $s=1$ quantum particle inside $\widehat{(YM)}_{\mathfrak{b}-colour-ghost}$, having just all the quantum numbers of gluons.

Taking into account Example \ref{quantum-photons-entangled-with-massive-bound-photons}, we can also state that the quantum electric charged bosons $W^{\pm}$ can be dynamically obtained as bound states of quantum chimeras whether they have a fine structure (that at the moment is unknown).

$\bullet$\hskip 3pt Let us emphasize that exotic nonlinear quantum propagators starting from $\widehat{(YM)}_{\mathfrak{b}-ghost}$ and ending into $\widehat{(YM)}_{\mathfrak{b}-colour-ghost}$ can encode gluons starting fron photons.

$\bullet$\hskip 3pt Taking into account Example \ref{quantum-photons-entangled-with-massive-bound-photons}, we can also state that quantum graviton can be obtained as quantum bound states of two quantum gluons.

As just emphasized we can consider massless neutrinos of the Standard Model as examples of quantum chimeras. Therefore above results centralize the role played by neutrinos in quantum world. With this respect, let us recall that recent experimental results on interactions of neutrinos with matter, recognized neutrinos to have little masses. However, this simple fact is not sufficient for itself to state that the Standard Model must be modified to include massive neutrinos. In fact we have proved that massless particles acquire mass when interact with massive particles. Therefore, we can continue to consider massless neutrinos according to the Standard Model.
\end{example}

\begin{remark}[Quantum entanglement and quantum bound states]\label{quantum-entanglement-and-quantum-bound-states}
Let us emphasize that quantum entanglement, a phenomenon related to the algebraic-topologic structure of nonlinear quantum propagators, is related also to the concept of quantum bound state. For example the nonlinear quantum propagator encoding the beta decay of a free neutron, namely $n\to p+e^-+\overline{\nu}_e$, means that there proton, $p$, electron $e^-$, and antineutrino $\overline{\nu}_e$, are quantum entangled with $n$. On the other hand, a quantum bound state of $\{\overline{\nu}_e,p,e^-\}$, can have spin $s=\frac{3}{2}$ and $s=\frac{1}{2}$. Therefore we can consider neutron $n$ a quantum bound state of $\{\overline{\nu}_e,p,e^-\}$ with collective spin $s=\frac{1}{2}$.\footnote{Note that the electron capture of proton, namely the process $p+e^-\to n+\nu_e$, is just obtained from quantum crossing symmetry of the beta neutron decay. (See \cite{PRAS29}.) This supports the idea that neutron is a quantum bound state of $\{p,e^-,\overline{\nu}_e\}$.} This suggests us to foresee a quantum bound state of $\{\overline{\nu}_e,p,e^-\}$ with collective spin $s=\frac{3}{2}$, different from $n$, and that we denote $n_{\frac{3}{2}}$. This point of view can be generalized to other decays. For example, a bound states of $\{\overline{\nu}_{e},e^-\}$ can have $s=1$ and $s=0$. Then since one has the decay $W^-\to \overline{\nu}_{e}+e^-$, we can consider $W^-$ a quantum bound state of $\{\overline{\nu}_{e},e^-\}$, with a collective spin $s=1$. (Similarly we can foresee a quantum bound state $W^{-}_{0}$ of $\{\overline{\nu}_{e},e^-\}$, with a collective spin $s=0$.) This agrees with the neutron beta decay that in the Standard Model is considered in two steps, $n\to p+W^-$ and $W^-\to e^-+\overline{\nu}_e$. Therefore, the nonlinear quantum propagator encoding the neutron beta decay can have a structure as pictured in Fig. \ref{beta-neutron-decay-nonlinear-quantum-propagator} in Appendix B. Thus we can consider $n$ also a quantum bound state of $\{p,W^-\}$ with collective spin $s=\frac{1}{2}$.\footnote{A quantum bound state of $\{p,W^-\}$ can have a collective spin $s=\frac{1}{2}$, or $s=\frac{3}{2}$.} On the other hand, from the Standard Model we know that neutron is made by two quarks down, one quark up and gluons. Thus a quantum bound state cannot be identified with the set of particles that are quantum entangled with such a bound state.
\end{remark}
\begin{example}[How many gluons there are in a proton ?]\label{how-many-gluons-there-are-in-a-proton}
In the Standard Model it is assumed that a proton is made by two up quarks and one down quark, glued together by some gluons. But no news there are about the number of such gluons. Taking into account that the spin of a gluon is $s=1$, the spin of a proton is $s=\frac{1}{2}$, and that a quantum bound state of these quantum particles: $\{ u,u,d,n\cdot g\}_{n\in\mathbb{N}}$, has a collective spin $s=\frac{1}{2}$, iff there exists an integer $m$, such that $m=1+n$, we see that a proton can have any number $n$ of gluons with $n\ge 1$. Therefore, we can assume that there are many different excited states in quantum proton. Let us denote them by ${}_np$, $n\ge 1$. Which of these different excited states is realized in quantum reactions it depends from the specific reaction.\footnote{Recent experimental results have emphasized problems on the exact dimension of proton and about the gluon contribution on the proton spin. (See, e.g., the following links: \href{http://www.mpq.mpg.de/~rnp/wiki/pmwiki.php/Main/Publications}{J. C. Bernauer and R. Pohl, Scientific American, Feb. 2014}, \href{http://www.scientificamerican.com/article/proton-spin-mystery-gains-a-new-clue1/}{C. Moskowitz, Scientific American, July 2014} and \cite{FLORIAN-SASSOT-STRATMANN-VOGELSANG,NOCERA-BALL-FORTE-RIDOLFI-ROJO}.) Really these difficulties could be related to the different aspects of proton in the forms ${}_np$, $n\ge 1$.} In Fig. \ref{orbital-configurations-n-proton} we have represented possible orbital configurations of such ${}_np$. These representations have only a symbolic meaning, since the most realistic description of proton should be like a plasma of quarks and gluons. With this respect, let us emphasize that gluons entering in interaction with quarks must necessarily cross the Goldstone boundary, entering inside $\widehat{(Higgs)}$. This is enough to state that gluons inside proton are massive quantum particles. Probably the mass of such gluons should be very little, (like neutrino's mass in interaction with matter), hence their contribution to the proton's mass should be negligible,
\end{example}
\begin{figure}[h]
\includegraphics[width=4.3cm]{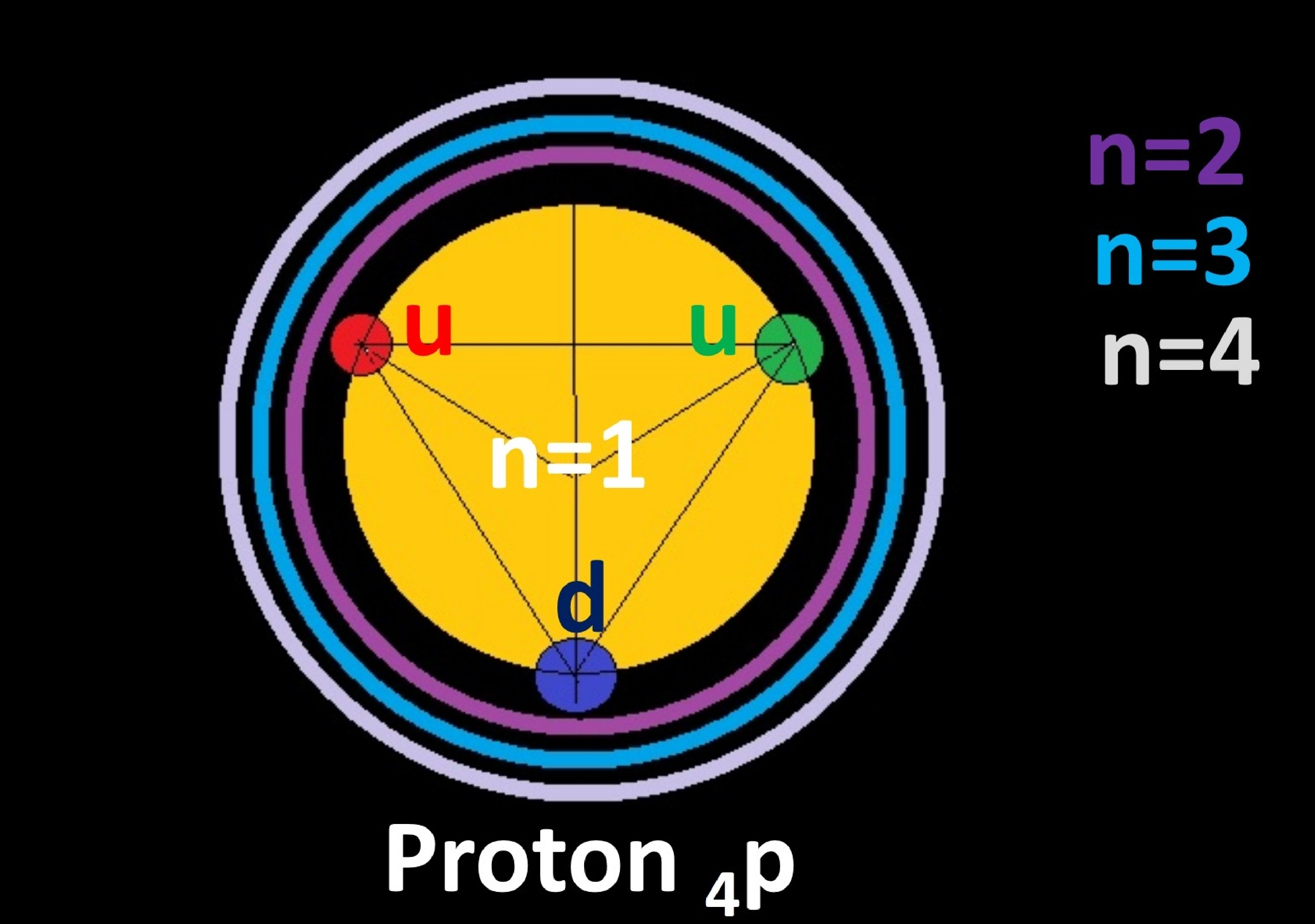}
\renewcommand{\figurename}{Fig.}
\caption{Representations of quantum orbitals of ${}_np$, for $n=1$, $n=2$, $n=3$ and $n=4$. There the central sphere symbolically denotes the first energy level of proton, namely with $n=1$ gluons. The other circled spheres denote the levels $n=2$, $n=3$, and $n=4$, of increasing dimension respectively. These representations have only a symbolic meaning, since the most realistic description of proton should be like a plasma of quarks and gluons.}
\label{orbital-configurations-n-proton}
\end{figure}

\begin{example}[There are gravitons in a proton ?]\label{gravitons-in-a-proton}
From Example \ref{exotic-quntum-dynamic-standard-model-extension} we can also state that quantum gravitons $G$ can be recognized inside protons ${}_np$, with $n\ge 2$. In fact, inside such protons two gluons ca be constrained to generate a quantum bound state with spin $s=2$, namely a massive-colour quantum graviton ${}_cG'$. In this sense a proton with many gluons can be considered a plasma of quarks, gluons and gravitons, as claimed in \cite{PRAS22}.
\end{example}

\begin{example}[Massive quantum gravitons, quantum Higgs particle and proton decays]\label{gravitons-proton-decays}
A possible quantum process where observe a massive quantum graviton is just in the rare (and not yet observed) proton decay $p\to e^++2\gamma$. Its half-life of the order of $10^{33}$ years, means that one could observe one decay for year in a sample containing $10^{33}$ protons. (For information about the actual status of research on this decay see, for example, the following link \href{http://en.wikipedia.org/wiki/Proton_decay}{Wikipedia-proton-decay}.) Taking into account the geometric structure of $\widehat{(YM)}$, we can state that nonlinear quantum propagators $V$ encoding such a decay, should identify a massive intermediate quantum particle, that could be identified with a quantum bound state of spin $s=2$ (massive quantum graviton) or $s=0$ (massive quantum Higgs particle).  In other words, $V=V_1\bigcup_{G'}V_2$, $\partial V_1=p\bigcup P_1\bigcup(e^+\bigcup G')$ and $\partial V_2=G'\bigcup P_2\bigcup(\gamma\bigcup \gamma)$, or $V'=V'_1\bigcup_{H^0}V'_2$, $\partial V'_1=p\bigcup P'_1\bigcup(e^+\bigcup H^0)$ and $\partial V'_2=H^0\bigcup P'_2\bigcup(\gamma\bigcup \gamma)$. (See Fig. \ref{proton-decay-and-quantum-graviton-higgs}.)
\end{example}

\begin{figure}[h]
\includegraphics[width=4.3cm]{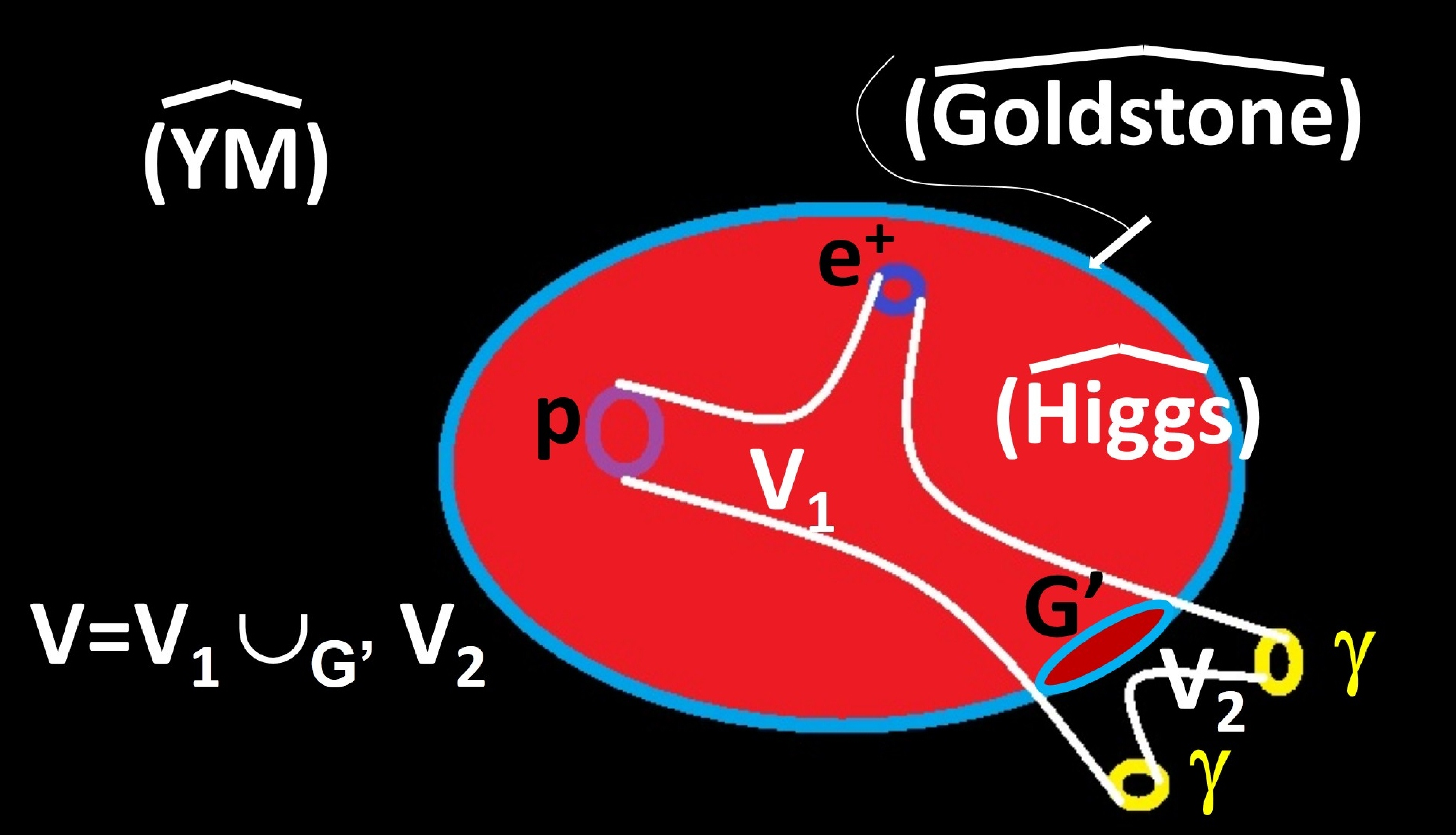}
\includegraphics[width=4.42cm]{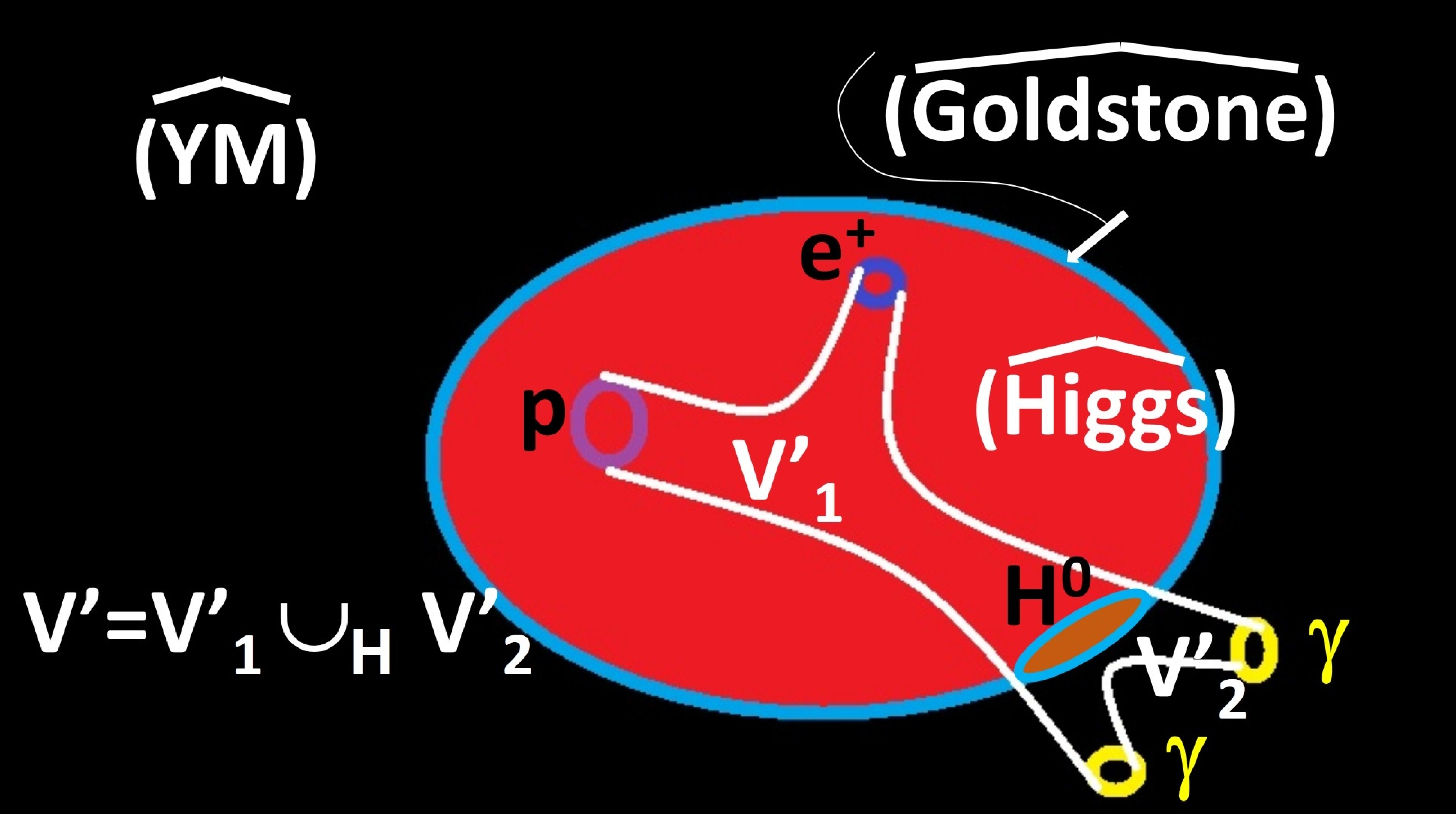}
\renewcommand{\figurename}{Fig.}
\caption{Representations of quantum proton decays, where $G'$ denotes a massive quantum graviton and $H^0$ is a massive quantum Higgs particle. More precisely $G'$ is considered a massive quantum bound state of $\{\gamma,\gamma\}$, with collective spin $s=2$, and $H^0$ is a massive quantum bound state of $\{\gamma,\gamma\}$, with collective spin $s=0$.}
\label{proton-decay-and-quantum-graviton-higgs}
\end{figure}

\section{\bf Quantum thermodynamic-exotic solutions in $\widehat{(YM)}[i]$}\label{sec-quantum-thermodynamic-exotic-solutions}

In this section we shall prove that $\widehat{(YM)}[i]$ admits solutions with negative local temperature. This is a quantum effect, classically strictly forbidden.\footnote{This exotic quantum phenomenon has been supported by some experimental confirmations. (See, e.g. \cite{CARR, PURCELL-POUND, SCHNEIDER}.) See also \cite{EDWARDS-TAYLOR}.}
We have the following theorem.

\begin{theorem}[Quantum thermodynamic-exotic solutions existence]\label{existence-quantum-thermodynamic-exotic-solutions}
$\widehat{(YM)}[i]$ admits {\em quantum thermodynamic-exotic solutions}, i.e., solutions that have negative temperature $\theta=(\partial s.e)$, where $e$ is the local interior energy and $s$ is the local entropy of such solutions respectively.
\end{theorem}

\begin{proof}
In Refs. \cite{PRAS22} we have proved that thermodynamics of the observed quantum super Yang-Mills PDE is encoded on the fiber bundle $\widetilde{W[i]}\equiv E[i]\times_NT^0_0N\cong
E[i]\times\mathbb{R}$, over $N$, whose
sections $(\widetilde{\mu},\beta)$ over $N$, represent an {\em observed
quantum fundamental field}, $\widetilde{\mu}$, and a function $\beta:N\to \mathbb{R}$,
{\em thermal function}. If $\beta=\frac{1}{\kappa_B\theta}$,
where $\kappa_B$ is the Boltzmann constant and $\theta$ is the
temperature, then the observed solution encodes a system in
equilibrium with a heat bath. The situation is resumed in Tab. \ref{local-thermodynamics-functions}.

\begin{table}[h]
\caption{Local thermodynamics functions of
{\boldmath$\scriptstyle
\widehat{(YM)}[i]$} solutions.}
\label{local-thermodynamics-functions}
\begin{tabular}{|l|l|l|l|}
\hline
\hfil{\rm{\footnotesize Name}}\hfil&\hfil{\rm{\footnotesize Definition}}\hfil &\hfil{\rm{\footnotesize Remark}}\hfil&\hfil{\rm{\footnotesize Order}}\hfil\\
\hline \hline
{\rm{\footnotesize partition function}}&{\rm{\footnotesize $Z=\TR(e^{-\beta\widetilde{H}})$}}&
{\rm{\footnotesize $Z=Z(\beta,\widetilde{\mu}^K_\alpha)$}}&\hfil{\rm{\footnotesize $1$}}\hfil\\
\hline
{\rm{\footnotesize interior energy}}&{\rm{\footnotesize $e=-(\partial\beta.\ln Z)$}}&{\rm{\footnotesize
$e=\kappa_B\, \theta^2\, (\partial\theta.\ln Z)$}}&\hfil{\rm{\footnotesize $1$}}\hfil\\
\hline
{\rm{\footnotesize fluctuation interior energy}}&{\rm{\footnotesize $<(\triangle E)^2>=<(E-e)^2>$}}&{\rm{\footnotesize
$<(\triangle E)^2>=(\partial\beta\partial \beta.\ln Z)$}}&\hfil{\rm{\footnotesize $2$}}\hfil\\
\hline
{\rm{\footnotesize entropy}}&{\rm{\footnotesize $s=\kappa_B\, (\ln Z+\beta e)$}}&{\rm{\footnotesize
$s=(\partial\theta.(\kappa_B\, \ln Z))$}}&\hfil{\rm{\footnotesize $1$}}\hfil\\
\hline
{\rm{\footnotesize free energy}}&{\rm{\footnotesize $f=e-\theta s$}}&{\rm{\footnotesize $f=-\kappa_B\, \theta\ln Z$}}&\hfil{\rm{\footnotesize $1$}}\hfil\\
\hline
\multicolumn {3}{l}{\rm{\footnotesize It is assumed a Lagrangian of first derivation order.}}\\
\end{tabular}
\end{table}
We shall use the following lemma.
\begin{lemma}[Thermodynamics covering of {$\widehat{(YM)}[i]$}]\label{thermodynamics-covering-ym}
Thermodynamics for solutions of $\widehat{(YM)}[i]$ is encoded by a $1$-dimensional differential covering of $\widehat{(YM)}[i]$.
\end{lemma}

\begin{proof}
Let us consider the infinite prolongation $\widehat{(YM)}[i]_{+\infty}\subset J\hat D^{\infty}(E[i])$ of $\widehat{(YM)}[i]$. This is endowed with a $4$-dimensional (completely integrable) distribution $\mathbf{E}_{\infty}[i]\subset T\widehat{(YM)}[i]_{+\infty}$ ({\em Cartan distribution}). A $r$-dimensional differential covering of $\widehat{(YM)}[i]_{+\infty}$ is a bundle, $\tau:\widetilde{\mathcal{E}}\to \widehat{(YM)}[i]_{+\infty}$, of finite rank $r$, such that on $\widetilde{\mathcal{E}}$ is defined a distribution $\mathbf{E}_{\tau}\subset T\widetilde{\mathcal{E}}$, such that $[\mathbf{E}_{\tau},\mathbf{E}_{\tau}]\subset\mathbf{E}_{\tau}$ and $\tau_{*}=T(\tau)|_{\mathbf{E}_\tau}:\mathbf{E}_\tau\to\mathbf{E}_{\infty}[i]$ induces an isomorphism between the corresponding fibers, i.e., one has the exact and commutative diagram (\ref{exact-commutative-diagram-covering}).
\begin{equation}\label{exact-commutative-diagram-covering}
\xymatrix{0\ar[r]&\mathbf{E}_{\tau}\ar[dr]\ar[rr]^{\tau_*}&&\tau^*\mathbf{E}_{\infty}[i]\ar[dl]\ar[r]&0\\
&&\widetilde{\mathcal{E}}\ar[d]&&\\
&&0&&\\}
\end{equation}
If the distribution $\mathbf{E}_{\infty}[i]$ is locally spanned by vector fields $\zeta_\alpha$, $\alpha\in\{0,1,2,3\}$, then there are vector fields $\widetilde{\zeta}_\alpha$, $\alpha\in\{0,1,2,3\}$, on $\widetilde{\mathcal{E}}$ such that $\tau_*(\widetilde{\zeta}_\alpha)=\zeta_\alpha$ and $[\widetilde{\zeta}_\alpha,\widetilde{\zeta}_\beta]=0$, $\forall\alpha,\, \beta\in\{0,1,2,3\}$. If $\nu$ is a gauge infinitesimal symmetry of the covering $\tau:\widetilde{\mathcal{E}}\to\widehat{(YM)}[i]_{+\infty}$, i.e., a $\tau-$vertical vector field $\nu:\widetilde{\mathcal{E}}\to T\widetilde{\mathcal{E}}$, such that $[\nu,\mathbf{E}_\tau]\subset \mathbf{E}_\tau$, then one has $[\nu,\widetilde{\zeta}_\alpha]=0$, $\forall\alpha\in\{0,1,2,3\}$. Let $\{w^j\}_{1\le j\le r}$ be local vertical coordinates of $\tau:\widetilde{\mathcal{E}}\to\widehat{(YM)}[i]_{+\infty}$. Then the local characterization of $\widetilde{\zeta}_\alpha$ is given in (\ref{local-characterization-covering-distribution}).

\begin{equation}\label{local-characterization-covering-distribution}
  \widetilde{\zeta}_\alpha=\zeta_\alpha+\nu_\alpha\, \left\{
  \begin{array}{ll}
  \hbox{\rm(a)}&\nu_\alpha=\sum_{1\le j\le r}a_\alpha^j\, \partial w_j\\
  &\\
\hbox{\rm(b)}&{[\nu_\alpha,\nu_\beta]}+\zeta_\alpha\nu_\beta-\zeta_\beta\nu_\alpha=0.\\
  \end{array}
  \right.
\end{equation}
(When $\widetilde{\zeta}_\alpha=\zeta_\alpha$ the covering is called a {\em trivial covering}.)

The local expression of $\widetilde{\mathcal{E}}$ is given by the set of equations reported in (\ref{local-expression-covering-pde}).\footnote{In (\ref{local-expression-covering-pde}) and (\ref{thermal quantum-super-yang-mills-pde}), $\widetilde{E}^{K\omega}_{(\gamma)}$ denotes all the prolongations of $\widetilde{E}^{K\omega}$.}
\begin{equation}\label{local-expression-covering-pde}
  \left\{
  \begin{array}{l}
    \hbox{\rm($\widehat{(YM)}[i]_{+\infty}$)}\, :\hskip 3pt \left\{
                                                       \begin{array}{l}
                                                          \widetilde{E}^{K\beta}\equiv(\partial_{\alpha}.\tilde R^{K\alpha\beta})+[\widehat C^K_{IJ}\tilde\mu^I_{\alpha},\tilde R^{J\alpha\beta}]_+=0\\
                                                        \tilde
R^K_{\alpha_1\alpha_2}=(\partial\xi_{[\alpha_1}.\tilde\mu^K_{\alpha_2]})+
\frac{1}{2}\widehat{C}{}^K_{IJ}\tilde\mu^I_{[\alpha_2}\tilde\mu^J_{\alpha_1]}\\
\widetilde{E}^{K\beta}_{(\gamma)}=0,\, (|\gamma|>0)\\
                                                       \end{array}
                                                       \right\}\\
                                                       \\
    (\partial\xi_\alpha.w^j) =a^j_\alpha(\xi^\beta,w^i,\mu^I_{\beta\sigma}).\\
  \end{array}
  \right.
\end{equation}
This overdetermined system is consistent iff condition (\ref{local-characterization-covering-distribution})(b) holds on $\widehat{(YM)}[i]_{+\infty}$.

From results in \cite{PRAS22}, \cite{PRAS29} and above considerations, it follows that thermodynamics of $\widehat{(YM)}[i]$ is encoded by $\widehat{(YM)}[i]\times\mathbb{R}$ that identifies a $1$-dimensional covering $\widehat{(YM)}[i]_{+\infty}\times\mathbb{R}$ of $\widehat{(YM)}[i]_{+\infty}$. Its local expression is reported in (\ref{thermal quantum-super-yang-mills-pde}).

\begin{equation}\label{thermal quantum-super-yang-mills-pde}
  \left\{
  \begin{array}{l}
    \hbox{\rm($\widehat{(YM)}[i]_{+\infty}$)}\, :\hskip 3pt \left\{
                                                       \begin{array}{l}
                                                          \widetilde{E}^{K\omega}\equiv(\partial_{\alpha}.\tilde R^{K\alpha\omega})+[\widehat C^K_{IJ}\tilde\mu^I_{\alpha},\tilde R^{J\alpha\omega}]_+=0\\
                                                        \tilde
R^K_{\alpha_1\alpha_2}=(\partial\xi_{[\alpha_1}.\tilde\mu^K_{\alpha_2]})+
\frac{1}{2}\widehat{C}{}^K_{IJ}\tilde\mu^I_{[\alpha_2}\tilde\mu^J_{\alpha_1]}\\
                                                      \widetilde{E}^{K\omega}_{(\gamma)}=0,\, (|\gamma|>0)\\
                                                        \end{array}\right\}\\
                                                        \\
(\partial\xi_\alpha.\beta) =a_\alpha(\xi^\omega,\beta,\mu^I_{\omega\sigma})\\
    {[\nu_\alpha,\nu_\omega]}+\zeta_\alpha\nu_\omega-\zeta_\omega\nu_\alpha=0\\
    \nu_\alpha=a_\alpha\, \partial \beta.\\
  \end{array}
  \right.
\end{equation}

 Let us now emphasize that there is a canonical $1$-dimensional covering of $\widehat{(YM)}[i]$, that we will denote by  ${}^{\Theta}\widehat{(YM)}[i]$ and that we call {\em thermal quantum super Yang-Mills} PDE. This equation is defined by the commutative diagram (\ref{thermal-quantum-super-yang-mills-commutative-diagram}). There the top horizontal line is exact too.

 \begin{equation}\label{thermal-quantum-super-yang-mills-commutative-diagram}
    \scalebox{0.6}{$\xymatrix{&0&\widehat{(YM)}[i]_{+\infty}\ar[l]\ar@{=}[d]&{}^{\Theta}\widehat{(YM)}[i]\ar[l]\ar@{=}[d]&&\\
    &J\hat D^{\infty}(E[i])\ar[d]&\widehat{(YM)}[i]_{+\infty}\ar@{_{(}->}[l]\ar[d]&\widehat{(YM)}[i]_{+\infty}\times_NJ\hat D^{\infty}(T^0_0N)\ar[l]\ar[d]\ar@{^{(}->}[r]&J\hat D^{\infty}(\widetilde{W[i]})\ar[d]\ar@{=}[r]^(0.35)\sim&J\hat D^{\infty}(E[i])\times_NJ\hat D^{\infty}(T^0_0N)\ar[d]\\
    &J\hat D^{2}(E[i])\ar[d]&\widehat{(YM)}[i]\ar@{_{(}->}[l]\ar[d]&\ar[l]\widehat{(YM)}[i]\times_NJ\hat D^{2}(T^0_0N)\ar[d]\ar@{^{(}->}[r]&J\hat D^{2}(\widetilde{W[i]})\ar[d]\ar@{=}[r]^(0.35)\sim&J\hat D^{2}(E[i])\times_NJ\hat D^{2}(T^0_0N)\ar[d]\\
     0&E[i]\ar[l]\ar[d]\ar@{=}[r]&E[i]\ar[d]&E[i]\times_NT^0_0N\ar[l]\ar[d]\ar@{=}[r]&\widetilde{W[i]}\ar[d]\ar@{=}[r]&\widetilde{W[i]}\ar[d]\\
     &N\ar[d]\ar@{=}[r]&N\ar[d]\ar@{=}[r]&N\ar[d]\ar@{=}[r]&N\ar[d]\ar@{=}[r]&N\ar[d]\\
    &0&0&0&0&0\\}$}
 \end{equation}

\end{proof}

Now, temperature identifies a continuous function $\theta:{}^{\Theta}\widehat{(YM)}[i]\to\mathbb{R}$. Set
\begin{equation}\label{exotic-quantum-thermodynamic-ym-pde}
  {}^{\bigcirc}\widehat{(YM)}[i]=\theta^{-1}(\mathbb{R}^{-})\subset {}^{\Theta}\widehat{(YM)}[i]
\end{equation}
where $\mathbb{R}^{-}=\{\lambda\in\mathbb{R}\, |\, \lambda<0\}\subset\mathbb{R}$ is an open set in $\mathbb{R}$. We call ${}^{\bigcirc}\widehat{(YM)}[i]$ the {\em exotic-thermal quantum super Yang-Mills} PDE. It is also formally integrable and completely integrable, since ${}^{\Theta}\widehat{(YM)}[i]$ is so and ${}^{\bigcirc}\widehat{(YM)}[i]$ is open therein. Therefore, for any initial condition $q\in {}^{\bigcirc}\widehat{(YM)}[i]$ passes some solution of ${}^{\bigcirc}\widehat{(YM)}[i]$. We call {\em quantum thermal-exotic solutions} such solutions. Since one has a natural projection ${}^{\Theta}\pi:{}^{\Theta}\widehat{(YM)}[i]\to \widehat{(YM)}[i]$, we get that also a natural sub-equation ${}^{\circ}\widehat{(YM)}[i]={}^{\Theta}\pi({}^{\bigcirc}\widehat{(YM)}[i])\subset \widehat{(YM)}[i]$. We call {\em exotic-thermodynamics quantum super Yang-Mills} PDE ${}^{\circ}\widehat{(YM)}[i]$. Therefore a quantum thermal-exotic solutions $V\subset {}^{\bigcirc}\widehat{(YM)}[i]$ identifies a solution $V'={}^{\Theta}\pi(V)\subset {}^{\circ}\widehat{(YM)}[i]\subset \widehat{(YM)}[i]$, that we call {\em quantum thermodynamic-exotic solution} in $\widehat{(YM)}[i]$. From above considerations, it follows that the set of such solutions is not empty. Therefore, theorem is done.
\end{proof}

\begin{example}\label{non-equilibrium-thermodynamic-processes}
Negative (absolute) temperature is related to the quantum structure of the physical system, encoded by its local observed quantum Hamiltonian $\widetilde{H}(p)$. As the local partition function $Z$ can be interpreted as a normalization factor for the local probability density $P(E)=\frac{1}{Z}N(E)e^{-\beta E}$, \cite{PRAS22, PRAS28, PRAS29}, it follows that must be $0<Z<\infty$. Since $Z=\TR(e^{-\beta\widetilde{H}(p)})$, the trace converges in general when $\beta\widetilde{H}(p)$ is positive semidefinite. Therefore, if $\widetilde{H}(p)$ is negative semidefinite, then $\beta$ must be negative. So, when $\theta=\frac{1}{\kappa_B\beta}$, we must necessarily have $\theta<0$. On the other hand if the fundamental quantum algebra $A$ is a $C^*$-algebra, and $\widetilde{H}(p)$ is a self adjoint element of $A$, we can use its {\em Jordan decomposition} to split it into a linear combination of positive elements of $A$. More precisely we can write $\widetilde{H}(p)=\widetilde{H}(p)_+-\widetilde{H}(p)_-$, where $\widetilde{H}(p)_+=(|\widetilde{H}(p)|+\widetilde{H}(p))/2$, $\widetilde{H}(p)_-=(|\widetilde{H}(p)|-\widetilde{H}(p))/2$ and $|\widetilde{H}(p)|=\widetilde{H}(p)_++\widetilde{H}(p)_-$. Since $[\widetilde{H}(p)_+,\widetilde{H}(p)_-]=0$, we can write $Sp(\widetilde{H}(p))=Sp(\widetilde{H}(p)_+)\oplus Sp(-\widetilde{H}(p)_-)$.\footnote{Let us recall that for two commuting elements $a,\, b\in A$ of a Banach algebra (and therefore also for a $C^*$-algebra) $A$, one has the following properties relating their spectra: (a) $Sp(a+b)=Sp(a)\oplus Sp(a)$; (b) $Sp(a.b)\subset Sp(a).Sp(b)$; (c) $r(a+b)\le r(a)+r(b)$; (d) $r(a).r(b)\le r(a).r(b)$; (e) $r(1)=1$; (f) $r(\lambda a)=|\lambda|r(a)$. ($r(a)$ denotes the spectral radius of $a\in A$.)} Therefore, we get
\begin{equation*}
    \begin{array}{ll}
     Z(\beta)&=\int_{Sp(\widetilde{H}(p))}N(E)e^{-\beta E}dE\\
     &=\int_{Sp(\widetilde{H}(p)_+)}N(E)e^{-\beta E}dE-\int_{Sp(\widetilde{H}(p)_-)}N(E)e^{-\beta E}dE\\
     &=\TR(e^{-\beta \widetilde{H}(p)_+})-\TR(e^{-\beta \widetilde{H}(p)_-}).
    \end{array}
\end{equation*}
For the convergence of traces it is necessary that $\beta \widetilde{H}(p)_{\pm}$ should be semidefinte positive. Since both $\widetilde{H}(p)_+$ and $\widetilde{H}(p)_-$ are positive, it follows that must be $\beta\ge 0$. Thus if the system is in equilibrium with a heat bath, namely $\theta=\frac{1}{\kappa_B\beta}$, for the temperature one should have $\theta\ge 0$. Therefore, in such cases quantum thermodynamic-exotic solutions cannot be obtained in thermodynamic equilibrium states, but only as non-equilibrium thermodynamic states.
\end{example}

\begin{example}\label{laser-example}
We can apply Theorem \ref{existence-quantum-thermodynamic-exotic-solutions} to laser systems. In fact, a laser system works just thanks to existence of quantum thermodynamic-exotic solutions.\footnote{For general informations on such systems see e.g., \href{http://en.wikipedia.org/wiki/Negative_temperature}{\tt Wikipedia/Negative-temperature} and references quoted therein.}
\end{example}
We can generalize above definition of quantum thermodynamic-exotic solution by the following one.
\begin{definition}\label{quantum-thermodynamic-exotic-solution}
We call {\em quantum thermodynamic-exotic solution} in $\widehat{(YM)}[i]$ a solution $V$ such that $V\bigcap {}^{\circ}\widehat{(YM)}[i]\equiv {}^{\circ}V\not=\varnothing$. Then we call also ${}^{\circ}V\subset V$ the {\em quantum thermodynamic-exotic component} of $V$.
\end{definition}

\begin{theorem}[Quantum entanglement and quantum thermodynamic-exotic solutions]\label{quantum-entanglement-thermodynamic-exotic-solutions}
$\widehat{(YM)}[i]$ admits solutions that support quantum entanglement and quantum thermodynamic-exotic solutions.
\end{theorem}

\begin{proof}
Let $V\subset \widehat{(YM)}[i]$ be a nonlinear quantum propagator such that $\partial V=N_0\sqcup P\sqcup N_1$, with $N_0\subset {}^{\circ}\widehat{(YM)}[i]$, beside a neighborhood ${}^{\circ}V$ of $V$. Then, from Section \ref{sec-quantum-entanglement-in-quantum-nonlinear-propagators} we know that $N_0$ and $N_1$ are quantum entangled particles, and $V$ is a quantum exotic-thermodynamic solution of $\widehat{(YM)}[i]$, according to Definition \ref{quantum-thermodynamic-exotic-solution}. Now the question is the following. Do exist such solutions ?
Let $N_0\subset \widehat{(YM)}[i]$ and $N_1\subset \widehat{(YM)}[i]$ be Cauchy data such that the following conditions are verified. (i) There exists an integral quantum super manifold $P\subset \widehat{(YM)}[i]$ such that $\partial P=\partial N_0\sqcup\partial N_1$. (In other words $\partial N_0$ and $\partial N_1$ belong to the same integral bordism class of $\widehat{(YM)}[i]\subset \hat J^2_4(E[i])$.)
(ii) $<\alpha,N_0>+<\alpha,N_1> =-<\alpha,P>$, for any quantum conservation law $\alpha$ of $\widehat{(YM)}$.
Then, according to some previous results contained in Part II, there exists a solution $V\subset \widehat{(YM)}[i]$ such that $\partial V=N_0\sqcup P\sqcup N_1$. (For details see \cite{PRAS19}(II).) On the other hand temperature is not a conservation law for $\widehat{(YM)}[i]$, hence we can satisfy above condition (i) with $N_0\subset {}^{\circ}\widehat{(YM)}[i]$ and $N_1\subset \widehat{(YM)}[i]\setminus{}^{\circ}\widehat{(YM)}[i]$.  Therefore the proof is done.
\end{proof}

\begin{example}[Asymmetry matter-antimatter and non-equilibrium thermodynamic processes in a quantum universe encoded by a quantum thermodynamic-exotic solution in {$\widehat{(YM)}[i]$}]\label{asymmetry-matter-antimatter-in-quantum-universe}
 A fashionable research subject is to justify the experimental observations that in our Universe the concentration of antimatter is substantially zero. Even if antiparticles are produced everywhere in the Universe in high-energy collisions, these antiparticle jets are annihilated by contact with matter, hence their presence in the actual Universe remains very rare.\footnote{It is important to underline that, starting from 1965, antinuclei and antiatoms are artificially produced in some laboratories (CERN, Brookhaven National Laboratory, Fermilab).} This could be a very strange situation, taking into account that also the original big-bang should create from nothing the actual Universe. In fact, by using Einstein's General Relativity, and extrapolating the expansion of the Universe backwards in time, one arrives to an infinite density and temperature at a finite time in the past. This singularity can be removed by considering that at microscopic level one should go beyond General Relativity and use quantum-(super)gravity. Really the time of the big-bang can be called the {\em Planck-epoch}. But the classic quantum mechanics should foresee a couple universe-antiuniverse, hence the presence of particles and antiparticles, in order to conserve zero vacuum quantum numbers. With this respect some authors have conjectured that particle interactions could violate some classic conservation laws in particle physics, as baryon-number, C invariance and CP symmetry and could present breakdown of chemical equilibrium during Planck-epoch. (See, e.g., \cite{SAKHAROV}.) Our theory of quantum-supergravity can, now, give a precise mathematical justification to the observed asymmetry matter-antimatter and to non-equilibrium processes. In fact, in order to describe the Universe in the Planck-epoch one can use a quantum super Yang-Mills PDE, $\widehat{(YM)}$. There, the Universe can be encoded as a quantum black-hole, hence with a very huge mass. As such it should be represented by a Cauchy data $B$, in the sub-equation $\widehat{(Higgs)}\subset\widehat{(YM)}$, where live solutions with quantum mass. However the solution $V$, encoding such a quantum universe, should start from a point $\bigstar$ in $\widehat{(YM)}$, outside $\widehat{(Higgs)}$, hence without quantum-mass. Then during its evolution such a quantum universe, should be represented by a nonlinear quantum propagator $V$, that, after crossed the {\em Goldstone boundary}, acquires a quantum mass. There one can talk about matter and antimatter. Since in the Planck-epoch there is not necessarily request the conservation of charge, it follows that it does neither necessitate the conservation of the symmetry between matter and antimatter. Therefore, the nonlinear quantum propagator $V$, encoding such a quantum universe, should be a $q$-exotic nonlinear quantum propagator $V$, in the sense of Theorem \ref{non-conservation-of-quantum-electric-charge}, Definition \ref{exotic-nonlinear-quantum-propagator} and Theorem \ref{properties-exoticnonlinear-quantum-propagators}. Furthermore, assuming also that $V$ is contained into ${}^{\circ}\widehat{(YM)}[i]$, it should be also a quantum thermodynamic-exotic solution of $\widehat{(YM)}[i]$. Therefore, in such a quantum universe interactions do not necessarily respect chemical equilibrium. (See Theorem \ref{existence-quantum-thermodynamic-exotic-solutions}, Example \ref{non-equilibrium-thermodynamic-processes}, Example \ref{laser-example} and Definition \ref{quantum-thermodynamic-exotic-solution}.) In order to fix ideas, we give in Fig. \ref{figure-planck-epoch-universe} a picture representation of the singular nonlinear quantum propagator $V$ that starting from the point $\bigstar$, origin of the big-bang, bords with $B$, the Cauchy data representing a huge-massive quantum black-hole, after crossed the Goldstone boundary. There $V$ identifies a quantum massive-universe $B_H$, that we call {\em Higgs-universe}. $V$ bords $B_H$ with $B$ inside $\widehat{(Higgs)}$. Furthermore, let us yet denote by $V$ the projection of $V\subset\widehat{(YM)}$, into $\widehat{(YM)}[i]$, and by $\widehat{(Higgs)}[i]$ the projection of $\widehat{(Higgs)}$ into $\widehat{(YM)}[i]$.\footnote{For details on the projection of $\widehat{(YM)}$ on $\widehat{(YM)}[i]$, see diagram (2) and \cite{PRAS19}(II), where has explicitly considered the relation between a quantum super PDE and its observed with respect to a quantum relativistic frame.} Then, in order $V$ should bord $\bigstar$ with $B_H$ it is necessary that a quantum thermodynamic-exotic component should be present in $V$, in the sense of Definition \ref{quantum-thermodynamic-exotic-solution}. Furthermore, for the part of $V$ inside $\widehat{(Higgs)}[i]$ one could have  negative temperatures in non-equilibrium thermodynamic processes. Therefore, also in this case one should have  $V\bigcap{}^{\circ}\widehat{(YM)}[i]\not=\varnothing$.

\begin{figure}[h]
\includegraphics[width=5cm]{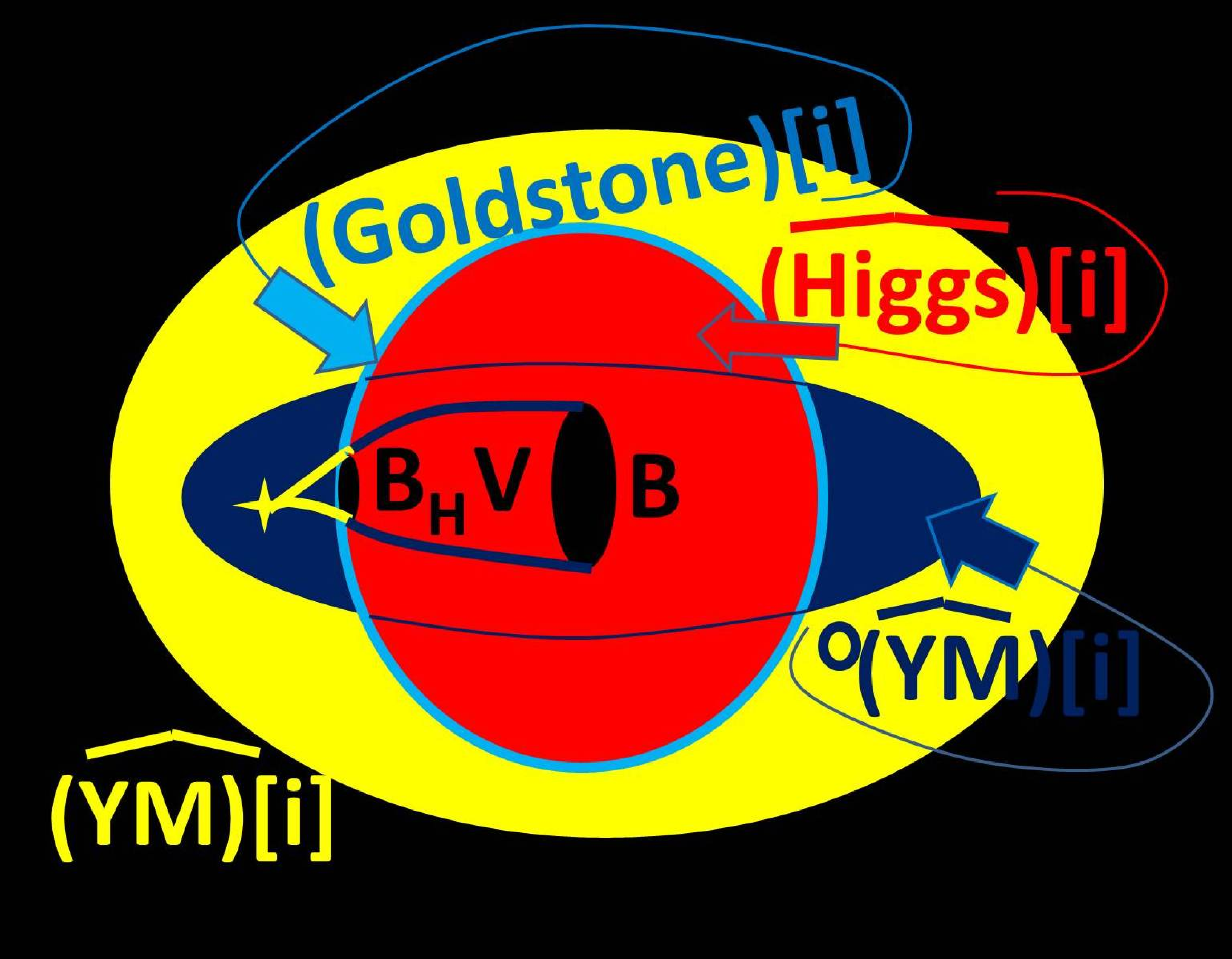}
\renewcommand{\figurename}{Fig.}
\caption{Picture representation of the singular nonlinear quantum propagator $V$ that starting from the point $\bigstar$, origin of the big-bang, bords with $B$, the Cauchy data representing a huge-massive quantum black-hole, after crossed the observed Goldstone boundary $\widehat{(Goldstone)}[i]$. There $V$ identifies a quantum massive-universe $B_H$, that we call {\em Higgs-universe}. $V$ bords $B_H$ with $B$ inside $\widehat{(Higgs)}$. Furthermore, $V$, or better its projection into $\widehat{(YM)}[i]$, is contained also in ${}^{\circ}\widehat{(YM)}[i]$ in order to allow non-equilibrium thermodynamic processes. In this picture, $\widehat{(YM)}[i]$ is represented by the yellow framework, $\widehat{(Higgs)}[i]$ is represented by a red region and ${}^{\circ}\widehat{(YM)}[i]$ is pictured as a blue region.}
\label{figure-planck-epoch-universe}
\end{figure}
\end{example}

\begin{definition}[Thermodynamic quantum exotic nonlinear propagator]\label{thermodynamic-quantum-exotic-nonlinear-propagator}
Let $V$ be a nonlinear quantum propagator in $\widehat{(YM)}[i]$, bording two quantum massive particles $B_1,\, B_2\subset\widehat{(Higgs)}[i]$, i.e., $\partial V=B_1\bigcup P\bigcup B_2$, such that the time $t_1$ of $B_1$ is different from the time $t_2$ of $B_2$. (Say $t_1<t_2$.) Then we say that $V$ is a {\em thermodynamic quantum exotic nonlinear propagator} if $<\omega_H,P>+\Delta H[i|t_1,t_2]_{00}\in A$ is positive definite, where $\Delta H[i|t_1,t_2]_{00}=H[i|t_2]_{00}-\Delta H[i|t_1]_{00}$ is defined in the sense of Equation (26)(I).
\end{definition}
\begin{proposition}\label{proposition-thermodynamic-quantum-exotic-nonlinear-propagator}
Under the hypotheses of Definition \ref{thermodynamic-quantum-exotic-nonlinear-propagator} we get that the quantum energy content of $B_2$ is greater than one of $B_1$
\end{proposition}

\begin{proof}
With respect to equation (24) in part I, we get
\begin{equation}\label{equation-thermodynamic-quantum-exotic-nonlinear-propagator}
 H[i|t_2]_0=H[i|t_1]_0+<\omega_{H}|_{P},P>+\Delta H[i|t_1]_{00}\in A.
\end{equation}

Therefore, when $V$ is a thermodynamic quantum exotic nonlinear propagator we get
$H[i|t_2]_0>H[i|t_1]_0$, since whether $H[i|t_1]_0$ or $<\omega_{H}|_{P},P>+\Delta H[i|t_1]_{00}$ are positive defined.
\end{proof}

\section{\bf Quantum Geometrodynamic Cosmology}\label{sec-quantu-geometrodynamic-cosmology}

In this section we shall prove that the actual expansion of our Universe has its origin in the so-called big-bang, namely it is a natural geometrodynamic consequence of its structure at the Planck epoch. We shall prove that, thanks to this quantum origin,  the energy-matter content of our Universe continuously increases. So the dark-energy-matter conjectured from some cosmologists is nothing else that new real energy-matter produced as a consequence of the conservation of quantum energy at the Planck epoch first and at the Einstein epoch next.

\begin{theorem}\label{big-bang-and-thermodynamic-exotic-component-propagator}
Let us consider the quantum nonlinear bordism $V$ relating the point $\bigstar\in\underline{(YM)}[i]$, origin of the big-bang, with respect to the relativistic quantum frame $i:N\to M$, and the massive Universe $B\subset\widehat{(Higgs)}[i]$, passing for the Higgs-universe $B_H\subset\widehat{(Goldstone)}[i]$: $\partial V=\partial_1V\bigcup P\bigcup B$, where $\partial_1V$ is part of the boundary of $V$ not contained into $\widehat{(Higgs)}[i]$. Then $\partial_1V$ is contained into ${}^{\circ}\widehat{(YM)}[i]$. i.e., it is the boundary of the part $V_1\subset V$ contained into ${}^{\circ}\widehat{(YM)}[i]$, hence $V$ has a thermodynamic quantum exotic component.
\end{theorem}

\begin{proof}
Let us first give a short proof by simply denoting $H[i|V]$ the observed quantum hamiltonian valued on the nonlinear quantum propagator $V$. Set $V=V_1\bigcup V_2$, where $V_1\not\subset \overline{\widehat{(Higgs)}[i]}$, i.e., $V_1$ is not contained in the closure of $\widehat{(Higgs)}[i]$, and $V_2=V\setminus V_1$. For the energy vacuum conservation must be $H[i|V]=0$. On the other hand $H[i|V]=H[i|V_1\bigcup V_2]=H[i|V_1]+H[i|V_2]$. Taking into account that $H[i|V_2]>0$, since $V_2$ has mass-gap, it necessarily follows that must be $H[i|V_1]=-H[i|V_2]<0$, hence $V_1$ is a thermodynamic exotic quantum part of $V$. Let us now give a more precise mathematical meaning to the evaluation $H[i|V]$. In fact to the observed quantum Hamiltonian one can canonically associate whether the meaning of distribution (see Theorem 5.3 in \cite{PRAS14}(III)) or quantum distribution (see Theorem 5.5 in \cite{PRAS14}(III)).\footnote{For complementary information see also the research books \cite{PRAS1,PRAS11}.} Here we shall consider, for simplicity, the second interpretation. With this respect, let us see $V$ as a fiber bundle on the proper-time of the quantum relativistic frame. Then we can write $V=\bigcup_{0\le t\le t_B}V_t$, with $V_0=\bigstar$, $V_{t_{B_H}}=B_H$ and $V_{t_B}=B$. Then one has
\begin{equation}\label{quantum-distribution-observed-hamiltonian}
    H[i|V]=\int_{[0,t_B]}[\int_{V_t}i_t^*<H,\alpha>]dt\in A,\, \forall\alpha\in\mathfrak{C}^3_{ar}(E)
\end{equation}
where $E\equiv M\times \widehat{A}\to M$ is a vector bundle, $\mathfrak{C}^3_{ar}(E)$ is the space of character-3-test functions and $i_t:V_t\subset V\subset N\to M$ is the mapping induced by the quantum relativistic frame $i:N\to M$. We get
\begin{equation}\label{quantum-distribution-observed-hamiltonian-a}
    0=H[i|V]=H[i|V_1]+H[i|V_2]\in A.
\end{equation}
In (\ref{quantum-distribution-observed-hamiltonian-a}) one has denoted $H[i|V_1]=\int_{[t_{0},t_{B_H})}[\int_{V_t}i_t^*<H,\alpha>]dt$ and $H[i|V_2]=\int_{[t_{B_H},t_B]}[\int_{V_t}i_t^*<H,\alpha>]dt$.
Since from Theorem 5.5 in \cite{PRAS14}(III) one has that $H[i|V_2] \in A$ must be definite positive, hence $H[i|V_1]\in A$ must be definite negative. As this holds for any characteristic test function $\alpha$, it follows that the observed quantum Hamiltonian on $V_1$ must be definite negative, therefore $V_1$ is a thermodynamic exotic quantum component of $V$.
\end{proof}
\begin{definition}\label{big-bang-thermodynamic-exotic-quantum-nonlinear-propagator}
We call {\em big-bang nonlinear quantum propagator} the nonlinear quantum propagator $V=V_1\bigcup B_H$, bording $\bigstar$, i.e., the point of the big-bang in ${}^{\circ}\widehat{(YM)}[i]$, with the massive Higgs-universe $B_H$.\footnote{In the following we write indifferently $\partial V_1$  or $\partial_1V$ to denote the boundary of $V_1$.}
\end{definition}

\begin{cor}[Relation between Higgs-universe quantum energy and the big-bang nonlinear quantum propagator]\label{relation-between-quantum-higgs-universe-mass-and-big-bang-quantum-nonlinear-propagator}
One has that the quantum energy content of $B_H$ is determined by the thermodynamic exotic quantum part of $V$:
\begin{equation}\label{equation-relation-between-quantum-higgs-universe-mass-and-big-bang-quantum-nonlinear-propagator}
H[i|B_H]=-H[i|V_1].
\end{equation}
\end{cor}

\begin{theorem}[Relation between Higgs-universe quantum energy and the boundary of the big-bang nonlinear quantum propagator]\label{relation-between-higgs-universe-quantum-energy-and-the-boundary-of-the-big-bang-quantum-nonlinear-propagator}
The Higgs-universe quantum energy is related also to the boundary of the big-bang nonlinear quantum propagator, by means of the quantum energy conservation form $\omega_H$:
\begin{equation}\label{equation-relation-between-higgs-universe-quantum-energy-and-the-boundary-of-the-big-bang-quantum-nonlinear-propagator}
    H[i|B_H]_0=<\omega_H,\partial_1V>+H[i|B_H]_{00}.
\end{equation}
This proves that $<\omega_H,\partial_1V>+H[i|B_H]_{00}$ must be definite positive, since $H[i|B_H]_0$ is so. Therefore the big-bang nonlinear quantum propagator is  a thermodynamic quantum exotic nonlinear propagator in the sense of Definition \ref{thermodynamic-quantum-exotic-nonlinear-propagator}.
\end{theorem}

\begin{proof}
Set $V=V_1\bigcup B_H$ as considered in Definition \ref{big-bang-thermodynamic-exotic-quantum-nonlinear-propagator}.
From the conervation of quantum energy we get\footnote{We have used the same notation of Part I. In other words
$<\omega_H,B_H>=\int_{B_H}[H_{B_H}-(\partial\tilde\mu^{\alpha j}_K.L)\tilde\mu^K_{\alpha j}]\otimes dx^1\wedge dx^2\wedge dx^3=H[i|B_H]_0-H[i|B_H]_{00}$.}
\begin{equation}\label{equation-relation-between-higgs-universe-quantum-energy-and-the-boundary-of-the-big-bang-quantum-nonlinear-propagator-a}
\begin{array}{ll}
                                                      0 & =<d\omega_H,V>\\
                                                       & =<\omega_H,\partial V_1\bigcup B_H>\\
                                                       &=<\omega_H,\partial V_1>+<\omega_H,B_H>\\
                                                       &=<\omega_H,\partial V_1>-H[i|B_H]_0+H[i|B_H]_{00}.\\
                                                    \end{array}
\end{equation}
Therefore, we get formula in (\ref{equation-relation-between-higgs-universe-quantum-energy-and-the-boundary-of-the-big-bang-quantum-nonlinear-propagator}) and the conclusion on the thermodynamic exotic property of the big-bang nonlinear quantum propagator holds.
\end{proof}

\begin{theorem}[Expansiveness-criterion for the Universe at the Planck epoch]\label{expansiveness-criterion-universe-planck-epoch}
A necessary and sufficient condition for the Universe in the Planck-epoch to be expansive is that the quantum non linear propagator, encoding this Universe, after crossed the Goldstone boundary, has the boundary part, $P$, contained between the quantum Higgs-universe and some other massive transversal section $B\subset\widehat{(Higgs)}$,\footnote{See the representation in Fig. \ref{figure-planck-epoch-universe}.} thermodynamic quantum exotic, in the sense of Definition \ref{thermodynamic-quantum-exotic-nonlinear-propagator}, i.e., for the quantum energy conservation form $\omega_H$, one has that $<\omega_H|_P,P>+\Delta H[i|t_1,t_2]_{00}\in A$ is positive definite. We call this condition the {\em quantum expansiveness condition} on the nonlinear quantum propagator encoding the Universe at the Planck epoch.
\end{theorem}
\begin{proof}
This theorem follows directly after Proposition \ref{proposition-thermodynamic-quantum-exotic-nonlinear-propagator}, Theorem \ref{big-bang-and-thermodynamic-exotic-component-propagator}, Theorem \ref{relation-between-higgs-universe-quantum-energy-and-the-boundary-of-the-big-bang-quantum-nonlinear-propagator} and Corollary \ref{relation-between-quantum-higgs-universe-mass-and-big-bang-quantum-nonlinear-propagator}. Let us here only emphasize that the quantum energy content of any nonlinear quantum propagator of the Universe at the Planck epoch is zero, according to the conservation of vacuum quantum numbers. In fact if $\tilde V=V\bigcup_{B_H}\hat V$, with $V=V_1\bigcup B_H$, where $V_1$ is the big-bang nonlinear quantum propagator and $\hat V$ is a nonlinear quantum propagator contained in $\widehat{(Higgs)}[i]$, with $\partial\hat V=B_H\bigcup P\bigcup B$, we get for the quantum energy conservation:
\begin{equation}\label{quantum-energy-conservation-consequences}
    \left\{
    \begin{array}{ll}
      0 =<\omega_H,\partial\tilde V>&=<\omega_H,\partial_1V\bigcup P\bigcup B> \\
      &=<\omega_H,\partial_1V>+<\omega_H,P>+<\omega_H,B>\\
      &=H[i|B_H]+<\omega_H,P>-H[i|B].\\
    \end{array}
    \right.
\end{equation}
Therefore we get $H[i|B]_0=H[i|B_H]_0+<\omega_H,P>+\Delta H[i|t_1,t_{B_H}]_{00}$. Since must be $<\omega_H,P>+\Delta H[i|t_1,t_{B_H}]_{00}$ and $H[i|B_H]_0$ positive definite, it follows that also $H[i|B]_0$ is positive definite, as it must be, since $B$ is a massive quantum universe. On the other hand, the quantum energy content of $\tilde V$ must be zero, hence we get:
\begin{equation}\label{quantum-energy-conservation-consequences-a}
    \left\{
    \begin{array}{ll}
      0 &=H[i|\tilde V]
      =H[i|V\bigcup_{B_H}\hat V]\\
      &=H[i|V_1\bigcup\hat V]\\
      &=H[i|V_1]+H[i|\hat V]\\
      &=-H[i|B_H]+H[i|\hat V].\\
      \end{array}
    \right.
\end{equation}
Therefore, we get $H[i|\hat V]=H[i|B_H]_0$. Since one has $H[i|B_H]_0=H[i|B]_0-(<\omega_H,P>+\Delta H[i|t_{B_H},t]_{00})$ one obtains the formula:
\begin{equation}\label{quantum-energy-conservation-consequences-b}
    H[i|\hat V]=H[i|B]_0-(<\omega_H,P>+\Delta H[i|t_{B_H},t]_{00}).
\end{equation}
By conclusion, the quantum energy content of $\hat V$ is the positive difference between the quantum energy content of $B$ and the quantum contribution of the boundary $<\omega_H,P>'+\Delta H[i|t_{B_H},t]_{00})$ identified by the quantum energy conservation form $\omega_H$.
\end{proof}

\begin{definition}[Planck-epoch-legacy]\label{planck-epoch-legacy}
We define  {\em Planck-epoch-legacy} the property of the Universe's boundary to have thermodynamic quantum exotic components. More precisely let $V$ be the nonlinear propagator encoding the Universe at the Einstein epoch, i.e. $V$ is a compact solution of the Einstein equation in the General Relativity, $(Ein)\subset J^2_4(W)$, with $\partial V=B_1\bigcup P\bigcup B_2$, where $B_1$ and $B_2$ are transversal sections of $V$, representing the Universe in its evolution period. Then the {\em Planck-epoch-legacy} means that the evaluation of the energy conservation form $\omega_H$ on $P$, $<\omega_H,P>++\Delta H[t_1,t_2]$, is positive definite.
\end{definition}
\begin{figure}[h]
\includegraphics[width=5cm]{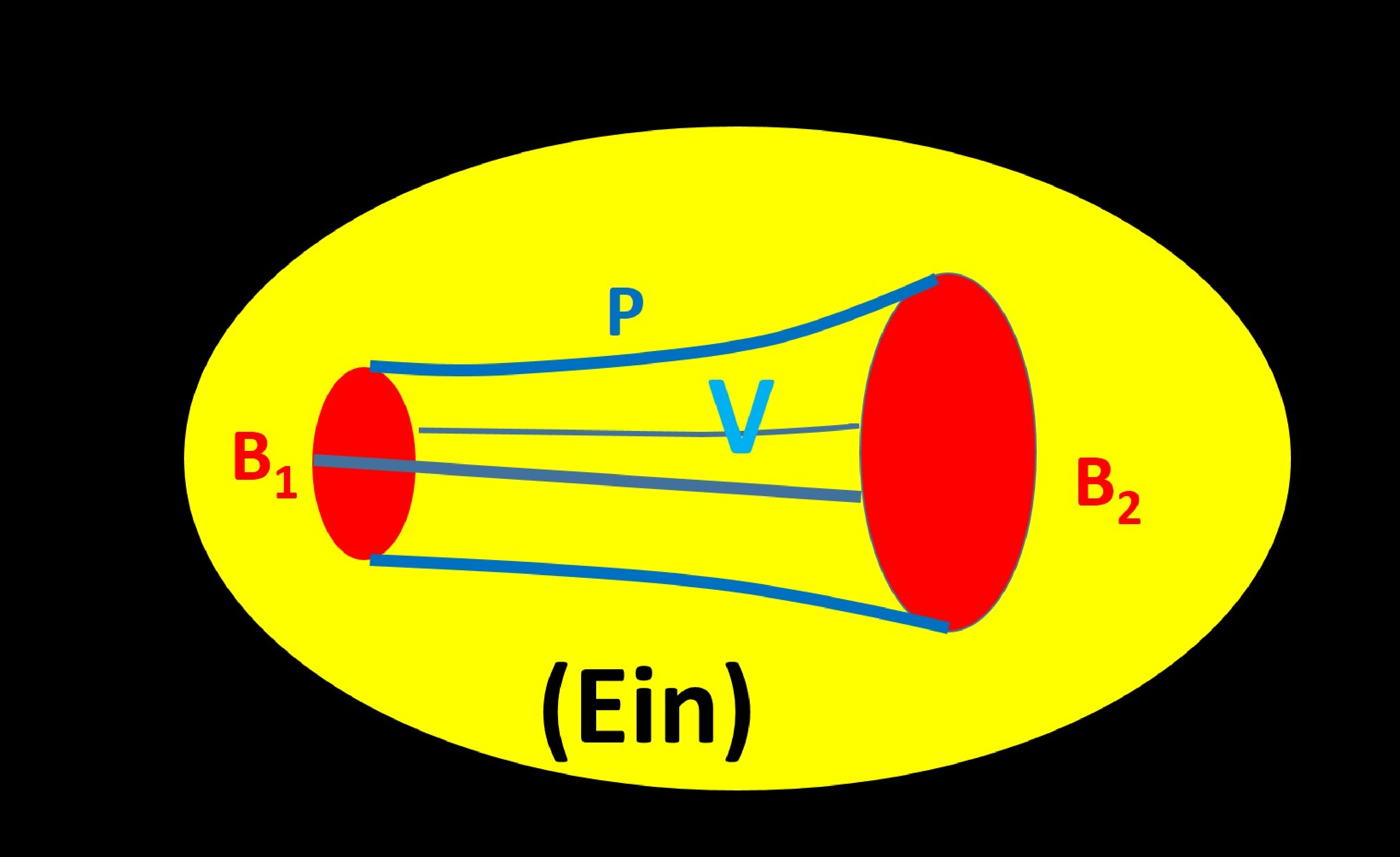}
\renewcommand{\figurename}{Fig.}
\caption{Representation of a nonlinear propagator $V$ bording universe $B_1$ with $B_2$ in the Einstein's equation of the General Relativity $(Ein)$: $\partial V=B_1+P+B_2\subset V\subset(Ein)$. $B_1$, $B_2$ and $P$ are $3$-dimensional integral manifolds. Furthermore, $P$ has a thermodynamic exotic component caused by the Planck epoch legacy.}
\label{figure-einstein-epoch-universe}
\end{figure}

\begin{theorem}[Expansive Universe in the Einstein epoch]\label{universe-einstein-epoch-expansion}
Assuming the Planck-epoch-legacy for the Universe in the Einstein epoch, this last is forced to increase its energy-matter content and to expand.
\end{theorem}
\begin{proof}
In fact since the Einstein equation admits the conservation of energy, one has a conservation $3$-form $\omega_H:(Ein)\to\Lambda^0_3(Ein)$, similar to the one reported in Theorem 3.20 in part I (equation (20). Then taking into account that also at macroscopic level we can encode Universe by means of a nonlinear propagator as described in Definition \ref{planck-epoch-legacy}, we get
\begin{equation}\label{energy-conservation-law-universe-einstein-epoch-expansion}
\begin{array}{ll}
   0 &=<d\omega_H,V>\\
  & =<\omega_H,\partial V>\\
  &=<\omega_H,B_1>-<\omega_H,B_2>+<\omega_H,P>\\
  &=H[i|t_1]_0-H[i|t_2]_0+<\omega_H,P>+\Delta H[t_1,t_2].\\
\end{array}
\end{equation}
Therefore one has $H[i|t_2]_0=H[i|t_1]_0+<\omega_H,P>+\Delta H[t_1,t_2]$. Taking into account that $<\omega_H,P>+\Delta H[t_1,t_2]>0$, according to the Planck-epoch-legacy, and also $H[i|t_1]_0>0$, it follows that $H[i|t_2]_0>H[i|t_1]_0$. Therefore the energy-matter content of the universe increases passing from $B_1$ to $B_2$. This should cause an expansion of the universe.
\end{proof}

\begin{cor}[The zero matter-energy content of the Universe's nonlinear propagator and boundary effect in the Einstein epoch]\label{zero-universe-einstein-epoch}
Since any Universe's nonlinear propagator $V$, in the Einstein epoch respects the conservation of energy, i.e., $<\omega_H,\partial V>=0$, equation {\em(\ref{energy-conservation-law-universe-einstein-epoch-expansion})} can be written as follows:
\begin{equation}\label{boundary-effect}
H[i|t_2]_0-<\omega_H,P>=H[i|t_1]_0+\Delta H[t_1,t_2].
\end{equation}
In other words at the Einstein epoch, the difference between the energy content of the Universe at the instant $t_2$ and the contribution of the boundary's nonlinear propagator, is just the energy content of the Universe at the instant $t_1<t_2$.
\end{cor}

\begin{remark}\label{remark-zero-universe-einstein-epoch}
Corollary \ref{zero-universe-einstein-epoch} agrees with Universe at the Planck epoch, and with its origin from the quantum vacuum.
\end{remark}

\begin{remark}[Dark-energy-matter]\label{cor-dark-energy-matter}
Nowadays the experimental justification of the increasing expansion of our Universe is attributed to some an huge presence of dark-energy-matter, namely energy-matter that is larger than one usually observed. The Pr\'astaro's theory of quantum supergravity, allows to understand that the Universe at the Plank epoch is encoded by a nonlinear quantum propagator with thermodynamic quantum exotic components. This nonlinear quantum propagator has zero quantum energy content and produces an expansion of the massive Higgs-universe, until its macroscopic level, called the Einstein epoch. Then the Planck-epoch-legacy has the effect to increase further Universe's expansion also at the Einstein epoch, thanks to the presence in its boundary of thermodynamic exotic components, as proved in Theorem \ref{universe-einstein-epoch-expansion}. By conclusion, the dark-energy-matter is really the cause of the actual Universe as conjectured by some cosmologic scientist, but it is not so strange or mysterious as it appears. It is a simple boundary effect of the geometrodynamic structure of the Universe legacy at its Planck epoch. Let us recall that in Example 4.19 of \cite{PRAS29} we have discussed also the meaning of so-called ``dark matter" saying that it can be made by massive virtual particles just codified by some theorems in  \cite{PRAS29}. (See Theorem 4.13, Theorem 4.15 and Theorem 4.18 therein.) In other words ``dark matter" can be considered a generic term to identify virtual massive particles produced when a solution of the quantum super Yang-Mills equation $\widehat{(YM)}$, crosses the Goldstone boundary, coming inside the Higgs quantum super PDE, $\widehat{(Higgs)}\subset \widehat{(YM)}$ contained into $\widehat{(YM)}$. This interpretation for ``dark matter" can support also photo-production of matter in high energy $\gamma\gamma$-scatterings.
\end{remark}
Therefore we can state the following theorem.

\begin{theorem}[Matter photo-production]\label{matter-photo-productiion}
$\gamma\gamma$-scatterings can produce solutions with mass-gap, i.e., living inside $\widehat{(Higgs)}[i]$.
\end{theorem}
\begin{proof}
In fact since it is admitted the annihilation reaction: $e^+ + e^{-}\to \gamma+\gamma$, for crossing symmetry it is also admitted the matter production $\gamma+\gamma\to e^++e^-$. This is a direct consequence of the algebraic topologic properties of nonlinear quantum propagators and their integral characteristic quantum numbers, (related to quantum conservation laws). For details see Theorem 4.13 in \cite{PRAS29}. Let us emphasize that solutions encoding such reactions are just nonlinear quantum propagators that acquire mass-gap by crossing the Goldstone boundary, when enter in  $\widehat{(Higgs)}[i]$. Similarly one can state existence of nonlinear quantum propagators in $\widehat{(YM)}$ encoding photo-production of proton-antiproton couples: $2n\gamma\to p+\bar p$. (See Example 4.13 in \cite{PRAS29}. Massive particles and massive gravitons can be obtained by means of suitable nonlinear quantum propagators as shown in Example \ref{quantum-photons-entangled-with-massive-bound-photons}.)\footnote{Similar considerations can be done by means of the crossing-symmetry applied to usual production of virtual $\pi^+$, decaying into positrons and neutrinos, obtained from massive photons identified with virtual states into $\gamma$-nucleons interactions. This is a new point of view looking to long-studied $\pi^+$-photo-production. (For general information on (classical) pions photo-production see, e.g., \cite{DRECHSEL-TIATOR, DUTTA, DUTTA-GAO-LEE} and references quoted therein.) Let us also emphasize that production of matter from photons has been conjectured by G. Breit and J. A. Wheeler \cite{BREIT-WHEELER}. Thanks to Pr\'astaro's algebraic topology of quantum super Yang-Mills PDEs, Breit-Wheeler-type processes become justified by a theorem. (Nowadays experiments with high energy lasers are considered too at the Imperial College, London and Max-Planck Institut, Heidelberg \cite{PIKE-MACKENROTH-HILL-ROSE}.)}
\end{proof}

\begin{conclusions}\label{experimental-verification-quantum-supergravity}
The actual status of our Universe is the most wonderful experimental verification that our Universe at the Planck-epoch was encoded by a nonlinear quantum propagator as described in Example \ref{asymmetry-matter-antimatter-in-quantum-universe} and in Theorem \ref{expansiveness-criterion-universe-planck-epoch} . This is also an experimental checking of the Pr\'astaro's formulation of quantum supergravity theory, built in the framework of his algebraic topology of quantum super PDEs. This theory gives also a precise mathematical support to some early conjectures by F. Hoyle, V. Narlikar and P. A. Dirac on the continuous creation of matter. (See, e.g., \cite{DIRAC,HOYLE,HOYLE-NARLIKAR}.)

Let us emphasize that the main novelty in the Pr\'astaro's theory formulated in \cite{PRAS22} is the introduction of the geometric concept of {\em nonlinear quantum propagator}. This has been possible thanks to the algebraic topology of quantum super PDE's, already published in some works \cite{PRAS1, PRAS2, PRAS3, PRAS4,PRAS5,PRAS11,PRAS12,PRAS14,PRAS15,PRAS19,PRAS21,PRAS22}. This theory realizes for the first time the unification of the Einstein's General Relativity with Quantum Mechanics, and solves most of the outstanding problems in modern fundamental physics. It contains as particular cases the Standard Model, Superstring Theories and M-theory. All these are canonically quantized classical theories (PDE's). With this respect, let us recall that the canonical quantization of classical PDE's from a side identifies some particular quantum (super) PDE's and from another point of view can be considered a linear integration of quantum (super) PDE's. (For details see \cite{PRAS14,PRAS19}.) In the canonical quantization it appears the symbol $\hbar$. This aspect can induce some scientist to believe that a theory, without the presence of such a symbol, cannot be considered a quantum theory. This is a too naive point of view ! In fact, from our theory one can understand that the process of quantization corresponds to a linear integration of quantum (super) PDE's. Instead the more general integration method corresponds to the identification of nonlinear quantum propagators. It is just by means of such more general geometric objects that we have been able to realize the unification between General Relativity and Quantum Mechanics. In other words, this Pr\'astaro's theory does not aim confute Standard Model, but proposes a more general theory that reconciles Einstein's General Relativity and Quantum Mechanics, including Standard Model as a particular case. In some sense the position of the Pr\'astaro's theory with respect to the Standard Model is similar to the one of the Einstein's General Relativity with respect to the Newton's Gravitation Theory. A similar comment can be made about the relation between Pr\'astaro's theory with Superstring Theories and M-theory \cite{HORAVA-WITTEN,WITTEN}. In fact these theories emphasize the importance to consider particles as extended objects taking into account collective contributions that cannot be reduced to ones of point-like particles. On the other hand extended objects are considered also in Pr\'astaro's theory like $p$-chain solutions (or in super sense like $p|q$-chain solutions) of suitable quantum (super) PDE's. However, since Superstring Theories and M-theory are canonically quantized classical theories, their final result is unable to encode quantum nonlinear phenomena. These, instead, find their full formulation by means of nonlinear quantum propagators in quantum super Yang-Mills PDE's, in the sense of the Pr\'astaro's theory. Therefore, despite the great successes of Standard Model, Superstring Theories and M-theory, the fact that they are canonically quantized classical theories, produces their impossibility to gain the unification of Einstein's General Relativity with Quantum Mechanics.\footnote{Let us also add that canonically quantized theories can be anomalous theories. Furthermore, let us also remark that some authors working in String Theories, adopt geometric quantizations of superstrings, or some other type of quantization, to obtain noncommutative geometric objects. But such quantizations are only virtual ones, namely they are completely unrelated to the dynamics where such classical objects should represent solutions of suitable classical (super) PDE's. Therefore these quantizations are only mathematical games, completely unrelated from any possible micro-world physics. (See, e.g., \cite{PRASTARO-Zbl1162-58002}.) Similar criticism can be applied to so-called canonical and loop quantum gravity, lattice gauge theories and lattice supergravities too. (See, e.g., \cite{PRASTARO-MR2870866}.)} Pr\'astaro's theory overcomes this obstacle by his algebraic topology of quantum super PDEs. Really canonical quantization encodes {\em quantum fluctuations} around solutions of classical PDE's or quantum PDE's. Quantum Super-Gravity is a different thing ! It must encode micro-world physics, namely physics at Planck length, where the underlying geometry is not more commutative, but becomes noncommutative thanks to very high energy concentrations.\footnote{In this sense $\hbar$ appears also in our theory since the Planck's length $l_{pl}$ is related to $\hbar$ by the following formula $l_{pl}=\sqrt{\frac{\hbar\, G}{c^3}}=1.616\times 10^{-35} m$. Here $G$ is the gravitation constant and $c$ is the light speed. In this context the classic limit is obtained by substituting the fundamental quantum (super)algebra $A$ of the theory with its subalgebra $\mathbb{R}\subset A$, hence with a restriction process. Local super-limits of quantum supergravity  theories can be also obtained. For details see \cite{PRAS14}(III).} Therefore Quantum Super-Gravity cannot be encoded by quantum fluctuations of solutions of the Einstein's equation, or some other classical PDE, as, for example, ones arising from Superstring Theories and M-theory.  Quantum Super-Gravity is necessarily a geometric formulation of PDE's in the category of quantum supermanifolds, focused on the quantum super Yang-Mills PDE's ({\em Quantum SG-Yang-Mills PDE's}) as formulated by A. Pr\'astaro in \cite{PRAS22}.
\end{conclusions}

\begin{appendices}

\appendix{\bf Appendix A: Nonlinear quantum propagators: Some Numerical Results I.}\label{appendixa}
\renewcommand{\theequation}{A.\arabic{equation}}
\setcounter{equation}{0}  

\begin{figure}[h]
\includegraphics[width=5cm]{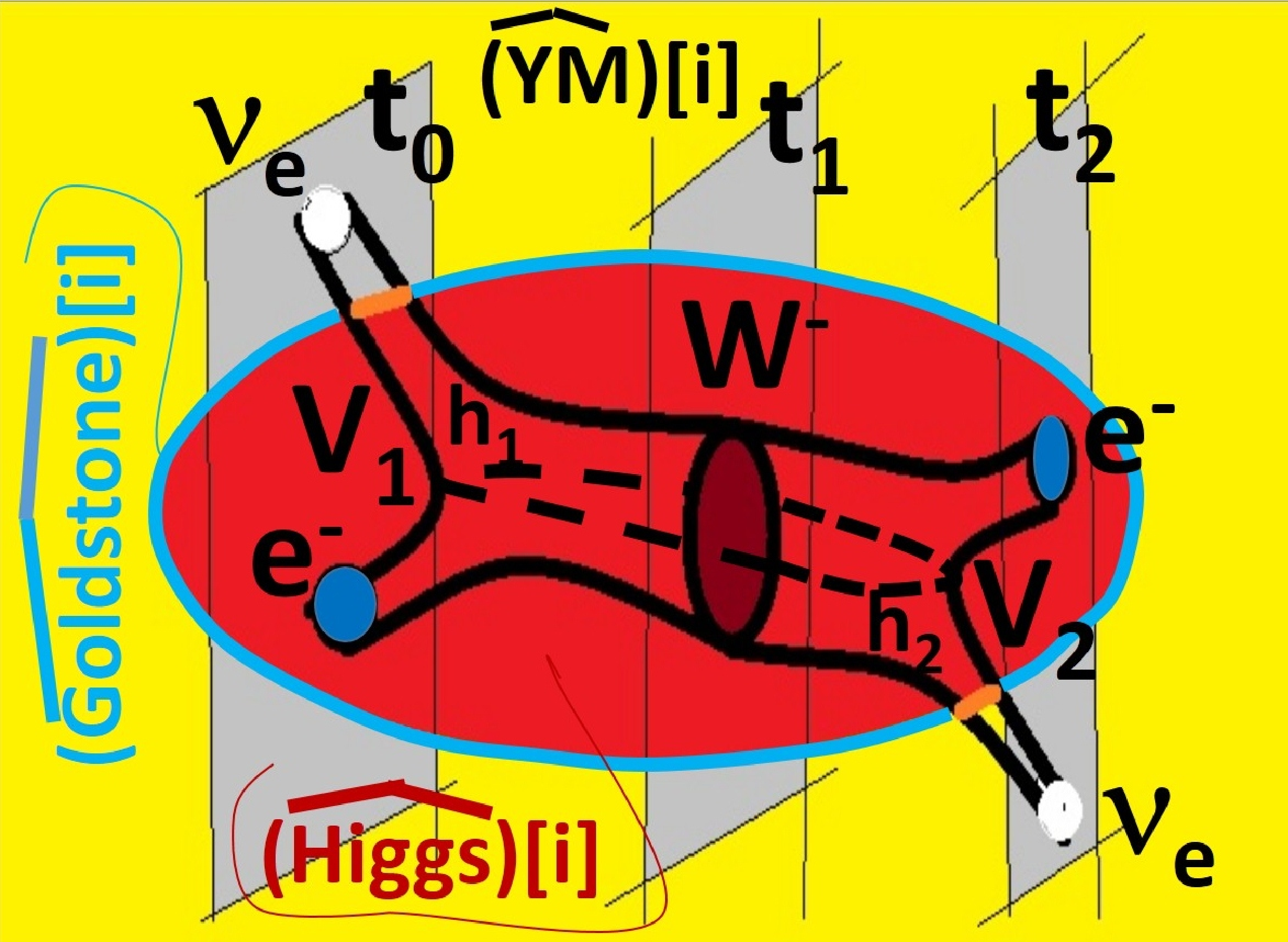}
\caption{Representation of the observed interaction neutrino-electron with the production of the virtual massive boson $W^-$. The singular nonlinear quantum propagator $V$ splits in two nonlinear quantum propagators $V_1$ and $V_2$, such that $V=V_1\bigcup_W V_2$.  The initial Cauchy data is $N_0=\nu_e\bigcup e^-$, the intermediate Cauchy data is $N_1=W^-$ and the final Cauchy data is
$N_2=e^-\bigcup \nu_e$. The neutrino is considered massless according to the Standard Model. However, since $V_1$ must cross the Goldstone boundary, neutrino acquires mass going inside $\widehat{(Higgs)}[i]$, i.e., it identifies a virtual massive neutrino $\nu_e'$, that is the one observed. Similarly, after the interaction neutrino remains a massive virtual neutrino inside $\widehat{(Higgs)}[i]$, hence it is experimentally detected as a massive particle. However one can think that such virtual massive neutrinos going outside $\widehat{(Higgs)}[i]$ should lose their masses. This is an example of nonlinear quantum propagator encoding an anomaly-massive quantum virtual particle ($W^-$). (See Theorem 4.23(II).)}
\label{neutrino-electron-interaction-nonlinear-quantum-propagator}
\end{figure}

In this appendix we shall emphasize the importance of nonlinear quantum propagators by considering some numerical results arising from experimental data. In particular we shall consider some calculations to obtain masses for the bosons $W^{\pm}$ and $Z^0$, in the celebrated Weinberg-Salam model for electro-weak reactions. (See the following Wikipedia link: \href{http://en.wikipedia.org/wiki/Electroweak_interaction}{Electroweak interaction}.) We will show that also in this important success of the Standard Model, the Pr\'astaro's Algebraic Topology of quantum super PDEs gives a precise geometric support and justification for the experimental data, that otherwise could not be completely understood.

As it is well-known in the Weinberg-Salam model for electro-weak reactions, the Gauge group is
\begin{equation}\label{ws-gauge-group}
  G_{WS}=SU(2)_{weak-isospin}\times U(1)_{weak-hypercharge}.
\end{equation}
To this group correspond $3$ gauge massless bosons $(W^1,W^2,W^3)$ from $SU(2)$ and  $1$ massless boson $(B^0)$ from $U(1)$. For the gauge symmetry breaking mechanism the above four gauge bosons become $(W^+,Z^0,W^-,\gamma)$, where the first three have masses. In the Weinberg-Salam model the way to calculate the $W$ mass is to consider the neutrino-electron scattering $\nu_e+e^-\to e^-+\nu_e$ by means of the Feynman diagram reported in Fig. \ref{feynman-diagram-electron-neutrino}
\begin{figure}[h]
\includegraphics[width=4.3cm]{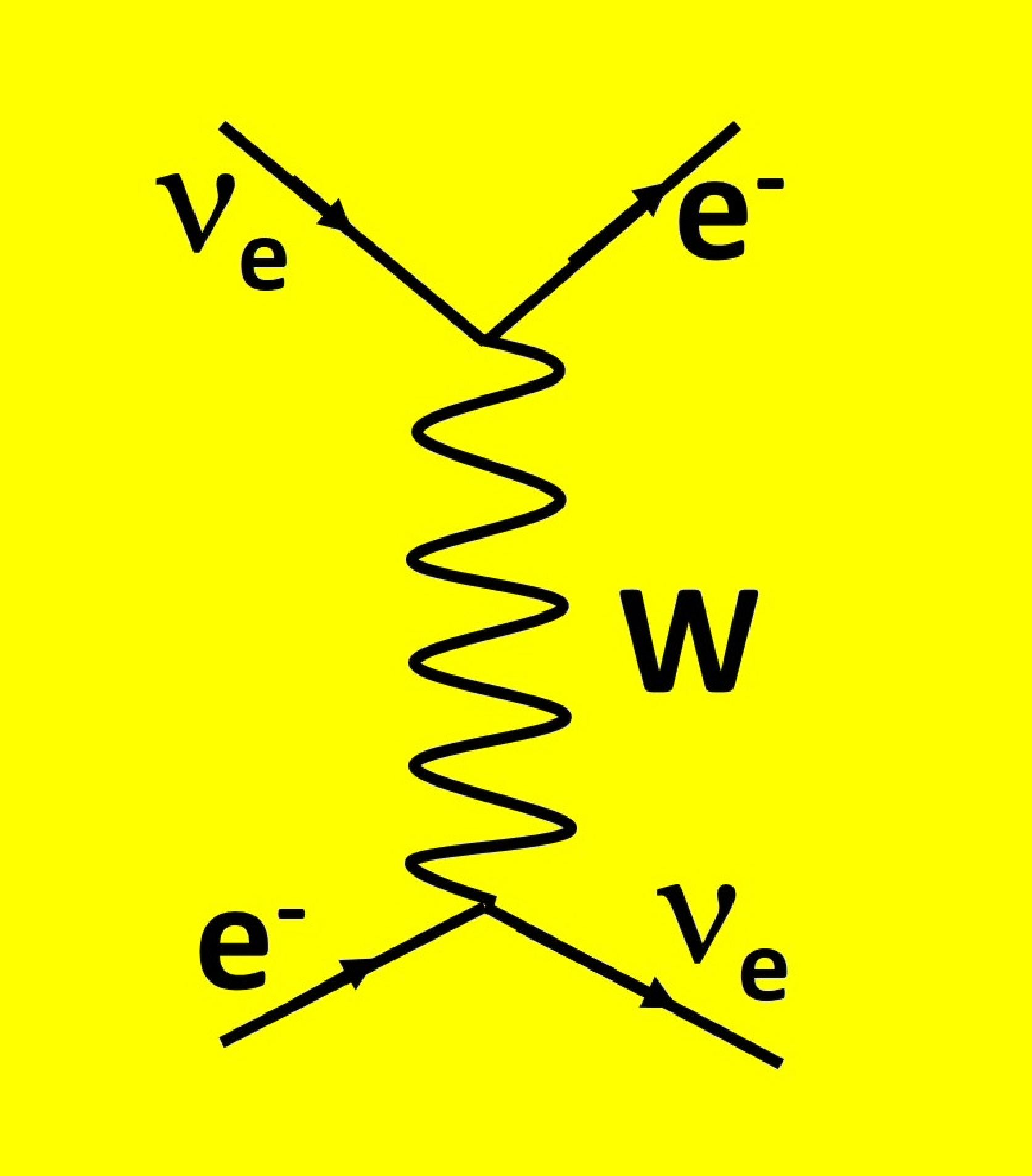}
\renewcommand{\figurename}{Fig.}
\caption{Feynman diagram of the scattering $e^-+\nu_e\to e^-+\nu_e$.}
\label{feynman-diagram-electron-neutrino}
\end{figure}

involving $W$ exchange. The amplitude of this process is given by
\begin{equation}\label{amplitude-weinber-salam-neutrino-electron}
  a=(\frac{g}{\sqrt{2}}J^-_\mu)(\frac{1}{M_W^2-q^2})(\frac{g}{\sqrt{2}}J^+{}^\mu)
\end{equation}
where the first (resp. third) term on the right of (\ref{amplitude-weinber-salam-neutrino-electron}) represents the top current, (resp. the bottom) vertex current, and the term in the middle denotes the propagator (in the Feynman language). Assuming the low-energy limit, namely $|q^2|\ll M_W^2$, one can write $a=\frac{g^2}{2}\frac{1}{M_W^2}J^-_\mu J^+{}^\mu$. By comparing this last amplitude with the one without the $W$-exchange, namely $a'=\frac{4 G_F}{\sqrt{2}}J^-_\mu J^+{}^\mu$, we get $M_W=\sqrt{\frac{g^2\sqrt{2}}{8G_F}}$, where $G_F=1.666\times 10^{-5}  GeV^{-2}$ is the Fermi coupling. Taking into account that $g\sin\theta_W=e$, where $\theta_W$ is the Weinberg angle and $\sin^2\theta_W=0.233$, and $\frac{e^2}{4\pi}=\frac{1}{137}$, we get
\begin{equation}\label{calculated-w-mass}
  M_W=\frac{37.28}{\sin\theta_W} =\frac{37.28}{0.4722}=78.949\, GeV.
\end{equation}
The experimental value is $M_W=80.41\, GeV$.\footnote{The relation between the $W$ mass and the mass of $Z^0$ is $\cos\theta_W=\frac{M_W}{M_{Z^0}}$. Therefore one has $M_{Z_0}=\frac{80.41}{\sqrt{1-0.223}}=91.271$.}  The difference $1.461\, GeV=(80.41-78.949)GeV$ can be reduced by considering higher order Feynman diagrams. However this difference shows that the Feynman method, namely the perturbative method, at the first order is not precise enough. Furthermore, this method cannot be used for strong interactions, since there it does not converge. Therefore the concept of exchange force is not well mathematically defined !

Let us, now, show how encode the interaction $e^-+\nu_e\to \nu_e+e^-$, with a nonlinear quantum propagator, without necessarily assuming low energy limit. We shall consider a nonlinear quantum propagator $V\subset \widehat{(YM)}[i]$, such that $\partial V=N_0\bigcup P\bigcup N_2$, where $N_0$ and $N_2$ are space-like Cauchy data $N_0=e^-\bigcup \nu_e$ and $N_2=e^-\bigcup \nu_e$ at different times $t_0<t_2$, such that $P$ is a time-like $3$-dimensional manifold, with $\partial P=\partial N_0\bigcup \partial N_2$. Furthermore $V$ splits in two nonlinear quantum propagators $V_1$ and $V_2$, such that $V=V_1\bigcup_{N_1} V_2$, where $N_1=W^-$, a virtual intermediate massive particle. See Fig. \ref{neutrino-electron-interaction-nonlinear-quantum-propagator}
 for a picture representing the interaction neutrino-electron. The geometric process is realized in two steps.

a) The massless neutrino $\nu_e$ interacts with the electron $e^-$ by means of a exchangion (handle $h_1$) producing the nonlinear quantum propagator $V_1$ and the virtual massive particle $W^-$. $V_1$ crossing the Goldstone boundary identifies a massive neutrino $\nu'_e$.

b) Since $W^-$ is a virtual particle it is unstable and emits an exchangion (handle $h_2$) that generates the nonlinear quantum propagator $V_2$ bording $W$ with $e^-$ and the massive neutrino $\nu_e'$. Again $V_2$ crossing the Goldstone boundary should transform the massive neutrino $\nu'_e$ into a massless one $\nu_e$.

The quantum energy content of $V$ is expressed by conservation of energy, resumed in the equations (\ref{conservation-energy-relations-in-neutrino-electron-interaction}).
\begin{equation}\label{conservation-energy-relations-in-neutrino-electron-interaction}
  \left\{
\begin{array}{ll}
  0=<\omega_H,\partial V>=&<\omega_H,N_0>+<\omega_H,N_2>+<\omega_H,P>\\
  0=<\omega_H,\partial V_1>=&<\omega_H,N_0>+<\omega_H,N_1>+<\omega_H,P_1>\\
  0=<\omega_H,\partial V_2>=&<\omega_H,N_1>+<\omega_H,N_2>+<\omega_H,P_2>.\\
\end{array}
 \right.
\end{equation}

From the second equation in (\ref{conservation-energy-relations-in-neutrino-electron-interaction}) we get (see Theorem 3.20(I)):
\begin{equation}\label{conservation-energy-relations-in-neutrino-electron-interaction-a}
H_0[e^-]+H_0[\nu_e]-H_0[W]=H_{00}[e^-]+H_{00}[\nu_e]-H_{00}[W]-<\omega_H,P_1>.
\end{equation}
Therefore we have the defect quantum energy:
\begin{equation}\label{conservation-energy-relations-in-neutrino-electron-interaction-b}
\mathfrak{H}[V_1]=H_{00}[e^-]+H_{00}[\nu_e]-H_{00}[W]-<\omega_H,P_1>.
\end{equation}

Similarly from the third equation in (\ref{conservation-energy-relations-in-neutrino-electron-interaction}) we get
\begin{equation}\label{conservation-energy-relations-in-neutrino-electron-interaction-c}
H_0[W]-H_0[\nu_e]-H_0[e^-]=H_{00}[W]-H_{00}[\nu_e]'-H_{00}[e^-]'-<\omega_H,P_2>.
\end{equation}
and the defect quantum energy:
\begin{equation}\label{conservation-energy-relations-in-neutrino-electron-interaction-d}
\mathfrak{H}[V_2]=H_{00}[W]+H_{00}[\nu_e]'-H_{00}[e^-]'-<\omega_H,P_2>.
\end{equation}
From the first equation in (\ref{conservation-energy-relations-in-neutrino-electron-interaction}) we have
\begin{equation}\label{conservation-energy-relations-in-neutrino-electron-interaction-e}
\left\{\begin{array}{ll}
  H_0[\nu_e]+H_0[e^-]-H_0[\nu_e]'-H_0[e^-]'&=H_{00}[\nu_e]+H_{00}[e^-]-H_{00}[\nu_e]'-H_{00}[e^-]' \\
&-<\omega_H,P_1\bigcup_W P_2>.\\
\end{array}\right.
\end{equation}
With the following defect quantum energy:
\begin{equation}\label{conservation-energy-relations-in-neutrino-electron-interaction-f}
\mathfrak{H}[V]=H_{00}[e^-]+H_{00}[\nu_e]-H_{00}[\nu_e]'-H_{00}[e^-]'-<\omega_H,P>.
\end{equation}
Therefore we get

\begin{equation}\label{conservation-energy-relations-in-neutrino-electron-interaction-g}
\mathfrak{H}[V]=\mathfrak{H}[V_1]+\mathfrak{H}[V_2].
\end{equation}
From experimental evaluations of $\mathfrak{H}[V_1]$ at the mass-gap level, denoted $\mathfrak{H}[V_1]_m$, we get:\footnote{Warning. In these calculations the mass of the neutrino does not significatively contribute. In fact the experimental data are $m(\nu_e)=2.2\times 10^{-9} \, GeV$, $m(e^-)=5.11\times 10^{-7}\, GeV$ and $m(W)=80.41\, GeV$. For this motivation we write indifferently $\nu_e$ or $\nu_e'$ in above calculations. In other words, we can simply consider the massless neutrino $\nu_e$, instead of the corresponding massive intercept when it crosses the Goldstone boundary going inside $\widehat{(Higgs)}[i]$. Similarly we can consider the massless neutrino $\nu_e$, instead of the corresponding massive particle, when it crosses the Goldstone boundary going outside $\widehat{(Higgs)}[i]$.}
\begin{equation}\label{conservation-energy-relations-in-neutrino-electron-interaction-h}
\mathfrak{H}[V_1]_m=H_0[e^-]_m+H_0[\nu_e]_m-H_0[W]_m=-80.409\, GeV.
\end{equation}
We can write:
\begin{equation}\label{conservation-energy-relations-in-neutrino-electron-interaction-l}
\mathfrak{H}[V_1]_m=H_{00}[e^-]_m+H_{00}[\nu_e]_m-H_{00}[W]_m-<\omega_H,P_1>_m=-80.409\, GeV.
\end{equation}
$V_1$ is not a thermodynamic quantum exotic nonlinear propagator, (in the sense of Definition \ref{thermodynamic-quantum-exotic-nonlinear-propagator}), since $\mathfrak{H}[V_1]<0$. By a similar calculation we get
\begin{equation}\label{conservation-energy-relations-in-neutrino-electron-interaction-m}
\mathfrak{H}[V_2]_m=80.409\, GeV=H_{00}[W]_m-H_{00}[\nu_e]_m'-H_{00}[e^-]_m'-<\omega_H,P_2>_m.
\end{equation}
We have also
\begin{equation}\label{conservation-energy-relations-in-neutrino-electron-interaction-n}
\left\{\begin{array}{ll}
  \mathfrak{H}[V]_m & =\mathfrak{H}[V_1]_m+\mathfrak{H}[V_2]_m\\
  &= 0\\
  &= H_{00}[e^-]_m+H_{00}[\nu_e]_m-H_{00}[\nu_e]_m'-H_{00}[e^-]_m'-<\omega_H,P>.\\
  \end{array}\right.
\end{equation}
Therefore we can state that the propagator $V$ has zero defect quantum energy, hence $V$ cannot be a thermodynamic quantum exotic nonlinear propagator, but the geometric contribution of $V$ is very important and is given by equation (\ref{conservation-energy-relations-in-neutrino-electron-interaction-o}).
\begin{equation}\label{conservation-energy-relations-in-neutrino-electron-interaction-o}
-<\omega_H,P>_m=H_{00}[\nu_e]_m'+H_{00}[e^-]_m'-H_{00}[e^-]_m-H_{00}[\nu_e]_m.
\end{equation}
Since $H_{00}[\nu_e]_m'+H_{00}[e^-]_m'>H_{00}[e^-]_m+H_{00}[\nu_e]_m$, it follows that $<\omega_H,P>_m<0$, hence $V$ is not a thermodynamic quantum exotic nonlinear propagator. Moreover, we find that the contribution of the geometric structure nonlinear quantum propagator is fundamental to obtain a precise mathematical formulation of quantum reactions, even if these are restricted to weak interactions only.\footnote{Let us remark that in this process we can assume that $W$ is in a stationary-state since the proper quantum time $\sigma^2(\hat t)$ of $W$ is
$\sigma^2(\hat t)=\frac{\hbar}{M_W}=\frac{6.5821\times 10^{-25}\, GeV\, s}{80.385\, GeV}=8.18\times 10^{-27}\, s$,
namely a very short time.} Furthermore, the geometric structure of the quantum super Yang-Mills PDEs, containing the constraint given by the quantum super Higgs PDEs, where live particles with mass-gap, gives a precise geometrodynamic justification to the presence of msssive bosons like $W^{\pm}$ and $Z^0$, but also justifies massive neutrinos, without contradict the Standard Model. This, and other considerations, allows us to state that the Algebraic Topology of quantum super Yang-Mills PDEs includes the Standard Model as a particular case, without incurring in the well-known problems just related to the Standard Model. (See, e.g., the following beautiful Wikipedia link, \href{http://en.wikipedia.org/wiki/Standard_Model}{Standard Model}, where the principal problems related to the Standard Model are carefully considered.)
\vskip 10pt

\appendix{\bf Appendix B: Nonlinear quantum propagators: Some Numerical Results II.}\label{appendixb}
\renewcommand{\theequation}{B.\arabic{equation}}
\setcounter{equation}{0}  

In this appendix we consider the so-called beta neutron decay $n^0\to p^++e^-+\overline{\nu}_e$. This is another example where the Standard Model and the Weinberg-Salam model for electroweak reactions have obtained important results. Really the Standard Model considers neutron made by an up quark and two down quarks (udd), and the beta neutron decay is interpreted by considering that a down quark inside $n$, changes the flavour by means of the decay $d\to u+W^-$, by means of the weak force. Then it follows the decay of the massive boson $W^-$: $W^-\to e^-+\overline{\nu}_e$. The corresponding Feynman diagram is reported in Fig. \ref{feynman-diagram-beta-neutron-decay}.
\begin{figure}[h]
\includegraphics[width=4.3cm]{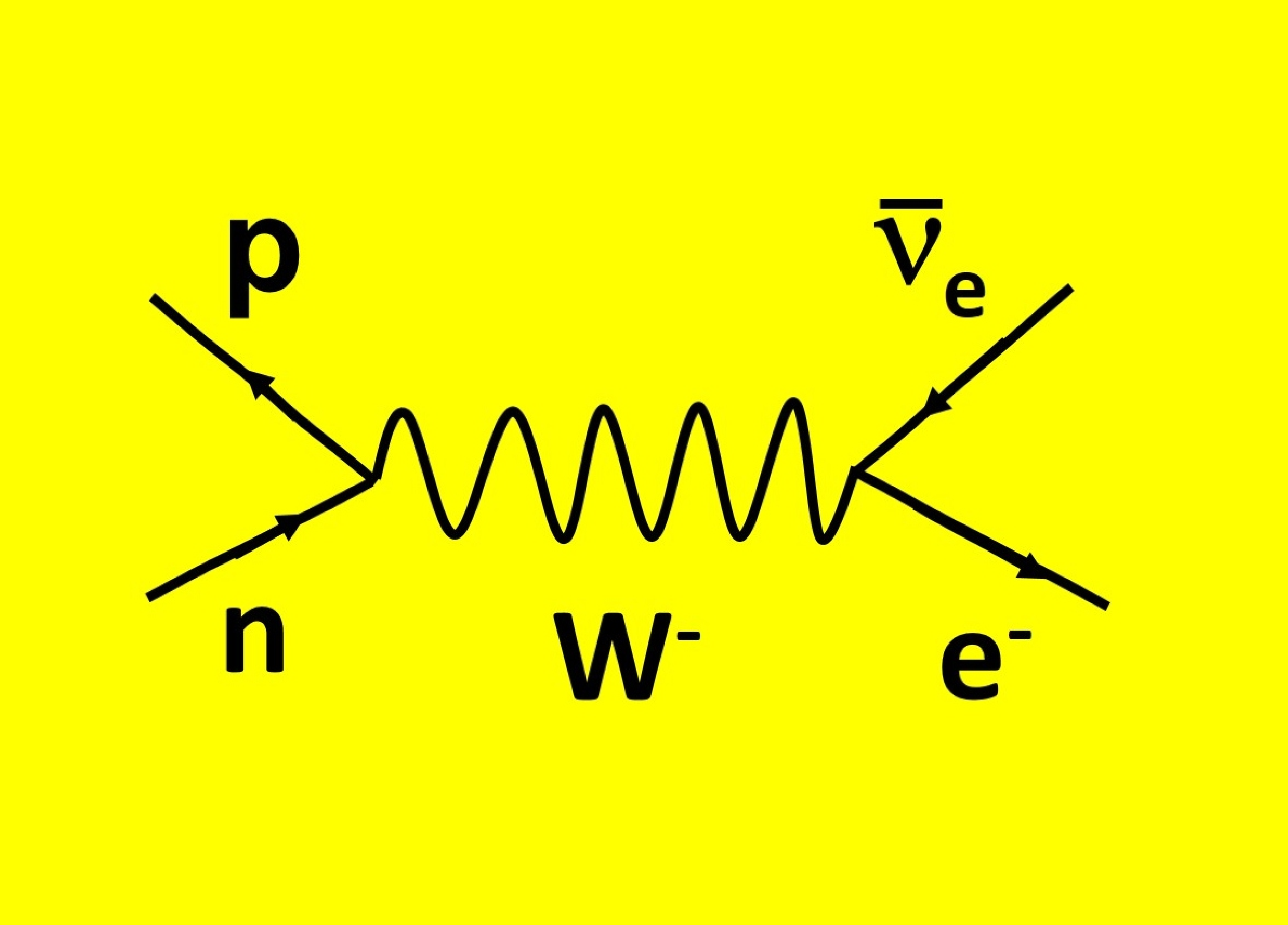}
\renewcommand{\figurename}{Fig.}
\caption{Feynman diagram of the beta neutron decay $n\to p+e^-+\overline{\nu}_e$.}
\label{feynman-diagram-beta-neutron-decay}
\end{figure}

 We will verify that also in this context nonlinear quantum propagators allow us to have a complete justification of the heavy $W^-$ boson as a virtual particle. In fact the observed nonlinear quantum propagator $V$ encoding such a decay has $\partial V=N_0\bigcup P\bigcup N_2$ with $N_0=n$, $N_2=p\bigcup e^-\bigcup\overline{\nu}_e$, and splits into $V=V_1\bigcup_{N_1} V_2$, where $V_1$ is a nonlinear quantum propagator such that $\partial V_1=N_0\bigcup P_1\bigcup N_1$, with $N_1=p\bigcup W^-$, $\partial P_1=\partial N_0\bigcup \partial N_1$ and $\partial V_2=N_1\bigcup P_2\bigcup N_2$, with $\partial P_2=\partial N_1\bigcup \partial N_2$. More precisely the process can be divided in two steps. (See also Fig. \ref{beta-neutron-decay-nonlinear-quantum-propagator}.)
\begin{figure}[h]
\includegraphics[width=5cm]{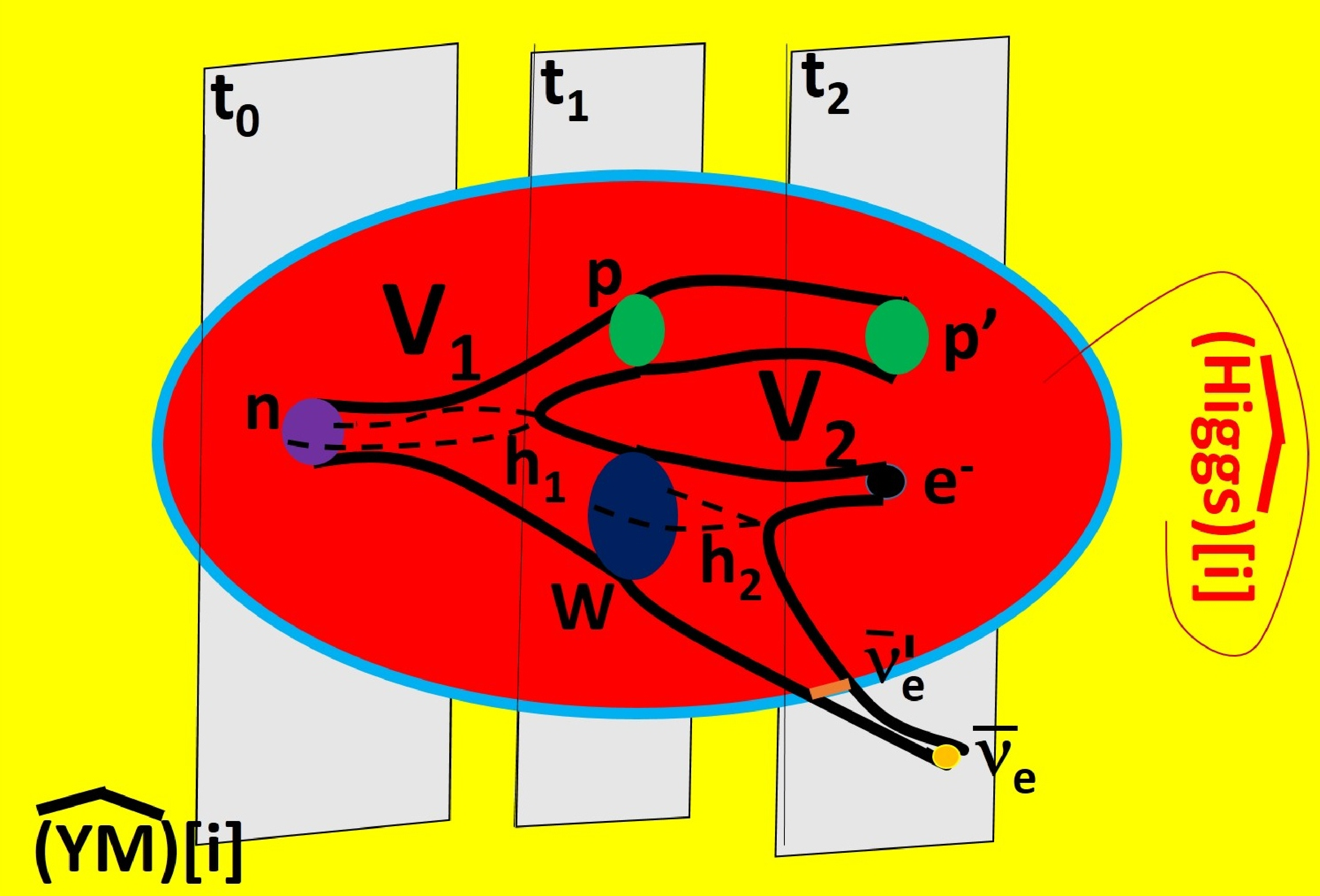}
\caption{Representation of the observed beta neutron decay with production of a virtual massive boson $W^-$. The singular nonlinear quantum propagator $V$ splits in two nonlinear quantum propagators $V_1$ and $V_2$, such that $V=V_1\bigcup_{N_1} V_2$.  The initial Cauchy data is $N_0=n^0$, the intermediate Cauchy data is $N_1=p\bigcup W^-$ and the final Cauchy data is
$N_2=e^-\bigcup \overline{\nu}_e\bigcup p$. The neutrino is considered massless according to the Standard Model. However, since $V_2$ must cross the Goldstone boundary, before going outside $\widehat{(Higgs)}[i]$, where is realized the beta decay of $n^0$, $\overline{\nu}_e$ can be considered massive, i.e., it identifies a virtual massive anti-neutrino $\overline{\nu}_e'$, that is the one observed. This is another example of nonlinear quantum propagator encoding an anomaly-massive quantum virtual particle ($W^-$). (See Theorem 4.23(II).)}
\label{beta-neutron-decay-nonlinear-quantum-propagator}
\end{figure}

 (a) The free neutron is unstable, hence emits an exchangion (handle $h_1$) that generates the nonlinear quantum propagator $V_1$ that bords $n$ with $p$ and the virtual massive particle $W$.

 (b) $W$ is unstable and emits an exchangion (handle $h_2$) generating a nonlinear quantum propagator $V'_2$ that bords $W$ with $N'_2=e^-\bigcup \overline{\nu}_e$. Furthermore, $p$ freely propagates with a steady-state nonlinear quantum propagator $V''_2\cong p\times I$. Therefore $V_2=V'_2\bigcup V''_2$ is obtained from $N_1$ by attaching an handle to $W$, hence encodes the decay of $N_1$.
 By applying the quantum energy conservation law (Theorem 3.20(I)), we get the following equations:

 \begin{equation}\label{conservation-energy-relations-in-beta-neutron-decay}
  \left\{
\begin{array}{ll}
  0&=H_0[n]-H_0[W]-H_0[p]-H_{00}[n]+H_{00}[W]+H_{00}[p]+<\omega_H,P_1>\\
  0&=H_0[W]+H_0[p]-H_0[p]'-H_0[e^-]-H_0[\overline{\nu}_e]\\
  &-H_{00}[W]-H_{00}[p]+H_{00}[p]'+H_{00}[e^-]++H_{00}[\overline{\nu}_e]+<\omega_H,P_2>\\
  0&=H_0[n]-H_0[p]-H_0[e^-]-H_0[\overline{\nu}_e]\\
  &-H_{00}[n]+H_{00}[p]'+H_{00}[e^-]+H_{00}[\overline{\nu}_e]+<\omega_H,P_1\bigcup P_2>.\\
\end{array}
 \right.
\end{equation}
Therefore, we get the following defect quantum energies:
 \begin{equation}\label{defect-quantum-energies-in-beta-neutron-decay}
  \left\{
\begin{array}{ll}
  \mathfrak{H}[V_1]&=H_{00}[n]-H_{00}[W]-H_{00}[p]-<\omega_H,P_1>\\
  &=H_0[n]-H_0[W]-H_0[p]\\
  \mathfrak{H}[V_2]&=H_{00}[W]+H_{00}[p]-H_{00}[p]'-H_{00}[e^-]-H_{00}[\overline{\nu}_e]-<\omega_H,P_2>\\
  &=H_0[W]+H_0[p]-H_0[p]'-H_0[e^-]-H_0[\overline{\nu}_e]\\
  \mathfrak{H}[V]&=H_{00}[n]-H_{00}[p]'-H_{00}[e^-]-H_{00}[\overline{\nu}_e]-<\omega_H,P_1\bigcup P_2>\\
  &=H_0[n]-H_0[p]-H_0[e^-]-H_0[\overline{\nu}_e].\\
\end{array}
 \right.
\end{equation}
We get $\mathfrak{H}[V]=\mathfrak{H}[V_1]+\mathfrak{H}[V_2]$. Furthermore, since $p$, after the decay of $n$ can be assumed free, we can state that $p$ is assumed in a steady-state, hence $H_{00}[p]=H_{00}[p]'$. Let us now consider the mass-gap content of above equations (\ref{defect-quantum-energies-in-beta-neutron-decay}).
 \begin{equation}\label{mass-gap-defect-quantum-energies-in-beta-neutron-decay}
 \left\{
\begin{array}{ll}
 \mathfrak{H}[V_1]_m&=H_{00}[n]_{m}-H_{00}[p]_m-H_{00}[W]_m-<\omega_H,P_1>_m=-80.386\, GeV\\
\mathfrak{H}[V_2]_m&=H_{00}[n]_m-H_{00}[e^-]_m-H_{00}[\overline{\nu}_e]_m-<\omega_H,P_2>_m=80.386\, GeV\\
\mathfrak{H}[V]_m&=H_{00}[n]_m-H_{00}[p]_m-H_{00}[e^-]_m-H_{00}[\overline{\nu}_e]_m-<\omega_H,P>_m\\
&=0.49\times 10^{-3}\, GeV.\\
\end{array}
\right.
\end{equation}

From such estimates one can understand the great importance of nonlinear quantum propagators. Without such geometric structures it should not be possible to justify the realization of a massive quantum virtual particle like $W^-$.\footnote{Let us also emphasize that in above calculations the mass of $\overline{\nu}'_e$, i.e., the corresponding virtual massive particle identified by $V_2$ crossing the Goldstone boundary, does not sensibly contribute.}

\end{appendices}


\begin{thebibliography}{2020}

\bibitem{AAD-ET-AL} G. Aad et al., \textit{Observation on centrality dependent dijet asymmetry in lead-lead collisions at $\sqrt{s_{NN}}=276$ Tev with the ATLAS detector at the LHC}, Phys. Rev. Lett. \textbf{105(17)}(2010), 252303--17.
{[CERN Press Releases,  \href{http://press.web.cern.ch/press-releases/2010/11/lhc-experiments-bring-new-insight-primordial-universe}{\textit{LHC experiments bring new insight into primordial universe}}, November 26, 2010. Retrieved December 2, 2010.]}

\bibitem{ABLIKIM} M. Ablikim et al., \textit{Observation of a charged charmoniumlike structure in $e^+e^- \to \pi^+pi^-J/\psi$ at $\sqrt{s}=4.26$
     GeV}, Phys. Rev. Lett. \textbf{110(25)}(2013), 252001 (8 pp).

\bibitem{AHARONOV-ET-AL} Y. Aharonov, et al. \textit{Revisiting Hardy's paradox: counterfactual statements, real measurements, entanglement and weak values}. Physics Letters A \textbf{301}(2002), 130–-138.


\bibitem{AHARONOV-ROHRLICH} Y. Aharonov and D. Rohrlich, \textit{Quantum Paradoxes: Quantum Theory for the Perplexed}, (Weinheim; Wiley–VCH), 2005.

\bibitem{AHARONOV-POPESCU-ROHRLICH-SKRZYPCZYK} Y. Aharonov, S. Popescu, D. Rohrlich, \& P. Skrzypczyk, \textit{Quantum Cheshire Cats}. New J. Phys. \textbf{15(11)}(2013), 113015 (8 pp).

\bibitem{ATIYAH} M. F. Atyiah, \textit{Bordism and cobordism}, Proc. Camb. Phil. Soc. \textbf{57}(1961), 200--208.

\bibitem{BEHBOOD-ET-ALT} N. Behbood, et alt. \textit{Generation of macroscopic singlet states in a cold atomic ensemble}, Phys. Rev. Letters \textbf{113}(2014), 093601--5.

\bibitem{BELL} J. S. Bell, \textit{On the Einstein Podolsky Rosen paradox}, Physics \textbf{1(3)}(1964), 195-200.

\bibitem{BREIT-WHEELER} G. Breit and J.A. Wheeler,\textit{ Collision of two light quanta}, Phys. Rev. II Ser. \textbf{43}(1934), 1087--1091.

\bibitem{CARR} L. D. Carr, \textit{Negative temperature ?}, Science \textbf{339(6115)}(2013), 42-43.

\href{http://dx.doi.org/10.1126/science.1232558}{DOI:10.1126/science.1232558}

\bibitem{CHANG} D. E. Chang, V. Gritsev, G. Morigi, V. Vuleti\'c, M. D. Lukin and E. A. Demler, \textit{Crystallization of strongly interacting photons in a nonlinear optical fibre}, Nature Physics, \textbf{4}(2008), 884--889.

\bibitem{FLORIAN-SASSOT-STRATMANN-VOGELSANG} D. de Florian, R. Sassot, M. Stratmann, and W. Vogelsang,
    \textit{Evidence for Polarization of Gluons in the Proton}, Phys. Rev. Lett. \textbf{113}(2014), 012001-- .


\bibitem{DENKMAYR-GEPPERT-SPONAR-LEMMEL-MATZKIN-TOLLAKSEN-HASEGAWA} T. Denkmayr, H. Geppert, S. Sponar, H. Lemmel, A. Matzkin, J. Tollaksen, \& Y. Hasegawa, \textit{Observation of a quantum Cheshire Cat in a matter-wave interferometer experiment}, Nature Communications
\textbf{5(4492)}(2014), 1--7.

\bibitem{DEW1}  B. S. de Witt, \textit{Quantum theory of gravity. I: The canonical theory}, Phys. Rev. \textbf{160(5)}(1967), 1113--1148;
\textit{Quantum theory of gravity. II: The manifestly covariant theory}, Phys. Rev. \textbf{162(5)}(1967), 1195--1239;
\textit{Quantum theory of gravity. III: Applications of the covariant theory}, Phys. Rev. \textbf{162(5)}(1967), 1239--11256.

\bibitem{DEW2}  B. S. de Witt, \textit{Supermanifolds}, Cambridge Univers. Press, Cambridge 1986.

\bibitem{DIRAC} P. A. Dirac, \textit{Cosmological models and the large numbers hypothesis},
Proc. R. Soc. Lond. A \textbf{338(1615)}(1974), 439-446. doi: 10.1098/rspa.1974.0095

\bibitem{DRECHSEL-TIATOR} D. Drechsel and L. Tiator, \textit{Threshold pion photoproduction on nucleons}, J. Phys. G. Nucl. Part. Phys. \textbf{18}(1992), 449-497.
\href{http://wwwkph.kph.uni-mainz.de/MAID//JPhysG18/JPhysG18.pdf}{http://wwwkph.kph.uni-mainz.de/MAID//JPhysG18/JPhysG18.pdf}


\bibitem{DUTTA} D. Dutta, \textit{Exclusive Pion Photoproduction at 12 GeV}. Massachusetts Institute of Technology. \href{http://www.jlab.org/Hall-C/talks/04_09_02/Dutta.pdf}{http://www.jlab.org/Hall-C/talks/04-09-02/Dutta.pdf}

\bibitem{DUTTA-GAO-LEE} D. Dutta, H. Gao and T.-S. H. Lee, \textit{Study of nucleon resonances with double polarization observables of pion photoproduction}, \href{http://arxiv.org/pdf/nucl-th/0111005.pdf}{\tt arXiv:nucl-th/0111005}

\bibitem{DYSON1} F. J. Dyson, \textit{Do gravitons exist ?},(Lecture given at Boston University, 8 Nov. 2005.)

\bibitem{DYSON2} F. J. Dyson, \textit{Is a graviton detectable ?}, Int. J. Modern Phys. A \textbf{28(25)}(2013), 1330041-1--1330035-14.

\bibitem{EDMONDS} A. R. Edmonds, \textit{Angular Momentum in Quantum Mechanics}, Third Edition. Princeton University Press, 1965, Princeton, USA.

\bibitem{EINSTEIN-PODOLSKY-ROSEN} A. Einstein, B. Podolsky and N. Rosen,  \textit{Can quantum-mechanical description of physical reality be considered complete ?}, Phys. Rev. \textbf{47}(1935), 777--780.

\bibitem{DIAKONOV-PETROV} D. Diakonov and V. Petrov, \textit{A heretical view on linear Regge trajectories}, \href{http://arxiv.org/pdf/hep-ph/0312144.pdf}{\tt arXiv:hep-ph/0312144}.

\bibitem{EDWARDS-TAYLOR} S. F. Edwards and J. B. Taylor, \textit{Negative temperature states of two-dimensional plasmas and vortex fluids}, Proc. R. Soc. London A \textbf{336(1606)}(1974), 257--271.

\bibitem{FRADKIN} E. Fradkin, \textit{Quantum physics: Debut of the quarter electron}, Nature. \textbf{452}(2008), 822--823.

\bibitem{GIAMMARCHI} M. G. Giammarchi et al., \textit{Search for electron decay mode $e\to\gamma+\nu$ with prototype of Borexino detector}, Physics Letters B. \textbf{525}(2002), 29--40.

\bibitem{GROMOV} M. Gromov, \textit{Partial Differential Relations}. Springer-Verlag,
Berlin 1986.

\bibitem{HIRSCH} M. Hirsch \textit{Differential Topology},  Springer-Verlag, New York, 1976.

\bibitem{HORAVA-WITTEN} P. Ho\v rava and E. Witten, \textit{Eleven dimensional supergravity on a manifold with boundary}, Nuclear Physics B \textbf{475 (1)}(1996), 94--114. \href{http://arxiv.org/abs/hep-th/9510209}{\tt  arXiv:hep-th/9510209}.

\bibitem{HOYLE} F. Hoyle, \textit{On the Cosmological Problem}, Monthly Notices of the Royal Astronomical Society \textbf{109(3)}(1949), 365-371.

\bibitem{HOYLE-NARLIKAR} F. Hoyle and J. V. Narlikar, \textit{Mach's Principle and the Creation of Matter},
Proc. R. Soc. Lond. A \textbf{273(1352)}(1963), 1-11.

\bibitem{LUKIN-VULETIC} O. Firstenberg, T. Peyronel, Qi-Yu. Liang, A. V. Gorshkov, M. D. Lukin and V. Vuleti\'c, \textit{Attractive photons in a quantum nonlinear medium}, Nature \textbf{502(03 October)}(2013), 71--75.


\bibitem{LIU} Z. Q. Liu et al., \textit{Study of $e^+e^-\to\pi^+\pi^-J/\psi$ and Observation of a Charged Charmoniumlike State at Belle}, Phys. Rev. Lett. \textbf{110(25)}(2013), 252002-9.

\bibitem{LYCH-PRAS} V. Lychagin and A. Pr\'astaro, \textit{Singularities of Cauchy data,
characteristics, cocharacteristics and integral cobordism},
Diff. Geom. Appls. \textbf{4}(1994), 283--300.

\bibitem{MATZKIN-PAN} A. Matzkin and A. K. Pan, \textit{Three-box paradox and 'Cheshire cat grin': the case of spin-1 atoms}, J. Phys. A: Math. Theor. \textbf{46(31)}(2013), 315307 (11 pp).


\bibitem{MILNOR1} J. Milnor, \textit{On manifolds homeomorphic to the $7$-sphere}. Ann. of Math. \textbf{64(2)}(1956), 399--405.

\bibitem{NASH} J. Nash, \textit{Real algebraic manifolds}. Ann. of Math.
\textbf{56(2)}(1952), 405--421.

\bibitem{NEUMANN} J. von Neumann, In Mathematische Grundlagen der Quantenmechanik, Springer, Berlin, (1932/1955). (English translation by R. T. Beyer, Princeton University Press.)

\bibitem{NOCERA-BALL-FORTE-RIDOLFI-ROJO} E. R. Nocera, R. D. Ball, S. Forte, G. Ridolfi, and J. Rojo, \textit{A first unbiased global determination of polarized PDFs and their uncertainties}. \href{http://arxiv.org/abs/arXiv:1406.5539}{\tt arXiv:1406.5539}.

\bibitem{NORMANN-BACHALL-GOLDHABER} E. B. Norman, J. N. Bachall and M. Goldhaber, \textit{Improved limit on charge conservation derived from ${}^{17}Ga$ solar neutrino experiments}. Phys. Rev. D \textbf{53}(1996), 4086--4088.

\bibitem{PIKE-MACKENROTH-HILL-ROSE} O. J. Pike, F. Mackenroth, E. G. Hill and S. J. Rose, \textit{A photon–photon collider in a vacuum hohlraum}, Nature Photonics May  (2014). \href{http://www.nature.com/nphoton/journal/vaop/ncurrent/full/nphoton.2014.95.html}{DOI: doi:10.1038/nphoton.2014.95}

\bibitem{PONTRJAGIN} L. S. Pontrjagin, \textit{Smooth manifolds and their applications homotopy theory}. Amer. Math. Soc. Transl. \textbf{11}(1959), 1--114.

\bibitem{PRADHAN} T. Pradhan, \textit{Electron decay}. \href{http://arxiv.org/pdf/hep-th/0312325.pdf}{\tt arXiv:hep-th/0312325v1}.

\bibitem{PRA5} A. Pr\'astaro, \textit{Quantum geometry of PDE's}, Rep. Math.
Phys. \textbf{30(3)}(1991), 273--354.

\bibitem{PRA6} A. Pr\'astaro, \textit{Geometry of super PDE's.}, Geometry in Partial Differential
Equations, A. Pr\'astaro and Th. M. Rassias (eds.), World Scientific
Publishing Co., Singapore (1994), 259--315.

\bibitem{PRA7} A. Pr\'astaro, \textit{Geometry of quantized super PDE's}, Amer. Math. Soc. Transl.
\textbf{167(2)}(1995), 165--192.

\bibitem{PRA8} A. Pr\'astaro, \textit{Quantum geometry of super PDE's}, Rep. Math. Phys \textbf{37(1)}(1996), 23--140.

\bibitem{PRAS13} A. Pr\'astaro, \textit{Geometry of PDE's. I: Integral bordism groups in PDE's}. J. Math. Anal. Appl. \textbf{319}(2006), 547--566; \textit{Geometry of PDE's. II: Variational PDE's and integral bordism groups}.
J. Math. Anal. Appl. \textbf{321}(2006), 930--948.

\bibitem{PRAS1} A. Pr\'astaro, \textit{Geometry of PDEs and Mechanics}, World Scientific Publishing, River Edge, NJ, 1996, 760 pp. ISBN 9810225202.

\bibitem{PRAS2} A. Pr\'astaro, \textit{(Co)bordism in PDEs and quantum PDEs}, Rep. Math. Phys. \textbf{38(3)}(1996), 443-455. DOI: 10.1016/S0034-4877(97)84894-X.

\bibitem{PRAS3} A. Pr\'astaro, \textit{(Co)bordism groups in quantum PDEs}, Acta Appl. Math. \textbf{64(2)}(2000), 111-217.

\bibitem{PRAS4} A. Pr\'astaro, \textit{Quantum manifolds and integral (co)bordism groups in quantum partial differential equations}, Nonlinear Anal. Theory Methods Appl. \textbf{47/4}(2001), 2609-2620.

\bibitem{PRAS5} A. Pr\'astaro, \textit{Quantum super Yang-Mills equations: Global existence and mass-gap}, Dynamic Syst. Appl. \textbf{4}(2004), 227-232. (Eds. G. S. Ladde, N. G. Madhin and M. Sambandham), Dynamic Publishers, Inc., Atlanta, USA. ISBN:1-890888-00-1.

\bibitem{PRAS11} {A. Pr\'astaro}, \textit{Quantized Partial Differential Equations}, World Scientific Publ., Singapore, 2004.

\bibitem{PRAS12} A. Pr\'astaro, \textit{Conservation laws in quantum super PDE's}, \href{Melbourne-Florida-2005.pdf}{Proceedings of the Conference on Differential \& Difference Equations and Applications (eds. R. P. Agarwal \& K.
Perera)}, Hindawi Publishing Corporation, New York (2006), 943--952.

\bibitem{PRAS14} A. Pr\'astaro, \textit{(Co)bordism groups in quantum super PDE's.I: Quantum supermanifolds},
Nonlinear Anal. Real World Appl. \textbf{8(2)}(2007), 505--538; \textit{(Co)bordism groups in quantum super PDE's.II: Quantum super PDE's}, Nonlinear Anal. Real World Appl. \textbf{8(2)}(2007), 480--504; \textit{(Co)bordism groups in quantum super PDE's.III: Quantum super Yang-Mills equations}, Nonlinear Anal. Real World Appl. \textbf{8(2)}(2007), 447--479.

\bibitem{PRAS15} A. Pr\'astaro, \textit{On quantum black-hole solutions of quantum super Yang-Mills equations}, Dynamic Syst. Appl. \textbf{5}(2008), 407-414. (Eds. G. S. Ladde, N. G. Madhin, C. Peng and M. Sambandham), Dynamic Publishers, Inc., Atlanta, USA. ISBN: 1-890888-01-6.

\bibitem{PRAS19} A. Pr\'astaro, \textit{Surgery and bordism groups in quantum partial differential equations.I: The quantum Poincar\'e conjecture}. Nonlinear Anal. Theory Methods Appl. \textbf{71(12)}(2009), 502--525; \textit{Surgery and bordism groups in quantum partial differential equations.II: Variational quantum PDE's}. Nonlinear Anal. Theory Methods Appl. \textbf{71(12)}(2009), 526--549.


\bibitem{PRAS21-1} A. Pr\'astaro,   \href{http://link.springer.com/chapter/10.1007/978-1-4939-1106-6_18}{\textit{Extended crystal PDE's}}.  \href{http://www.springer.com/mathematics/analysis/book/978-1-4939-1105-9}{Mathematics Without Boundaries: Surveys in Pure Mathematics. P. M. Pardalos and Th. M. Rassias (Eds.) Springer-Heidelberg New York Dordrecht London, (2014), 415--481.} ISBN 978-1-4939-1106-6 (Online) 978-1-4939-1105-9 (Print). \href{http://dx.doi.org/10.1007/978-1-4939-1106-6}{DOI: 10.1007/978-1-4939-1106-6}.
    \href{http://arxiv.org/abs/0811.3693}{\tt arXiv:0811.3693[math.AT]}.


\bibitem{PRAS21} A. Pr\'astaro, \textit{Quantum extended crystal PDE's}, \href{http://nonlinearstudies.com/index.php/nonlinear}{Nonlinear Studies \textbf{18(3)}(2011), 447--485.} \href{http://arxiv.org/abs/1105.0166}{\tt arXiv:1105.0166[math.AT]}.

\bibitem{PRAS22} A. Pr\'astaro, \textit{Quantum extended crystal super PDE's}, Nonlinear Analysis. Real World Appl. \textbf{13(6)}(2012), 2491--2529. \href{http://dx.doi.org/10.1016/j.nonrwa.2012.02.014}{DOI: 10.1016/j.nonrwa.2012.02.014}. \href{http://arxiv.org/abs/0906.1363}{\tt arXiv:0906.1363[math.AT].}

\bibitem{PRAS23} A. Pr\'astaro, \textit{Exotic heat PDE's}, Commun. Math. Anal. \textbf{10(1)}(2011), 64--81. \href{http://arxiv.org/abs/1006.4483}{\tt arXiv:1006.4483[math.GT]}; \textit{Exotic heat PDE's.II},  Essays in Mathematics and its Applications. (Dedicated to Stephen Smale for his $80^{th}$ birthday.)  (Eds. P. M. Pardalos and Th.M. Rassias). Springer-Heidelberg, New York Dordrecht London, (2012), 369--419. \href{http://dx.doi.org/10.1007/978-3-642-28821-0}{DOI: 10.1007/978-3-642-28821-0.} \href{http://arxiv.org/abs/1009.1176}{\tt arXiv:1009.1176[math.AT].}

\bibitem{PRAS25} A. Pr\'astaro, \href{http://link.springer.com/chapter/10.1007/978-1-4614-3498-6_36}{\textit{Exotic $n$-d'Alembert PDE's and stability}}. \href{http://www.springer.com/mathematics/book/978-1-4614-3497-9?cm_mmc=NBA-_-Jun-12_EAST_10792128-_-product-_-978-1-4614-3497-9}{Nonlinear Analysis: Stability, Approximation and Inequalities. (Dedicated to Themistocles M. Rassias for his 60th birthday.) G. Georgiev (USA), P. Pardalos (USA) and H. M. Srivastava (Canada) (eds.), Springer Optimization and its Applications Volume 68(2012), 571-586}. ISBN 978-1-4614-3498-6.
\href{http://dx.doi.org/10.1007/978-1-4614-3498-6}{DOI: 10.1007/978-1-4614-3498-6}.
  \href{http://arxiv.org/abs/1011.0081}{\tt arXiv:1011.0081[math.AT]}.

\bibitem{PRAS26} A. Pr\'astaro, \href{http://link.springer.com/chapter/10.1007/978-1-4939-1106-6_18}{\textit{Exotic PDE's}}. \href{http://www.springer.com/mathematics/analysis/book/978-1-4939-1123-3}{Mathematics Without Boundaries: Surveys in Interdisciplinary Research. P. M. Pardalos and Th. M. Rassias (Eds.) Springer-Heidelberg New York Dordrecht London, (2014), 471--531.} ISBN 978-1-4939-1123-3 (print) 978-1-4939-1124-0 (eBook).
    \href{http://dx.doi.org/10.1007/978-1-4939-1124-0}{DOI: 10.1007/978-1-4939-1124-0}.
\href{http://arxiv.org/abs/1101.0283}{\tt arXiv:1101.0283[math.AT]}.


\bibitem{PRAS27} A. Pr\'astaro, \textit{Quantum exotic PDE's}. Nonlinear Anal. Real World Appl. \textbf{14(2)}(2013), 893--928. \href{http://dx.doi.org/10.1016/j.nonrwa.2012.04.001}{DOI: 10.1016/j.nonrwa.2012.04.001.}. \href{http://arxiv.org/abs/1106.0862}{\tt arXiv:1106.0862[math.AT]}.

\bibitem{PRAS28} A. Pr\'astaro, \textit{Strong reactions in quantum super PDE's. I: Quantum hypercomplex exotic super PDE's}. \href{http://arxiv.org/abs/1205.2894}{\tt arXiv:1205.2894}, 1--39.

\bibitem{PRAS29} A. Pr\'astaro, \textit{Strong reactions in quantum super PDE's. II: Nonlinear quantum propagators}. \href{http://arxiv.org/abs/1205.2894}{\tt arXiv:1205.2894}, 40--80.


 \bibitem{PRAS30} A. Pr\'astaro, \textit{The Landau's problems. I-II}. \textit{The Landau's problems. I: The Goldbach's conjecture proved}; \textit{The Landau's problems. II: Landau's problems solved}.
\href{http://arxiv.org/abs/1205.2894}{\tt arXiv:1205.2894v11}.

 \bibitem{PRASTARO-Zbl1162-58002} A. Pr\'astaro, Review of the paper by S.J. Brain and S. Majid, \textit{Quantization of twistor theory by cocycle twist}, Commun. Math. Phys. \textbf{284(3)}(2008),713--774, in \href{http://www.zentralblatt-math.org/MIRROR/zmath/en/advanced/?q=an:05530504&type=pdf&format=complete}{Zentralblatt MATH \textbf{Zbl 1162.58002}}.


 \bibitem{PRASTARO-MR2870866} A. Pr\'astaro, Review of the paper by N. Seiberg and W. Taylor, \textit{Charge lattices and consistency of 6D supergravity}, J. High Energy Phys. \textbf{6(1)}(2011),1029-8479, in \href{http://www.ams.org/mathscinet/pdf/2870866.pdf}{MathScinet \textbf{MR2870866}}.


\bibitem{PURCELL-POUND} E. M. Purcell and R. V. Pound, \textit{A nuclear spin system at negative temperature}. Phys. Rev. \textbf{81(2)}(1951), 279--280.

\bibitem{ROTHMAN-BOUGHN} T. Rithman and S. Boughn, \textit{Can gravitons be detected ?}, \href{http://arxiv.org/pdf/gr-qc/0601043.pdf}{\tt arXiv:gr-qc/0601043}

\bibitem{SAKHAROV} A. D. Sakharov, \textit{Violation of CP invariance, C asymmetry, and
baryon asymmetry of the universe}, Pis'ma Zh. Eksp. Teor. Fiz. \textbf{5}(1967), 32-35, reprinted in:
Kolb, E. W., Turner, M. S. (Eds.), The Early Universe, pp. 371–373;
Lindley, D. et al. (Eds.), Cosmology and Particle Physics, pp. 106–109; Sov. Phys. Usp. \textbf{34(5)}(1991), 392–393;
Usp. Fiz. Nauk \textbf{161(5)}(1991), 61–64.

\bibitem{SCHNEIDER} U. Schneider, \textit{Negative absolute temperature for motional degrees of freedom}, Science \textbf{339(6115)}(2013), 52-55.
       \href{http://www.sciencemag.org/content/339/6115/52}{DOI: 10.1126/science.1227831}.

 \bibitem{SCHRODINGER} E. Schr\"odinger, \textit{Die gegenw\"artige situation in der Quantunmechanik}, Die Naturwissenschaften, \textbf{23}(1935), 807--812, 823-828, 844-849.

\bibitem{SMALE} S. Smale, \textit{Generalized Poincar\'e conjecture in dimension greater than four}. Ann. of Math. \textbf{74(2)}(1961), 391--406.

 \bibitem{STONG} R. E. Stong, \textit{Notes on Bordism Theories}. Amer. Math. Studies. Princeton Univ. Press, Princeton, 1968.

\bibitem{SWITZER} A. S. Switzer,  \textit{Algebraic Topology-Homotopy and Homology}, Springer-Verlag, Berlin,
1976.

\bibitem{THOM} R. Thom, \textit{Quelques propri\'et\'e globales des vari\'et\'es
diff\'erentieles}, Comm. Math. Helv. \textbf{28}(1954), 17--86.

\bibitem{NIEUWENHUIZEN} P. van Nieuwenhuizen, \textit{Supergravity}, Phys. Reports. \textbf{68(4)}(1981), 189--398.

\bibitem{WALL2} C. T. C. Wall,  \textit{Surgery on Compact Manifolds}, London Math. Soc. Monographs \textbf{1},
Academic Press, New York, 1970; 2nd edition (ed. A. A. Ranicki),
Amer. Math. Soc. Surveys and Monographs \textbf{69}, Amer. Math.
Soc., 1999.

\bibitem{WITTEN} E. Witten, \textit{String theory dynamics in various dimensions}, Nuclear Physics B \textbf{443 (1)}(1995), 85--126. \href{http://arxiv.org/abs/hep-th/9503124}{\tt arXiv:hep-th/9503124}.


\bibitem{ZEILINGER} A. Zeilinger et al. \textit{Quantum entanglement of high angular momenta}, Science 2 \textbf{338(6107)}(2012), 640--643. \href{http://www.sciencemag.org/content/338/6107/640}{DOI. 10.1126/science.1227193}.

\end{thebibliography}
\end{document}